\documentclass[preprint,12pt]{elsarticle}

\usepackage{amssymb}
\usepackage{indentfirst}
\usepackage{graphicx}
\usepackage{epstopdf}
\usepackage{epsfig}
\usepackage{overpic}
\usepackage{array}
\usepackage{color}
\usepackage{multirow}
\usepackage{float}
\usepackage{subfigure}
\usepackage{amssymb}
\usepackage{amsmath,amssymb, amsthm}
\usepackage[mathscr]{euscript}
\usepackage{latexsym,bm}
\usepackage{booktabs}
\usepackage{bbm}
\usepackage{diagbox}
 \usepackage[all]{xy}
\usepackage{diagbox}
\usepackage{algorithm}  
\usepackage{algorithmicx} 
\usepackage{algpseudocode} 
\usepackage{mathtools} 

\newtheorem{theorem}{Theorem}
\newtheorem{lemma}{Lemma}
\newtheorem{proposition}{Proposition}

\def\me{\mathbbm{e}}

\def\bx{\bm{x}}

\def\bv{\bm{v}}
\def\bu{\bm{u}}

\def\bvarepsilon{\bm{\varepsilon}}
\def\D{\textup{d}}

\makeatletter
\@addtoreset{equation}{section}
\makeatother

\begin{document}

\newsavebox{\tablebox}

\begin{frontmatter}

%% Title, authors and addresses

%% use the tnoteref command within \title for footnotes;
%% use the tnotetext command for theassociated footnote;
%% use the fnref command within \author or \address for footnotes;
%% use the fntext command for theassociated footnote;
%% use the corref command within \author for corresponding author footnotes;
%% use the cortext command for theassociated footnote;
%% use the ead command for the email address,
%% and the form \ead[url] for the home page:
%% \title{Title\tnoteref{label1}}
%% \tnotetext[label1]{}
%% \author{Name\corref{cor1}\fnref{label2}}
%% \ead{email address}
%% \ead[url]{home page}
%% \fntext[label2]{}
%% \cortext[cor1]{}
%% \affiliation{organization={},
%%             addressline={},
%%             city={},
%%             postcode={},
%%             state={},
%%             country={}}
%% \fntext[label3]{}

\title{A short-memory operator splitting scheme for constant-Q viscoelastic wave equation}

%% use optional labels to link authors explicitly to addresses:
%% \author[label1,label2]{}
%% \affiliation[label1]{organization={},
%%             addressline={},
%%             city={},
%%             postcode={},
%%             state={},
%%             country={}}
%%
%% \affiliation[label2]{organization={},
%%             addressline={},
%%             city={},
%%             postcode={},
%%             state={},
%%             country={}}

\author[label1]{Yunfeng Xiong}
 \ead{xiongyf@math.pku.edu.cn}

\affiliation[label1]{organization={LMAM and School of Mathematical Sciences, Peking University},%Department and Organization
 %           addressline={}, 
   %         city={},
            postcode={100871}, 
            state={Beijing},
            country={China}}

 \author{Xu Guo\corref{cor1}\fnref{label2}}
 \ead{guoxu@sdu.edu.cn}
 \cortext[cor1]{To whom correspondence}

\affiliation[label2]{organization={Geotechnical and Structural Research Center, Shandong University},%Department and Organization
%            addressline={}, 
            city={Jinan},
            postcode={250061}, 
            state={Shandong},
            country={China}}

\begin{abstract}
%% Text of abstract
We propose a short-memory operator splitting scheme for solving the constant-Q wave equation, where the fractional stress-strain relation contains multiple Caputo fractional derivatives with order much smaller than 1. The key is to exploit its extension problem by converting the flat singular kernels into strongly localized ones, so that the major contribution of weakly singular integrals over a semi-infinite interval can be captured by a few Laguerre functions with proper asymptotic behavior. Despite its success in reducing both memory requirement and arithmetic complexity, we show that numerical accuracy under prescribed memory variables may deteriorate in time due to the dynamical increments of projection errors. Fortunately, it can be considerably alleviated by introducing a suitable scaling factor $\beta > 1$ and pushing the collocation points closer to origin. An operator splitting scheme is introduced to solve the resulting set of equations, where the auxiliary dynamics can be solved exactly, so that it gets rid of the numerical stiffness and discretization errors. Numerical experiments on both 1-D diffusive wave equation and 2-D constant-Q $P$- and $S$-wave equations are presented to validate the accuracy and efficiency of the proposed scheme.
\end{abstract}

%%Graphical abstract
%\begin{graphicalabstract}
%%\includegraphics{grabs}
%\end{graphicalabstract}

%%Research highlights
%\begin{highlights}
%
%\item We propose a short-memory operator splitting scheme for solving the constant-Q viscoelastic wave equation via its extension problem and efficient scaling technique.
%
%\item Based on the error estimates of the generalized Laguerre-Gauss-type interpolation, we give a solid theoretical analysis of the Yuan-Agrawal method or the diffusive approximation. 
%
%%We prove that the projection errors may be amplified in time evolution due to the impulse of source term, and introducing a suitable scaling factor $\beta > 1$ can considerably alleviate such numerical phenomenon.
%
%\item We design an efficient algorithm for calculating the Mainardi function, which is used to describe the fundamental solutions for the diffusive wave equation. 
%
%%This in turn provides a reasonable reference solutions for our numerical experiments.
%
%
%\end{highlights}

\begin{keyword}
Viscoelastic wave equation \sep Caputo fractional derivative \sep Laguerre spectral method \sep Scaling technique \sep Short-memory principle

%% keywords here, in the form: keyword \sep keyword

%% PACS codes here, in the form: \PACS code \sep code

%% MSC codes here, in the form: \MSC code \sep code
\MSC[2010] \sep 
65M22 \sep 	%Numerical solution of discretized equations for initial value and initial-boundary value problems involving PDEs
74S25  \sep	%Spectral and related methods applied to problems in solid mechanics
41A05  \sep	%Interpolation in approximation theory
35R11  \sep	%Fractional partial differential equations
33F05 	%Numerical approximation and evaluation of special functions 
%% or \MSC[2008] code \sep code (2000 is the default)

\end{keyword}

\end{frontmatter}

%% \linenumbers

%% main text
%%%%%%%%%%%%%%%%%%%%
\section{Introduction}
\label{sec.intro}

Seismic wave propagation has anelastic characteristics in real earth materials due to the loss of energy by the geometrical effect of the enlargement of the wavefront and by the intrinsic absorption of the earth \cite{Kjartansson1979,Virieux1986,CarcioneCavalliniMainardiHanyga2002}. Thus both the effects of attenuation and velocity dispersion should be taken into account for obtaining an accurate wavefield modeling. Among various models, Kjartansson's constant-Q model containing the fractional strain-stress relation provides a concise description of the wave attenuation by only two parameters: $Q$ for the portion of energy lost and $c_0$ for the phase velocity \cite{Kjartansson1979,bk:Mainardi2010}, and the quality factor $Q$ is usually observed as a frequency-independent (constant) parameter \cite{SunGaoLiu2019,XingZhu2021}. In most circumstances the loss caused by scattering and absorption is relatively small, so that the attenuation of the seismic energy is treated as a small perturbation on the propagation \cite{Kjartansson1979,SunGaoLiu2019}. Following the basic idea of Kjartansson's model, Carcione developed the theory of constant-Q $P$- and $S$- wave modeling and derived a set of parabolic equations involving multiple temporal Caputo fractional derivatives, which exhibit different attenuation levels in different propagation directions \cite{Carcione2009}. In the past few decades, 2-D and 3-D constant-Q wave modelings, as well as their variants,  have been drawing a growing attention in various seismic applications \cite{ZhuCarcione2014,ZhuHarrisBiondi2014,SunFomelZhuHu2016,ShuklaCarcionePestanaJaiswalOzdenvar2019}.

Despite its mathematical elegance, solving  the constant-Q wave equation remains a huge numerical challenge due to the extremely high memory requirement and high arithmetic complexity induced by the very small fractional power \cite{Diethelm2008}, which is about $10^{-3}$ to $10^{-1}$ in real viscoelastic media \cite{CarcioneCavalliniMainardiHanyga2002}.  In recent years, many sophisticated algorithms have been proposed to tackle the Caputo fractional derivatives in the diffusive wave models, such as the finite difference schemes based on the weighted Gr{\"u}nwald-Letnikov formula \cite{Carcione2009,WangVong2014} and those based on the interpolation polynomials \cite{SunZhaoSun2019,ShenLiSun2020}, the 
spectral collocation method \cite{YangChenHuangWei2016} and the meshless method \cite{HosseiniShivanianChen2016}, but their efficiency might still be hampered by the storage of long history. It has been reported by the author of \cite{Zhu2017} that the storage of $500$ memory length and the grid of $120^3$ in a 3-D orthorhombic example requires approximately $90$GB computer memory,  thereby dramatically restricting the numerical resolution and hampering its usage in real seismic applications such as reverse time migration.
 For this reason, the authors of \cite{ChenHolm2004,ZhuCarcione2014,ZhuHarris2014} suggested to replace the temporal fractional derivative by a spatial fractional Laplacian operator with the same dispersion relation and developed efficient numerical algorithms for the resulting memory-free nonlocal wave equations \cite{Carcione2010,ZhuCarcione2014,YaoZhuHussainKouri2017,ShuklaCarcionePestanaJaiswalOzdenvar2019}. However, it still requires to establish their rigorous mathematical connection, especially when the artificial boundary conditions are involved \cite{DuHanZhangZheng2018}.

Intuitively speaking, the drawback of long-memory requirement in the existing algorithms  is rooted in the flatness of the singular kernel in the Caputo fractional derivative, so that the remedy is to seek a localization of such kernel. Our short-memory operator splitting (SMOS) scheme is based on the extension problem of the Caputo derivative by introducing an extra memory variable $y \in (0, +\infty)$  \cite{MartinezSanzMarco1988,GaleMianaStinga2013}, and  the initial effective condition must be derived  from the fractional strain-stress relation to preserve the proper asymptotic behavior. In this way, the original constant-Q wave equation can be reformulated into an equivalent set of dynamical system without fractional derivatives, but involving an integral over a semi-infinite domain with a strongly localized Abel kernel. Such nonlocal component is usually resolved by the Laguerre-Gauss quadrature in the spirit of the Yuan-Agrawal method \cite{YuanAgrawal2002,LuHanyga2005} or the diffusive approximation \cite{BlancChiavassaLombard2014,Diethelm2008}. Compared with the sum-of-exponentials technique \cite{ShenLiSun2020}, the Yuan-Agrawal method allows a simpler truncation as the major contribution of the pre-asymptotic range of the weakly singular integral can be accurately captured by  only a few Laguerre-Gauss quadrature nodes (memory variables) \cite{Diethelm2008,bk:ShenTangWang2011}.

Regarding that the Yuan-Agrawal method still lacks a rigorous theoretical justification, we provide a thorough analysis of the dynamical errors  within the framework of generalized Laguerre interpolation \cite{Tang1993,Shen2000,MaSunTang2006,XiaShaoChou2020,GuoWangWang2006} and point out its potential weakness, that is, the dynamical increments of projection errors due to the impulse (non-decay) of the source term. In other words, the numerical accuracy under the prescribed number of collocation points may deteriorate sharply in time evolution. Fortunately, we also show that it can be significantly alleviated by introducing a scaling factor $\beta > 1$ in the Laguerre functions and pushing the collocations points closer to origin. 
Another challenge is the severe numerical stiffness in solving the auxiliary dynamics, which usually requires the usage of $A$-stable implicit integrators and introduces additional numerical errors \cite{Diethelm2008}. SMOS utilizes a key observation that although it breaks the Hamiltonian structure of the elastic counterpart, the constant-Q wave equation can still be evolved by an appropriate operator splitting. In particular, the auxiliary equations can be solved exactly, thereby getting rid of the numerical stiffness and significantly boosting the numerical stability.  Numerical experiments also validate that the Strang splitting can achieve a second-order convergence in time, and higher-order scheme can be achieved in principle by Yoshida's method \cite{Yoshida1990}. Combining these ingredients together allows us to maintain the accuracy, to significantly shorten the effective memory length and to reduce the arithmetic complexity simultaneously.

%Thus the starting point of SMOS scheme is the extension problem of the Caputo derivative \cite{MartinezSanzMarco1988,GaleMianaStinga2013} by introducing an extra memory variable $y \in (0, +\infty)$, in the spirit of the Yuan-Agrawal method \cite{YuanAgrawal2002,LuHanyga2005,bk:LiZeng2015} or the diffusive approximation \cite{BlancChiavassaLombard2014,Diethelm2008}. In this way, the original constant-Q wave equation can be reformulated into an equivalent set of dynamical system without fractional derivatives, but involving a nonlocal component with a strongly localized Abel kernel $y^{4\gamma - 1}/\Gamma(4\gamma)$. {\cb This endows our short-memory principle with two-fold meanings.} First, it is suggested to use to capture theand avoid the redundant nodes in the post-asymptotic range, say, {\sl more is not always better}. Second, the potential weakness of the Yuan-Agrawal method, as pointed out in  our theoretical analysis, is   Fortunately,  it can be dramatically alleviated by introducing a scaling factor $\beta > 1$ in the Laguerre polynomials \cite{Tang1993,Shen2000,MaSunTang2006,,XiaShaoChou2020}, that is, {\sl scaling is better}. 

The rest of this paper is organized as follows. In Section \ref{sec.background} we briefly review the constant-Q  viscoelastic wave equation and derive its extension problem, as well as the proper initial condition. In Section \ref{sec.scheme}, we will first analyze the dynamical increments of Laguerre-Gauss projection errors and present how the scaling factor $\beta > 1$ enhances the accuracy, then give the operator splitting scheme. Numerical experiments on both 1-D diffusive wave equation and 2-D constant-Q wave equation are presented in Section \ref{sec.numerical} to validate the convergence, accuracy and efficiency of the proposed scheme. Specifically, the relation between the number of memory variables and accuracy will be carefully investigated. The conclusion and several discussions are drawn in Section \ref{sec.conclusion}.

%%%%%%%%%%%%%%%%%%%%
\section{The extension problem of constant-Q wave equation}
\label{sec.background}

The dynamics of constant-Q wave propagation are governed by three sets of equations with $\bx \in \mathbb{R}^d$, $d\le 3$ \cite{Carcione2009}. The first set is the conservation of linear momentum:
\begin{equation}\label{conservation_momentum}
\rho(\bx) \frac{\partial^2}{\partial t^2} u_i(\bx, t) = \frac{\partial}{\partial x_j} \sigma_{ij}(\bx, t) + f_i(\bx, t), \quad i, j = 1, \dots, 3,
\end{equation}
where $\sigma_{ij}$ are the components of the stress tensor, $u_i$ are the components of the displacement vector, $\rho$ is the mass density and $f_i$ are components of the body forces per unit (source term). The summation over repeated indices $j$ is assumed in Eq.~\eqref{conservation_momentum}. The second set is the definition of strain tensor $\varepsilon_{ij}$, which can be obtained in terms of the displacement components as
\begin{equation}\label{defintion_strain}
\varepsilon_{ij}(\bx, t) = \frac{1}{2} \left(  \frac{\partial}{\partial x_i} u_j(\bx, t) +  \frac{\partial}{\partial x_j} u_i(\bx, t) \right), \quad i, j = 1, \dots, 3.
\end{equation}

The constitutive equation, or the stress-strain relation, involves a relaxation by power creep functions. Equivalently, it contains temporal Caputo fractional derivatives \cite{bk:Mainardi2010,Carcione2009},
\begin{equation}\label{stress_strain_relation}
\begin{split}
\sigma_{ij} = &\frac{M_P(\bx)}{\Gamma(1-2\gamma_P)} \left(\frac{t}{t_0}\right)^{-2\gamma_P} H(t) \ast \frac{\partial}{\partial t} \varepsilon_{ij} - \frac{2 M_S(\bx)}{\Gamma(1-2\gamma_S)} \left(\frac{t}{t_0}\right)^{-2\gamma_S} H(t) \ast \frac{\partial}{\partial t}\varepsilon_{ij}\\
= &\mathcal{E}(\bx) {_C}D_t^{2\gamma_P} \varepsilon_{kk} \delta_{ij} + 2 \mu(\bx){_C}D_t^{2\gamma_S} \left[ \varepsilon_{ij} - \varepsilon_{kk} \delta_{ij} \right],
\end{split}
\end{equation}
%\begin{equation}\label{stress_strain}
%\sigma_{ij}(x, t) = \frac{M_0(\bx)}{\Gamma(1-2\gamma)} \left(\frac{t}{t_0}\right)^{-2\gamma} H(t) \ast \frac{\partial \varepsilon_{ij}}{\partial t}(\bx, t) = \left(\frac{M_0(\bx)}{t_0^{-2\gamma}}\right) {_C}D_t^{2\gamma}\varepsilon_{ij}(\bx, t),
%\end{equation}
where the definition of the extended Caputo fractional derivative reads that
\begin{equation}\label{def.extended_Caputo_time_fractional_operator}
{_C}D_t^{\alpha} \bvarepsilon \coloneqq 
\left\{
\begin{split}
&\frac{1}{\Gamma(1 - \alpha)} \int_0^t (t - \tau)^{-\alpha} \left[\frac{\partial}{\partial \tau} \bvarepsilon(\bx, \tau)\right] \D \tau, &\quad 0 < \alpha < 1,\\
&\frac{1}{\Gamma(2 - \alpha)} \int_0^t (t - \tau)^{1-\alpha} \left[\frac{\partial^2}{\partial \tau^2} \bvarepsilon(\bx, \tau)\right] \D \tau, &\quad 1 < \alpha < 2,
\end{split}
\right.
\end{equation}
and $M_P(\bx) = \rho(\bx) c_P^2 \cos^{2} \left(\frac{\pi \gamma_P}{2}\right), M_S(\bx) = \rho(\bx) c_S^2 \cos^{2} \left(\frac{\pi \gamma_S}{2}\right)$ are bulk moduli, $c_P$ and $c_S$ are $P$-wave and $S$-wave velocities, respectively, $\Gamma$ is the Euler's Gamma function, $t_0$ is the same reference time, $\gamma_P, \gamma_S$ are dimensionless parameters, and $H$ is the Heaviside step function. The summation over repeated indices $k$ is assumed in Eq.~\eqref{stress_strain_relation}. The powers $\gamma_P, \gamma_S$ of the fractional derivatives are characterized by constant quality factors $Q_P$ and $Q_S$ for $P$- and $S$- waves, respectively.
\begin{equation}\label{def.constant}
\gamma_P = \pi^{-1} \tan^{-1} (Q_P^{-1}), \quad \gamma_S = \pi^{-1} \tan^{-1} (Q_S^{-1}).
\end{equation}

The main difficulty in solving the constant-Q wave equation lies in the singular kernel $(t - \tau)^{-2\gamma}$ in Eq.~\eqref{def.extended_Caputo_time_fractional_operator} since it is global and rather flat when $\gamma \ll 1$, so that direct discretization requires to store a long history to maintain the accuracy. A localization of the singular kernel can be achieved by introducing an auxiliary function $\Phi[\bvarepsilon](\bx, y, t)$ with an extra memory variable $y$, so that it extends the constant-Q wave equation into another nonlocal problem with an extra dimension. Such essential idea was exploited in the Yuan-Agrawal method \cite{YuanAgrawal2002,LuHanyga2005} and the diffusive approximation \cite{BlancChiavassaLombard2014}.

In the following part, we will derive the extension problem of the Caputo fractional derivative. As typical examples, the extension of unidimensional and two-dimensional constant-Q wave equations are provided. Then we try to give an initial condition for the auxiliary functions $\Phi[\bvarepsilon](\bx, y, t)$, which  was usually assumed to be zero in the previous literatures \cite{YuanAgrawal2002,LuHanyga2005,Diethelm2008}. But actually it must be derived from the stress-strain relation as the violation may lead to incorrect asymptotic behavior \cite{SchmidtGaul2006}.

\subsection{Extension problem}

The extension of the fractional derivative \eqref{def.extended_Caputo_time_fractional_operator} is known as the fractional power in functional analysis \cite{MartinezSanzMarco1988,YuanAgrawal2002}, which is also the basis of the Yuan-Agrawal method,
\begin{equation}
\begin{split}
{_C}D_t^\alpha \bvarepsilon(\bx, t) & = \frac{1}{\Gamma(1 - \alpha)} \int_0^t (t - \tau)^{-\alpha} \left[\frac{\partial}{\partial \tau} \bvarepsilon(\bx, \tau)\right] \D \tau \\
& = \frac{\sin(\pi \alpha)}{\pi} \int_{0}^{+\infty} \int_0^t \frac{\me^{-z} z^{\alpha - 1}}{(t - \tau)^\alpha} \left[\frac{\partial}{\partial \tau} \bvarepsilon(\bx, \tau)\right] \D \tau \D z \\
& = \frac{2 \sin(\pi \alpha)}{\pi} \int_{0}^{+\infty} \int_0^t y^{2\alpha - 1} \me^{-(t - \tau) y^2} \left[\frac{\partial}{\partial \tau} \bvarepsilon(\bx, \tau)\right] \D \tau \D y, 
\end{split}
\end{equation}
where the second equality utilizes the reflection formula for the Gamma function,
\begin{equation}\label{identity_gamma_function}
\Gamma(\alpha) = \int_0^{+\infty} \me^{-z} z^{\alpha - 1} \D z, \quad \Gamma(\alpha) \Gamma(1-\alpha) = \frac{\pi }{\sin(\pi \alpha)}.
\end{equation}
The third equality can be verified by the variable substitution $z = (t - \tau) y^2$.

Now we define the auxiliary function $\Phi[\bvarepsilon](\bx, y, t): \mathbb{R}^d \times \mathbb{R} \times [0, \infty) \to \mathbb{R}$ with an extra variable,
\begin{equation}\label{def.auxiliary}
\Phi[\bvarepsilon](\bx, y, t) = \int_0^t \me^{-(t - \tau) y^2} \left[\frac{\partial}{\partial \tau} \bvarepsilon(\bx, \tau)\right] \D \tau.
\end{equation}
It deserves to mention that our definition of $\Phi[\bvarepsilon](\bx, y, t)$ is slightly different from that in \cite{YuanAgrawal2002}, where the singular kernel is incorporated into the auxiliary function. This gets rid of the singularity at $y \to 0^+$ in the definition, as well as the dependence on the fractional order $\gamma$.

According to Eq.~\eqref{def.auxiliary}, it satisfies the auxiliary relaxed dynamics,
\begin{equation}\label{dynamics_response}
\frac{\partial }{\partial t} \Phi[\bvarepsilon](\bx, y, t) = \frac{\partial}{\partial t} \bvarepsilon(\bx, t) - y^2 \Phi[\bvarepsilon](\bx, y, t),
\end{equation}
and the fractional Caputo derivative is represented as a weakly singular integral
\begin{equation}\label{singular_integral}
{_C}D_t^\alpha \bvarepsilon(\bx, t) = \frac{2\sin(\pi \alpha)}{\pi} \int_0^{+\infty}  y^{2\alpha - 1} \Phi[\bvarepsilon](\bx, y, t) \D y.
\end{equation}
Let $g(\bx,t) = \frac{\partial}{\partial t} \varepsilon(\bx, t)$, then the exact solution of the auxiliary dynamics \eqref{dynamics_response} reads that
\begin{equation}\label{exact_solution}
\Phi[\bvarepsilon](\bx, y, t) = \me^{-y^2 t} \Phi[\bvarepsilon](\bx, y, 0) + \int_0^t \me^{-y^2 \tau} g(\bx, t-\tau) \D \tau.
\end{equation}
For sufficiently large $y$ and by integrations by parts, it has that
\begin{equation}\label{integrations_by_parts}
\begin{split}
\Phi[\varepsilon](\bx,y, t) = &\me^{-y^2 t} \phi(y, 0) + \frac{1}{y^2} \left[ g(\bx, t) - \me^{-y^2 t} g(\bx, 0)\right]  \\
&+ \int_0^t \me^{-y^2 \tau} \frac{\partial}{\partial t} g(\bx, t-\tau) \D \tau,
\end{split}
\end{equation}
which implies that Eq.~\eqref{singular_integral} is well defined since $\Phi[\bvarepsilon]$ decays as $\mathcal{O}(y^{-2})$.

By introducing the velocity vector $\bv(\bx, t) = \frac{\partial }{\partial t} \bu(\bx, t)$ to replace $\bu$, the extension problem of the constant-Q wave equation is cast into a parabolic system,
\begin{equation}
\frac{\partial }{\partial t} \begin{pmatrix} \bv \\ \bm{\sigma} \\ \Phi[\bvarepsilon] \end{pmatrix} =\begin{pmatrix} 0 & \mathcal{L}_3 & 0 \\0 & 0 & \mathcal{L}_2 \\ \mathcal{L}_1 & 0 & -y^2 \mathcal{I}_d \end{pmatrix} \begin{pmatrix} \bv \\ \bm{\sigma} \\ \Phi[\bvarepsilon] \end{pmatrix} + \begin{pmatrix} \bm{f} \\ 0  \\ 0\end{pmatrix}.
\end{equation}
or equivalently, 
\begin{equation}\label{attenuated_wave_eqn_set}
\frac{\partial }{\partial t} \begin{pmatrix} \bv \\ \Phi[\bvarepsilon] \end{pmatrix} =\begin{pmatrix} 0 & \mathcal{L}_3 \circ \mathcal{L}_2 \\ \mathcal{L}_1 & -y^2 \mathcal{I}_d \end{pmatrix} \begin{pmatrix} \bv \\ \Phi[\bvarepsilon] \end{pmatrix} + \begin{pmatrix} \bm{f} \\ 0 \end{pmatrix}
\end{equation}
with $\mathcal{I}_d$ a $d\times d$ identity matrix. The operator $y^2 \mathcal{I}_d$ breaks the Hamiltonian structure of the elastic wave equation and results in a relaxation in $\Phi[\bvarepsilon]$.

\subsection{Typical examples}

Now we give two typical examples of constant-Q wave equation that characterizes an interpolated regime between the wave propagation and the heat diffusion. In unidimensional case, the stress-strain relation is that 
\begin{equation}\label{1d_stress_strain}
\sigma(x, t) = C \rho(x) {_C}D_t^{2\gamma} \varepsilon(x, t),
\end{equation}
where $C$ is a constant depending on $c_0, \gamma, t_0$. For its extension, the linear bounded operator $\mathcal{L}_2$ and the gradient operators $\mathcal{L}_1$,  $\mathcal{L}_3$ are given by
\begin{align}
\mathcal{L}_1 v(x, t) & = \frac{1}{\rho(x)} \frac{\partial}{\partial x} v(x, t), \\
(\mathcal{L}_2 \Phi[\bvarepsilon])(x, t) & = \frac{2C\sin(2\pi \gamma)}{\pi}  \left[\rho(x) \int_0^{+\infty} y^{4\gamma-1} \Phi[\bvarepsilon](x, y, t) \D y\right], \label{memory_strain_relation}\\
\mathcal{L}_3 \sigma(x, t) &=  \rho(x) \frac{\partial}{\partial x} \sigma(x, t).
\end{align}

Equivalently, substituting \eqref{1d_stress_strain} into Eqs.~\eqref{conservation_momentum} and \eqref{defintion_strain} yields that 
\begin{equation}
\frac{\partial^2}{\partial t^2} v(x, t) = \frac{C}{\rho(x)} \frac{\partial}{\partial x}\left(\rho(x)\frac{\partial}{\partial x}  {_C}D_t^{2\gamma} v(x, t)\right) + \frac{\partial}{\partial t} f(x, t).
\end{equation}
When $\rho(x) = \rho_0$ is constant and $\frac{\partial}{\partial t}f(x, t) \equiv 0$, it reduces to the standard diffusive wave equation
\begin{equation}\label{diffusive_wave_eqn}
{_C}D_t^{2- 2\gamma} v(x, t) = C \frac{\partial^2}{\partial x^2} v(x, t).
\end{equation}
For $v(x, 0^+) = v_0(x)$ and $\frac{\partial }{\partial t}v (x, 0^+) = 0$, the exact solution reads that
\begin{equation}\label{exact_solution_1d_diffusive}
\begin{split}
v(x, t) = \int_{-\infty}^{+\infty} G(x - y, t; 1 - \gamma) v_0(y) \D y, 
\end{split}
\end{equation}
where the Green's function $G(x, t; \nu)$ is given by  
\begin{equation}
G(x, t; \nu) = \frac{1}{2\sqrt{C} t^{\nu}} M_{\nu}\left(\frac{x}{\sqrt{C} t^{\nu}}\right), 
\end{equation}
and the Mainardi function $M_{\nu}(x)$ is a special kind of the Wright function of the second kind \cite{bk:Mainardi2010}, namely, $M_{\nu}(z) \coloneqq W_{-\nu, 1 - \nu}(-z), 0 < \nu < 1$.
%\begin{equation}
%M_{\nu}(z) \coloneqq W_{-\nu, 1 - \nu}(-z), \quad 0 < \nu < 1.
%\end{equation}
%where $W_{\lambda, \mu}(z)$ is defined by a contour integral.
%\begin{equation}
%W_{\lambda, \mu}(z) = \frac{1}{2\pi \mi } \int_{\mathcal{C}} \frac{\me^{\sigma+ z \sigma^{-\lambda}}}{\sigma^{\mu}} \D \sigma.
%\end{equation}

Intuitively, the temporal fractional PDE \eqref{diffusive_wave_eqn} describes the intermediate state of pure diffusion and pure wave dispersion, known as the fractional diffusion-wave phenomenon \cite{bk:Mainardi2010}. The Mainardi function under different $\nu$ is presented in Figure \ref{Mainardi_function}.
When $\gamma = 1/2$, it corresponds to the heat kernel and describes the wave dissipation. 
%\begin{equation}
%\frac{\partial}{\partial t} \varepsilon(x, t) = C \frac{\partial^2}{\partial x^2} \varepsilon(x, t).
%\end{equation}
When $\gamma = 0$, it corresponds to the wave equation and describes the wave dispersion.
%\begin{equation}
%\frac{\partial^2}{\partial t^2} \varepsilon(x, t) = C \frac{\partial^2}{\partial x^2} \varepsilon(x, t).
%\end{equation}
For $0 < \gamma < 1/2$, the solution may exhibit both wave dissipation and velocity dispersion.

Now considering the wave propagation in two-dimensional $(x, z)$-domain, the stress-strain relation reads that \cite{Carcione2009,ZhuCarcione2014}
\begin{equation}\label{2d_stress_strain}
\begin{split}
& \frac{\sigma_{11}(x, z, t)}{\rho(x, z)} =  C_P \cdot {_C}D_t^{\gamma_P} (\varepsilon_{11}(x, z, t) + \varepsilon_{33}(x, z, t)) - 2 C_S \cdot {_C}D_t^{\gamma_S} \varepsilon_{33}(x, z, t), \\ 
& \frac{\sigma_{33}(x, z, t)}{\rho(x, z)} =  C_P \cdot{_C}D_t^{\gamma_P} (\varepsilon_{11}(x, z, t) + \varepsilon_{33}(x, z, t)) - 2 C_S \cdot {_C}D_t^{\gamma_S} \varepsilon_{11}(x, z, t), \\ 
&  \frac{\sigma_{13}(x, z, t)}{\rho(x, z)} = 2  C_S \cdot {_C}D_t^{\gamma_S}  \varepsilon_{13}(x, z, t).
\end{split}
\end{equation}
and the linear operators in its extension problem read that
\begin{equation}
\mathcal{L}_1 \begin{pmatrix} v_1 \\ v_3 \end{pmatrix} 
= \frac{1}{\rho(x, z)}
\begin{pmatrix}
\frac{\partial }{\partial x} & 0 \\ 
0 & \frac{\partial }{\partial z} 
\\ \frac{1}{2}\frac{\partial }{\partial z} & \frac{1}{2}\frac{\partial }{\partial x}
\end{pmatrix}
 \begin{pmatrix} v_1 \\ v_3 \end{pmatrix} 
\end{equation}
and 
\begin{equation}
\begin{split}
\mathcal{L}_2 \begin{pmatrix} \Phi[\varepsilon_{11}] \\ \Phi[\varepsilon_{33}] \\ \Phi[\varepsilon_{13}] \end{pmatrix}
= &\rho
\begin{pmatrix}
 \tilde{C}_{P, \gamma_P} \int_0^{+\infty}  y^{4 \gamma_P -1} \left(\Phi[\varepsilon_{11}]+  \Phi[\varepsilon_{33}]\right) \D y  \\
 \tilde{C}_{P, \gamma_P} \int_0^{+\infty}  y^{4 \gamma_P -1} \left(\Phi[\varepsilon_{11}]+  \Phi[\varepsilon_{33}]\right) \D y  \\
0
\end{pmatrix} \\
&+ \rho
\begin{pmatrix}
-2  \tilde{C}_{S, \gamma_S} \int_0^{+\infty} y^{4 \gamma_S -1}\Phi[\varepsilon_{33}] \D y \\
-2  \tilde{C}_{S, \gamma_S}  \int_0^{+\infty} y^{4 \gamma_S -1}\Phi[\varepsilon_{11}] \D y \\
2  \tilde{C}_{S, \gamma_S}  \int_0^{+\infty} y^{4 \gamma_S -1}\Phi[\varepsilon_{13}] \D y
\end{pmatrix}, 
\end{split}
\end{equation}
where $\tilde{C}_{P, \gamma_P} = \frac{2\sin(2\pi \gamma_P)}{\pi} C_P$, $\tilde{C}_{S, \gamma_S} = \frac{2\sin(2\pi \gamma_S)}{\pi} C_S$, and
\begin{equation}
\mathcal{L}_3 \begin{pmatrix} \sigma_{11} \\ \sigma_{33} \\ \sigma_{13} \end{pmatrix} = \rho(x, z)
\begin{pmatrix}
\frac{\partial }{\partial x} & 0 &  \frac{\partial }{\partial z} \\ 
0& \frac{\partial }{\partial z} &  \frac{\partial }{\partial x} \\
\end{pmatrix} 
\begin{pmatrix} \sigma_{11} \\ \sigma_{33} \\ \sigma_{13} \end{pmatrix}.
\end{equation}
Formally, when $\gamma_S = \gamma_P = 0$, the fractional derivatives reduce to identity operators. The whole set of equations are equalivalent  to the first-order hyperbolic system in the elastic regime \cite{Virieux1986}.

%When $\gamma_S = \gamma_P = 0$, by introducing the velocity vector $(v_1, v_3)$, 
%\begin{equation}
%\begin{split}
%v_1(x, z, t) = \frac{\partial }{\partial t} u_1(x, z, t), \\
%v_3(x, z, t) = \frac{\partial }{\partial t} u_3(x, z, t), \\
%\end{split}
%\end{equation}
%the set of equations \eqref{2d_conservation_momentum}-\eqref{2d_stress_strain} reduces to the first-order hyperbolic system
%\begin{equation}\label{1}
%\left\{
%\begin{split}
%&\frac{\partial}{\partial t} v_1(x, z, t) = \rho^{-1} \left[\frac{\partial}{\partial x} \sigma_{11}(x, z, t) + \frac{\partial}{\partial z} \sigma_{13}(x, z, t)\right] + f_1(x, z, t), \\
%&\frac{\partial}{\partial t} v_3(x, z, t) = \rho^{-1} \left[\frac{\partial}{\partial x} \sigma_{13}(x, z, t) + \frac{\partial}{\partial z} \sigma_{33}(x, z, t)\right] + f_3(x, z, t),  \\
%&\frac{\partial}{\partial t} \sigma_{11}(x, z, t) = \rho C_P \frac{\partial}{\partial x} v_1(x, z, t) + \rho (C_P - 2 C_S) \frac{\partial}{\partial z} v_3(x, z, t), \\
%&\frac{\partial}{\partial t} \sigma_{33}(x, z, t) =  \rho (C_P - 2 C_S) \frac{\partial}{\partial x} v_1(x, z, t) + \rho C_P \frac{\partial}{\partial z} v_3(x, z, t),\\
%&\frac{\partial}{\partial t} \sigma_{13}(x, z, t) = \rho C_S \frac{\partial}{\partial x} v_3(x, z, t) + \rho C_S \frac{\partial}{\partial z} v_1(x, z, t). 
%\end{split}
%\right.
%\end{equation}
%However, when $\gamma_S \ne \gamma_P$, the set of equations becomes parabolic.

\subsection{Initial condition for the extension problem with proper asymptotics}

One has to be very careful about the initial condition of $\Phi[\bvarepsilon](\bx, y, t)$ as simply posing zero initial condition may result in an incorrect asymptotic behavior \cite{SchmidtGaul2006}. Actually, it requires to satisfy both the stress-strain relation \eqref{stress_strain_relation} and 
\begin{equation}
{_C}D_t^\alpha \bvarepsilon(\bx, t) |_{t = 0^+} = \frac{2\sin(\pi \alpha)}{\pi} \int_0^{+\infty}  y^{2\alpha - 1} \Phi[\bvarepsilon](\bx, y, t) \big |_{t = 0^+}\D y.
\end{equation}
For unidimensional case, it requires to satisfy
\begin{equation}
\sigma(x, t) |_{t = 0^+} = \frac{2 C\sin(2 \pi \gamma) \rho(x)}{\pi} \int_0^{+\infty}  y^{4\gamma - 1} \Phi[\varepsilon](x, y, t) |_{t = 0^+} \D y. 
\end{equation}
so that the initial condition of $\Phi[\bvarepsilon](x, y, t)$ should be
\begin{equation}
\Phi[\varepsilon](x, y, t) |_{t = 0^+} = \frac{\Gamma(1 - 2\gamma)}{C\rho(x)} \cdot \me^{-y^2} \sigma(x, t) |_{t = 0^+},
\end{equation}
Here we utilize Eq.~\eqref{identity_gamma_function} and the following identity $\Gamma(2\gamma) = 2\int_{0}^{+\infty} y^{4\gamma - 1} \me^{- y^2 } \D y$.

Similarly, the initial conditions of the auxiliary functions in two-dimensional case are given by 
\begin{equation}
\begin{split}
&\Phi[\varepsilon_{11}](x, z, y, t) |_{t=0^+} = \frac{\me^{-y^2}}{\rho(x, z)}\left[a_{P, S} (\sigma_{11} + \sigma_{13})  +  b_{S} (\sigma_{11} - \sigma_{13})\right](x, z, t) |_{t=0^+}, \\
&\Phi[\varepsilon_{33}](x, z, y, t) |_{t=0^+} = \frac{\me^{-y^2}}{\rho(x, z)}\left[a_{P, S} (\sigma_{11} + \sigma_{13})  +  b_{S} (\sigma_{11} - \sigma_{13})\right](x, z, t) |_{t=0^+}, \\
&\Phi[\varepsilon_{13}](x, z, y, t) |_{t=0^+}  = \frac{\me^{-y^2}}{\rho(x, z)} b_S \sigma_{13}(x, z, t) |_{t=0^+},
\end{split}
\end{equation}
where
\begin{equation}
a_{P, S} = \frac{ \Gamma(1 - 2\gamma_P) \Gamma(1 - 2\gamma_S)}{4C_P \Gamma(1 - 2\gamma_S) - 4C_S \Gamma(1 - 2\gamma_P)}, \quad b_S = \frac{\Gamma(1 - 2\gamma_S)}{4 C_S }.
\end{equation}
In particular, when the initial stress tensor is zero, it recovers the setting in the existing literatures.

%%%%%%%%%%%%%%%%%%%%
\section{Short-memory operator splitting scheme}
\label{sec.scheme}

The standard treatment for evaluating the weakly singular integral \eqref{singular_integral} over a semi-infinite domain, as suggested in the Yuan-Agrawal method and the diffusive approximation \cite{YuanAgrawal2002,LuHanyga2005,BlancChiavassaLombard2014,Diethelm2008}, is the Laguerre-Gauss quadrature, and the auxiliary function $\Phi[\bvarepsilon](\bx, y, t)$ at each time step has to be stored in computation. This endows our short-memory principle with two-fold meanings. First, it requires to reduce the number of  Laguerre-Gauss quadrature nodes to alleviate the storage of $\Phi[\bvarepsilon](\bx, y, t)$. Owing to the localized Abel kernel $y^{4\gamma-1}, \gamma \ll 1$ in Eq.~\eqref{singular_integral}, we can use only a few Laguerre-Gauss quadrature nodes to capture the major contribution of the pre-asymptotic range \cite{Diethelm2008,bk:ShenTangWang2011} and avoid the redundant nodes in the post-asymptotic range. Second, it is relevant to choose an appropriate scaling factor $\beta > 1$ to avoid the deterioration of numerical accuracy in dynamics.

To illustrate the latter point, we need to establish a rigorous theoretical analysis of the spectral approximation within the framework of generalized Laguerre interpolation \cite{GuoWangWang2006}. To the best of our knowledge, the mathematical analysis of the Yuan-Agarawl method and the diffusive approximation is rarely found in literatures except that based on asymptotic analysis \cite{Diethelm2008}.  We will show that there is a dynamical increment of projection error, which is induced by the impulse of the source term in the auxiliary relaxed dynamics \eqref{dynamics_response}. As a result, it hampers the accuracy of the Laguerre-Gauss quadrature due to the ill-posedness of the Laguerre functions at infinity. This partially accounts for the criticism on the original Yuan-Agrawal method \cite{SchmidtGaul2006,Diethelm2008}. Fortunately, we will show that  the growth of errors can be considerably alleviated by using the scaling technique \cite{Tang1993,Shen2000,bk:ShenTangWang2011}. 

An efficient splitting scheme is then introduced to solve the set of equations. The relevant point is to solve the dynamics of auxiliary functions $\Phi[\bvarepsilon](\bx, y, t)$ exactly, thereby avoiding both the numerical stiffness and numerical errors in the standard ODE solvers. We mainly focus on the Strang splitting for brevity, but higher-order scheme, such as the Yoshida approximation, can be straightforwardly obtained \cite{Yoshida1990}.

\subsection{Generalized Laguerre-Gauss projection and scaling}

Regarding the fact that the Laguerre functions grow at infinity, we introduce a smooth truncating function $\chi(y)$ with sufficiently large cut-off radius $r_c$. Thus the domain is decomposed into a pre-asymptotic range and a post-asymptotic range,
 \begin{equation}\label{decomposition}
 \begin{split}
 \int_0^{+\infty}  y^{4\gamma - 1} \Phi[\bvarepsilon](\bx, y, t)  \D y =  &\int_0^{+\infty}\omega_{4\gamma-1, \beta}\widetilde{\Phi}[\bvarepsilon](\bx, y, t) \D y \\
 &+  \int_0^{+\infty}  y^{4\gamma - 1} \Phi[\bvarepsilon](\bx, y, t) (1 - \chi( \beta y)) \D y.
 \end{split}
 \end{equation}
  where the weight is $\omega_{\alpha, \beta}(y) = y^{\alpha} \me^{-\beta y}$ and $\beta$ is called the scaling factor, 
 \begin{equation}
\widetilde{\Phi}[\bvarepsilon](\bx, y, t) =  \xi(\beta y) \Phi[\bvarepsilon](\bx, y, t), \quad \xi( y) = \me^{ y} \chi( y). 
 \end{equation}

The first term of Eq.~\eqref{decomposition} can be evaluated by  generalized Laguerre-Gauss quadrature,
 \begin{equation}\label{Laguerre_Gauss_quadarture}
 \int_0^{+\infty} y^{4\gamma - 1} \me^{-\beta y} \left[\me^{\beta y} \Phi[\bvarepsilon](\bx, y, t) \right] \D y \approx \sum_{j=0}^{M} \omega_j^{(4\gamma-1, \beta)} \me^{\beta y_j^{(4\gamma-1, \beta)}}  \Phi[\bvarepsilon](\bx, y_j^{(4\gamma-1, \beta)}, t).
\end{equation}
Following \cite{GuoWangWang2006}, there is a simple relation between $y_j^{(\alpha, 1)}$ and $y_j^{(\alpha, \beta)}$, as well as $\omega_j^{(\alpha,1)}$ and $\omega_j^{(\alpha, \beta)}$,
\begin{equation}
y_j^{(\alpha,\beta)} = {\beta}^{-1} y_j^{(\alpha,1)}, \quad \omega_j^{(\alpha, \beta)} = {\beta^{-(\alpha+1)}} \omega_j^{(\alpha,1)},
\end{equation}
where $\{y_j^{(\alpha,1)}\}_{j=0}^M$ are the zeros of the Laguerre functions $\mathscr{L}_{M+1}^{(\alpha)}(y)$, and the weights are given by
\begin{equation}
\omega_j^{(\alpha,1)} = \frac{\Gamma(M+\alpha + 1)}{(M+\alpha+1)(M+1)!} \frac{y_j^{(\alpha)}}{[L^{(\alpha)}_{M}(y_j^{(\alpha)})]^2}, \quad 0 \le j \le M.
\end{equation}
When $\beta = 1$, it reduces to the improved Yuan-Agrawal method \cite{Diethelm2008}.

The second part of Eq.~\eqref{decomposition} decays as
\begin{equation}\label{2nd_decay}
\int_0^{+\infty} y^{4\gamma-1} \Phi[\bvarepsilon](\bx, y, t)(1 - \chi(\beta y)) \D y\lesssim \int_0^{+\infty}  y^{4\gamma - 3} (1 - \chi(\beta y)) \D y \lesssim \left(\frac{r_c}{\beta}\right)^{4\gamma - 2}.
\end{equation}
Such decay rate is adequate for the convergence of the weakly singular integral, but produces very slow algebraic convergence rate for the Laguerre-Gauss quadrature after the pre-asymptotic range \cite{bk:ShenTangWang2011}, or even leads to numerical instability for too large $M$ due to the exponential factors.

 \subsection{Dynamical increments of projection error}
 
 The formula \eqref{Laguerre_Gauss_quadarture} essentially utilizes the Laguerre-Galerkin approximation to the relaxed dynamics \eqref{dynamics_response}. To characterize the numerical error, we introduce an $L_{\omega_{\alpha, \beta}}^2$-orthogonal projection $\mathcal{P}_{M, \alpha, \beta}$ with $\alpha = 4\gamma -1$ \cite{Shen2000,GuoWangWang2006},
\begin{equation}
\mathcal{P}_{M, \alpha, \beta} \widetilde{\Phi}[\bvarepsilon](\bx, y, t) = \sum_{m=0}^{M} a_m(\bx, t) \mathscr{L}_{m}^{(\alpha, \beta)}(y),
\end{equation}
where $\mathscr{L}_{m}^{(\alpha, \beta)}(y)$ are generalized Laguerre-Gauss polynomials with $\mathscr{L}_{m}^{(\alpha, 1)}(y) = \mathscr{L}_{m}^{(\alpha)}(y)$ \cite{GuoWangWang2006,XiaShaoChou2020},
\begin{equation}
a_m(\bx, t) = \frac{1}{\gamma_{m}^{(\alpha,\beta)}}(\widetilde{\Phi}[\bvarepsilon], \mathscr{L}_{m}^{(\alpha, \beta)})_{\omega_{\alpha, \beta}}, \quad (u, v)_{\omega_{\alpha, \beta}} =\int_0^{+\infty} y^\alpha \me^{-\beta y} u(y) v(y)  \D y,
\end{equation}
and $\gamma_{m}^{(\alpha,\beta)} = \frac{\Gamma(m+\alpha+1)}{\beta^{\alpha+1} \Gamma(m+1)}$, with the semi-norm $|\cdot|_{A^r_{\alpha, \beta}}$ and norm $\Vert \cdot \Vert_{A^r_{\alpha, \beta}}$
\begin{equation}
|v|_{A^r_{\alpha, \beta}}^2 = \Vert \partial_y^r v \Vert_{\omega_{\alpha+r, \beta}}^2 = ( \partial_y^r v ,  \partial_y^r v )_{\omega_{\alpha+r, \beta}}, \quad \Vert v \Vert_{A^r_{\alpha, \beta}} = \left(\sum_{k=0}^r |v|_{A^k_{\alpha, \beta}}^2 \right)^{1/2}. 
\end{equation} 
and the orthogonal relation $( \mathscr{L}_{m}^{(\alpha, \beta)},  \mathscr{L}_{l}^{(\alpha, \beta)})_{\omega_{\alpha, \beta}} = \gamma_{m}^{(\alpha, \beta)} \delta_{l,m}$.

Replacing $\widetilde{\Phi}[\bvarepsilon](\bx, y, t)$ by the projection $\mathcal{P}_{M, \alpha, \beta} \widetilde{\Phi}[\bvarepsilon](\bx, y, t)$, and using the fact that \cite{GuoWangWang2006}
\begin{equation}
 \int_0^{+\infty}\omega_{\alpha, \beta}(y) \mathscr{L}_{m}^{(\alpha, \beta)}(y) \D y =  \sum_{j = 0}^{M} \omega_j^{(\alpha, \beta)}  \mathcal{L}^{(\alpha, \beta)}_{m} (y_j^{(\alpha, \beta)}), \quad 0 \le m \le M,
\end{equation}
where $y_j^{(\alpha, \beta)}$ are roots of the generalized Laguerre polynomial $\mathscr{L}_{M+1}^{(\alpha, \beta)}(y)$, it arrives at
\begin{equation*}
\begin{split}
 \int_0^{+\infty}\omega_{\alpha, \beta}(y) \mathcal{P}_{M, \alpha, \beta} \widetilde{\Phi}[\bvarepsilon](\bx, y, t) \D y &= \sum_{j=0}^{M} \omega_j^{(\alpha, \beta)} \sum_{m=0}^{M+1} a_m(\bx, t) \mathscr{L}_{m}^{(\alpha, \beta)}(y_j^{(\alpha, \beta)}) \\
 & \approx \sum_{m=0}^M \omega_j^{(\alpha, \beta)} \me^{\beta y_j^{(4\gamma-1, \beta)}}  \Phi[\bvarepsilon](\bx, y_j^{(\alpha)}, t) \chi(\beta y_j^{(\alpha)}),
 \end{split}
\end{equation*}
 which recovers the formula \eqref{Laguerre_Gauss_quadarture} when $y_M^{(\alpha)} < r_c$.  We need to analyze the projection error 
\begin{equation}
\mathcal{E}_{M, \alpha, \beta}(\bx, y, t) = \widetilde{\Phi}[\bvarepsilon](\bx, y, t) - \mathcal{P}_{M, \alpha, \beta} \widetilde{\Phi}[\bvarepsilon](\bx, y, t)
\end{equation}
as the norm $\Vert \mathcal{E}_{M, \alpha, \beta}(\bx, y, t) \Vert_{\omega_{\alpha, \beta}}$ can be used to bound the numerical error of the generalized Laguerre-Gauss formula \eqref{Laguerre_Gauss_quadarture},
\begin{equation}
\Big | \int_0^{+\infty}\omega_{\alpha, \beta}(y) \mathcal{E}_{M, \alpha, \beta}(\bx, y, t)  \D y \Big | \le \big | \int_0^{+\infty}\omega_{\alpha, \beta}(y) \D y \big |^{\frac{1}{2}}  \Vert \mathcal{E}_{M, \alpha, \beta}(\bx, y, t) \Vert_{\omega_{\alpha, \beta}}.
\end{equation}
The following theorem characterizes the dynamical bound of $\Vert \mathcal{E}_{M, \alpha, \beta}(\bx, y, t) \Vert_{\omega_{\alpha, \beta}}$.  

\begin{theorem}\label{error_bound} Suppose $\Vert \Phi[\varepsilon](\bx, y, 0) \Vert_{A^1_{\alpha, \beta}} < \infty$ and $\max_{\bx,t} |g(\bx, t)| < \infty$, then for any $\alpha > - 1$ and $\beta$, it has that 
\begin{equation*}
\begin{split}
&\Vert \mathcal{E}_{M, \alpha, \beta}(\bx, y, t) \Vert_{\omega_{\alpha, \beta}} \le  M^{-\frac{1}{2}} \beta^{-\frac{\alpha}{2} -1} t \left(\beta \Vert \partial_y \xi  \Vert_{\omega_{\alpha, 1}} + 2 \beta^{-\frac{1}{2}} t \Vert \xi \Vert_{\omega_{\alpha+1, 1}} \right)  \max_{\bx,t} |g(\bx, t)|\\
%&+ M^{-\frac{1}{2}} \beta^{-\alpha -\frac{3}{2}}  \Vert \partial_y \Phi[\varepsilon](\bx, y, 0) \Vert_{\omega_{\alpha, \beta}}  \\
&\quad \quad + M^{-\frac{1}{2}} \beta^{-\frac{\alpha}{2} -1} \left(\beta\Vert \partial_y \xi  \Vert_{\omega_{\alpha, 1}} + 2\beta^{-\frac{1}{2}} t \Vert \xi \Vert_{\omega_{\alpha+1, 1}} + \Vert \xi \Vert_{\omega_{\alpha, 1}}\right) \Vert \Phi[\varepsilon](\bx, y, 0) \Vert_{A^1_{\alpha, \beta}}  .
\end{split}
\end{equation*}
\end{theorem}

The proof of Theorem \ref{error_bound} is based on the well-known bound for the $L_{\omega_{\alpha, \beta}}^2$-projection \cite{GuoWangWang2006}. 
\begin{lemma}\label{lemma_proj_error}
For any $v$  with $|v|_{A_{\alpha, \beta}^r} < \infty$, an integer $r$, and $0 \le \mu \le r$, 
\begin{equation}
\Vert \mathcal{P}_{M, \alpha, \beta}v - v \Vert_{A_{\alpha, \beta}^{\mu}} \le c (\beta M)^{\frac{\mu- r}{2}} |v|_{A_{\alpha, \beta}^r}.
\end{equation}
\end{lemma}

Here we first consider $\mu =0$ and $r = 1$. Despite the spectral accuracy,  it points out that $\Vert \mathcal{E}_{M, \alpha, \beta}(\bx, y, t) \Vert_{\omega_{\alpha, \beta}}$ may be amplified  due to the increments from the source term $g(\bx, t)$.
\begin{proof}[Proof of Theorem \ref{error_bound}] It starts with
\begin{equation*}
\partial_y \widetilde{\Phi}[\bvarepsilon](\bx, y, t) = \left(\partial_y \xi(\beta y)\right) \Phi[\bvarepsilon](\bx, y, t) + \xi(\beta y) \partial_y \Phi[\bvarepsilon](\bx, y, t).
\end{equation*}
By the triangular inequality, it has that
\begin{equation*}
\begin{split}
\Vert \partial_y \widetilde{\Phi}[\bvarepsilon](\bx, y, t) \Vert_{\omega_{\alpha, \beta}} &\le \Vert \partial_y \xi(\beta y) \Phi[\bvarepsilon](\bx, y, t) \Vert_{\omega_{\alpha, \beta}} + \Vert \xi(\beta y) \partial_y \Phi[\bvarepsilon](\bx, y, t) \Vert_{\omega_{\alpha, \beta}}.
\end{split}
\end{equation*}

Since $\Phi(\bx, y, t)$ satisfies Eq.~\eqref{exact_solution} and
\begin{equation*}
\partial_y \Phi(\bx, y, t) = -2y t \me^{-y^2 t} \Phi[\bvarepsilon](\bx, y, 0) + \me^{-y^2 t} \partial_y \Phi[\bvarepsilon](\bx, y, 0) - 2 y \int_0^t  \tau \me^{-y^2 \tau} g(\bx, t - \tau) \D \tau,
\end{equation*}
by the triangular inequality, it further yields that
\begin{equation*}
\begin{split}
\Vert \partial_y \xi(\beta y)  \Phi[\bvarepsilon](\bx, y, t) \Vert_{\omega_{\alpha, \beta}} \le &\Vert \partial_y \xi(\beta y)  \Vert _{\omega_{\alpha, \beta}}\Vert \me^{-y^2 t}\Phi[\bvarepsilon](\bx, y, 0) \Vert_{\omega_{\alpha, \beta}} \\
&+ \Vert \partial_y \xi(\beta y) \int_0^t \me^{-y^2 \tau} g(\bx, t-\tau) \D \tau \Vert_{\omega_{\alpha, \beta}} \\
\le &  \Vert \partial_y \xi(\beta y)  \Vert _{\omega_{\alpha, \beta}} (\Vert \Phi[\bvarepsilon](\bx, y, 0) \Vert_{\omega_{\alpha, \beta}} + t  \max_{\bx,t} |g(\bx, t)|),
\end{split}
\end{equation*}
and
\begin{equation*}
\begin{split}
\Vert \xi(\beta y)  \partial_y \Phi[\bvarepsilon](\bx, y, t) \Vert_{\omega_{\alpha, \beta}} \le & 2 t \Vert \xi(\beta y) \Vert_{\omega_{\alpha+1, \beta}} \Vert  \Phi[\bvarepsilon](\bx, y, 0) \Vert_{\omega_{\alpha, \beta}} \\
&+  \Vert \xi(\beta y) \Vert_{\omega_{\alpha, \beta}}   \Vert \partial_y \Phi[\bvarepsilon](\bx, y, 0) \Vert_{\omega_{\alpha, \beta}} \\
&+ 2 t^2 \Vert \xi(\beta y) \Vert_{\omega_{\alpha+1, \beta}}  \max_{\bx,t}|g(\bx, t)|,
\end{split}
\end{equation*}
which utilizes the fact that
\begin{equation}\label{t_estimate}
\Big |  \int_0^t  \tau^k \me^{-y^2 \tau } g(\bx, t - \tau) \D \tau \Big | \le t^k  \left( \int_0^t  \me^{-y^2 \tau} \D \tau\right)  \max_{\bx,t} |g(\bx, t) \le t^{k+1}  \max_{\bx,t}|g(\bx, t)|.
\end{equation}
Finally, since
\begin{equation} 
 \Vert \partial_y^k \xi(\beta y) \Vert_{\omega_{\alpha,\beta}} =   \beta^{-\frac{\alpha+1}{2} +k} \Vert \partial_y \xi(y) \Vert_{\omega_{\alpha,1}} ,
\end{equation}
it completes the proof.
\end{proof}

The result can be generalized to $r > 1$ as $\Phi[\bvarepsilon](\bx, y, t)$ is infinitely smooth in $y$-variable.
\begin{theorem}\label{error_bound_2} Suppose $\Vert \me^\frac{y}{2} \Phi[\varepsilon](\bx, y, 0) \Vert_{A^r_{\alpha, \beta}} < \infty$ and $\max_{\bx,t} |g(\bx, t)| < \infty$, then for any $\alpha > - 1$ and $\beta$, $r \ge 1$, it has that 
\begin{equation*}
\begin{split}
&\Vert \mathcal{E}_{M, \alpha, \beta}(\bx, y, t) \Vert_{\omega_{\alpha, \beta}}  \lesssim  c_1 M^{-\frac{r}{2}} \beta^{-\frac{\alpha+1}{2}} (1+\beta^{-1})^{r} \Vert \me^{-\frac{ y}{2}} \partial_y^k \xi(y) \Vert_{\omega_{\alpha+r, 1}}  \\
&\quad \quad + c_2 M^{-\frac{r}{2}} \beta^{-\frac{\alpha+1}{2}} (1 + \beta^{-1})^{2r} (1+t)^r  \left( \sum_{s_1 = 0}^{r} \Vert \partial_y^k \xi(y) \Vert_{\omega_{\alpha+r+2s_1,1}}\right) \max_{\bx,t} |g(\bx, t)|,
\end{split}
\end{equation*}
where $c_1$ and $c_2$ are independent of $\beta$ and $M$.
\end{theorem}
\begin{proof}
It starts with
\begin{equation*}
\begin{split}
\Vert \partial_y^r \widetilde{\Phi}[\bvarepsilon](\bx, y, t) \Vert_{\omega_{\alpha+r, \beta}} \le &\underbracket{\sum_{k=0}^r C_r^k \Vert \me^{-\frac{\beta y}{2}}\partial_y^k \xi(\beta y) \Vert_{\omega_{\alpha+r, \beta}}  \Vert \me^{\frac{\beta y}{2}}\partial_y^{r-k} (\me^{-y^2 t}\Phi[\bvarepsilon](\bx, y, 0)) \Vert_{\omega_{\alpha+r, \beta}}}_{\textbf{I}} \\
&+\underbracket{\sum_{k=0}^r C_r^k \Vert \partial_y^k \xi(\beta y) \cdot \partial_y^{r-k} (\int_0^t \me^{ -y^2 \tau} g(\bx, t-\tau) \D \tau)) \Vert_{\omega_{\alpha+r, \beta}}}_{\textbf{II}}.
\end{split}
\end{equation*}

For the first term, since
\begin{equation*}
\Vert \me^{{\beta y}/{2}} \cdot \partial_y^{r-k} (\me^{-y^2 t}\Phi[\bvarepsilon](\bx, y, 0)) \Vert_{\omega_{\alpha+r, \beta}} =  \Vert \me^{{y}/{2}}\cdot \partial_y^{r-k} (\me^{-y^2 t}\Phi[\bvarepsilon](\bx, y, 0)) \Vert_{\omega_{\alpha+r, 1}} 
\end{equation*}
we have that
\begin{equation*}
\textbf{I} \lesssim \sum_{k=0}^r C_r^k \beta^{-\frac{\alpha+r+1}{2} +k} \Vert \me^{-\frac{ y}{2}}\partial_y^k \xi(y) \Vert_{\omega_{\alpha+r, 1}}
= \beta^{-\frac{\alpha+r+1}{2}} (1 + \beta)^{r} \Vert \me^{-\frac{ y}{2}} \partial_y^k \xi(y) \Vert_{\omega_{\alpha+r, 1}}
\end{equation*}

For the second term, for arbitrary $0 \le s_1, s_2 \le r$, using Eq.~\eqref{t_estimate} yields
\begin{equation*}
\Vert \partial_y^k \xi(\beta y) \int_0^t  y^{s_1} t^{s_2} \me^{-y^2 \tau} g(\bx, t-\tau) \D \tau \Vert_{\omega_{\alpha+r, \beta}} \le t^{s_2+1}  \max_{\bx,t} |g(\bx, t)| \cdot \Vert \partial_y^k \xi(\beta y) \Vert_{\omega_{\alpha+r+2s_1, \beta}}
\end{equation*}
and $\Vert \partial_y^k \xi(\beta y) \Vert_{\omega_{\alpha+r+2s_1, \beta}} = \beta^{-\frac{\alpha+r+1}{2}  - s_1 + k}  \Vert \partial_y^k \xi(y) \Vert_{\omega_{\alpha,1}}$.

Summing over all possible terms yields the following estimate, 
\begin{equation*}
\begin{split}
\textbf{II} &\lesssim \sum_{s_1 = 0}^r \sum_{s_2 = 0}^r \sum_{k=0}^r t^{s_2+1} C_r^k  \beta^{-\frac{\alpha+r+1}{2}  - s_1 + k}  \max_{\bx,t} |g(\bx, t)| \cdot  \Vert \partial_y^k \xi(y) \Vert_{\omega_{\alpha+r+2s_1,1}} \\
&\lesssim \beta^{-\frac{\alpha+r+1}{2} } (1 + t)^{r+1} ( 1 + \beta)^r (1 + \frac{1}{\beta})^{r}  \max_{\bx,t} |g(\bx, t)| \cdot  \sum_{s_1 = 0}^{r} \Vert \partial_y^k \xi(y) \Vert_{\omega_{\alpha+r+2s_1,1}}. 
\end{split}
\end{equation*}
Combining with Lemma \ref{lemma_proj_error}, it completes the proof.
\end{proof}

According to Theorem 6.31.2 in \cite{bk:Szego1974}, the largest zero of $\mathscr{L}_M^{(\alpha)}$ satisfies $y_j^{(\alpha)} \lesssim 4M$. Combining Theorem \ref{error_bound_2} and the estimate \eqref{2nd_decay} of truncation error, we arrive at the Proposition \ref{prop_LG}. 
\begin{proposition}\label{prop_LG} Under the conditions in Theorem \ref{error_bound_2}, for $-1 < \alpha < 1$, it has that
\begin{equation*}
\begin{split}
\Big | \int_0^{+\infty}\omega_{\alpha, \beta}(y) \mathcal{E}_{M, \alpha, \beta}(\bx, y, t)  \D y \Big | \le & \underbracket{C_1 M^{\alpha-1} \beta^{1-\alpha}}_{\textup{truncation error}} + \underbracket{C_2 M^{-\frac{r}{2}} \beta^{-\frac{\alpha+1}{2}} (1+\beta^{-1})^{r}}_{\textup{projection error}} \\
& \hspace{-2cm}+ \underbracket{C_3 M^{-\frac{r}{2}} \beta^{-\frac{\alpha+1}{2}} (1 + \beta^{-1})^{2r} (1+t)^r  \max_{\bx,t} |g(\bx, t)|}_{\textup{projection error from source term}},
\end{split}
\end{equation*}
where $C_1$ is independent of $\beta$, $C_1$ and $C_2$ and $C_3$ are also independent of $M$.
\end{proposition}

Proposition \ref{prop_LG} states two facts. First, the accuracy of the Laguerre-Gauss quadrature under the prescribed collocation points may be diminished due to the terms $(1+t)^r$. Second, introducing a large scaling factor $\beta > 1$ can help reduce the dynamical increments of numerical errors in the pre-asymptotic range (see the term $\beta^{-\frac{\alpha+1}{2}}$). Although it may also augment the truncation errors as well (see the term $\beta^{1-\alpha})$, fortunately the pre-asymptotic range is much more important when $\gamma \ll 1$, so that choosing a large scaling factor $\beta$ can dramatically enhance the numerical accuracy.  

Here we provide an illustrative example to demonstrate the dynamical increments of projection error and the influences under different scaling factors.  Considering
\begin{equation}\label{dynamics_example}
I(t) = \int_0^{+\infty} y^{4\gamma-1} \phi(y, t) \D y, \quad \frac{\partial }{\partial t} \phi(y, t) =  1 - y^2 \phi(y, t), \quad \phi(y, 0) = 1,
\end{equation}
the exact solution of dynamics is $\phi(y, t) = \me^{-y^2 t}  + y^{-2} (1 - \me^{-y^2 t})$. Using Eq.~\eqref{Laguerre_Gauss_quadarture} to evaluate the integral $I(t)$,  with $M+1 = 500$, $\beta =1$ adopted as the reference, we can see in Figure \ref{Laguerre_Gauss_dynamics} that the accuracy of the Laguerre-Gauss quadrature deteriorates in time evolution, regardless of the number of collocation points.  After choosing a scaling factor $\beta > 1$, one can observe a slight reduction of accuracy at the early stage, but the amplification of numerical errors can be considerably suppressed in time evolution, while further increasing $\beta$ may lead to larger truncation error. But overall, the dynamical increments of numerical errors can be alleviated. By contrast, $\beta < 1$ will lead to a more severe increments of errors.

%%%%%%%%%%%%%%%%%%%%%%%%%%%%%%%%%%%%%%%%%%%%%%
\begin{figure}[!h]
\centering
\subfigure[$M+1 = 8$, $\gamma=0.1$.]{\includegraphics[width=0.49\textwidth,height=0.27\textwidth]{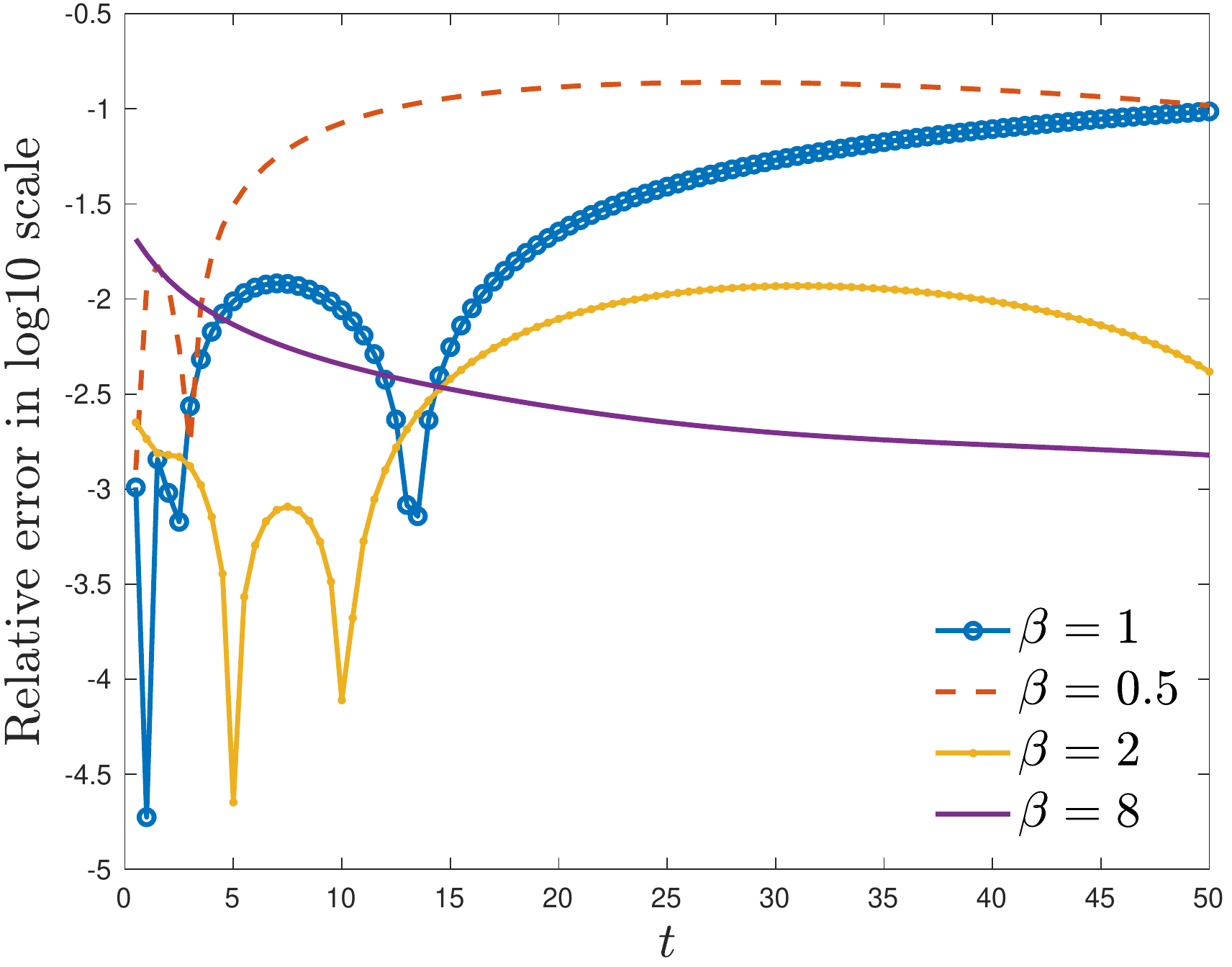}}
\subfigure[$M+1 = 32$, $\gamma=0.1$.]{\includegraphics[width=0.49\textwidth,height=0.27\textwidth]{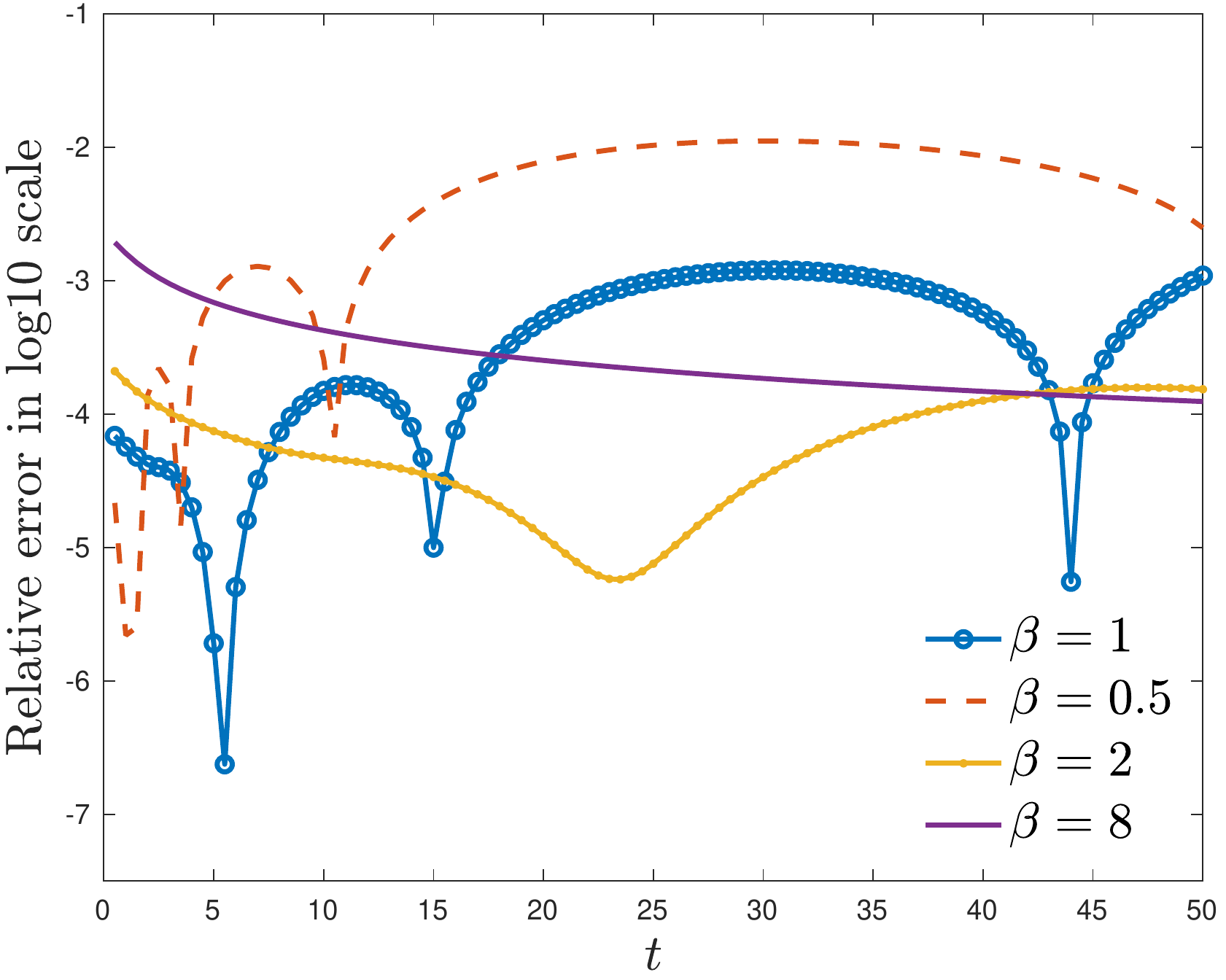}}
\\
\subfigure[$M+1 = 8$, $\gamma= 0.001$.]{\includegraphics[width=0.49\textwidth,height=0.27\textwidth]{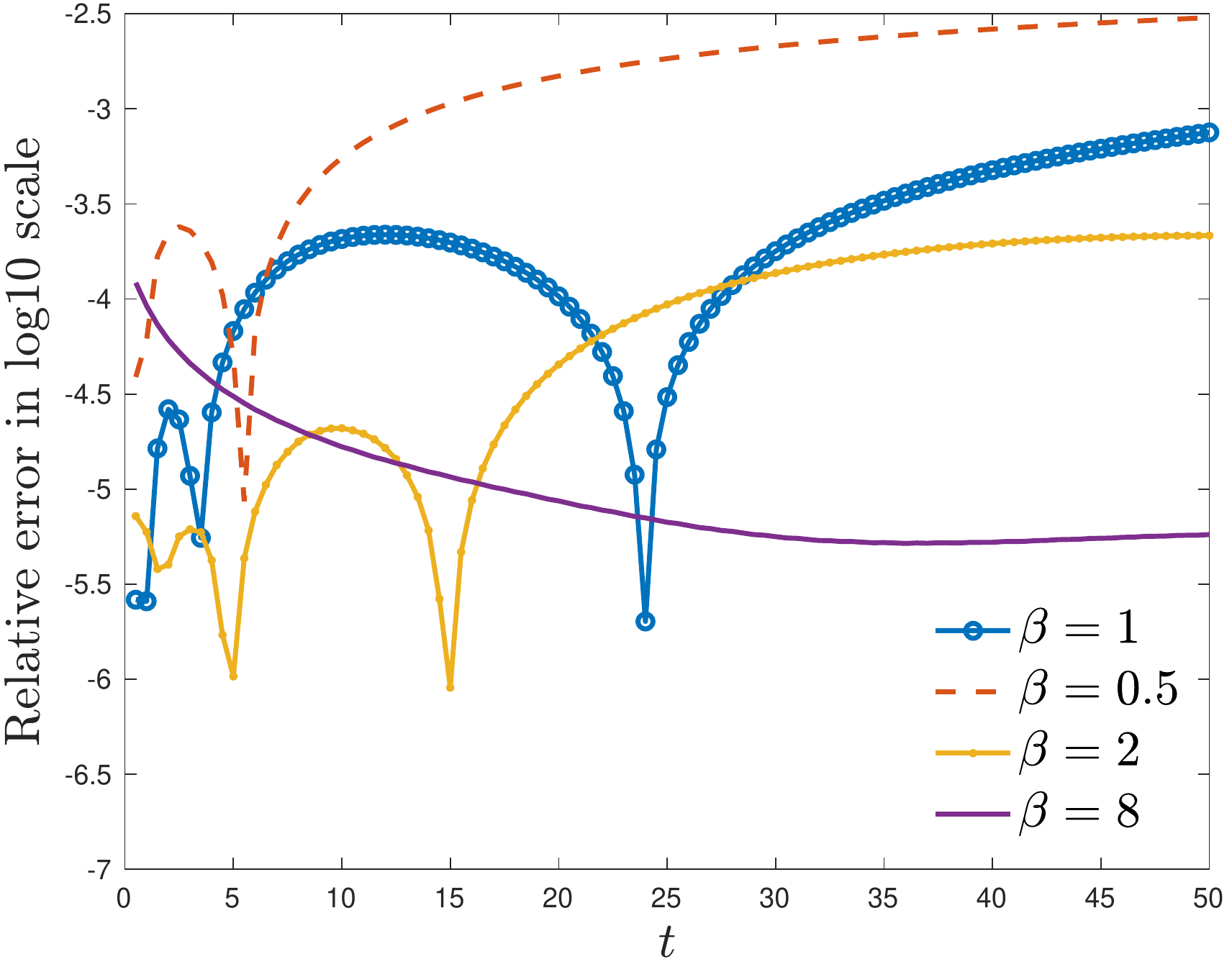}}
\subfigure[$M+1 = 32$, $\gamma= 0.001$.]{\includegraphics[width=0.49\textwidth,height=0.27\textwidth]{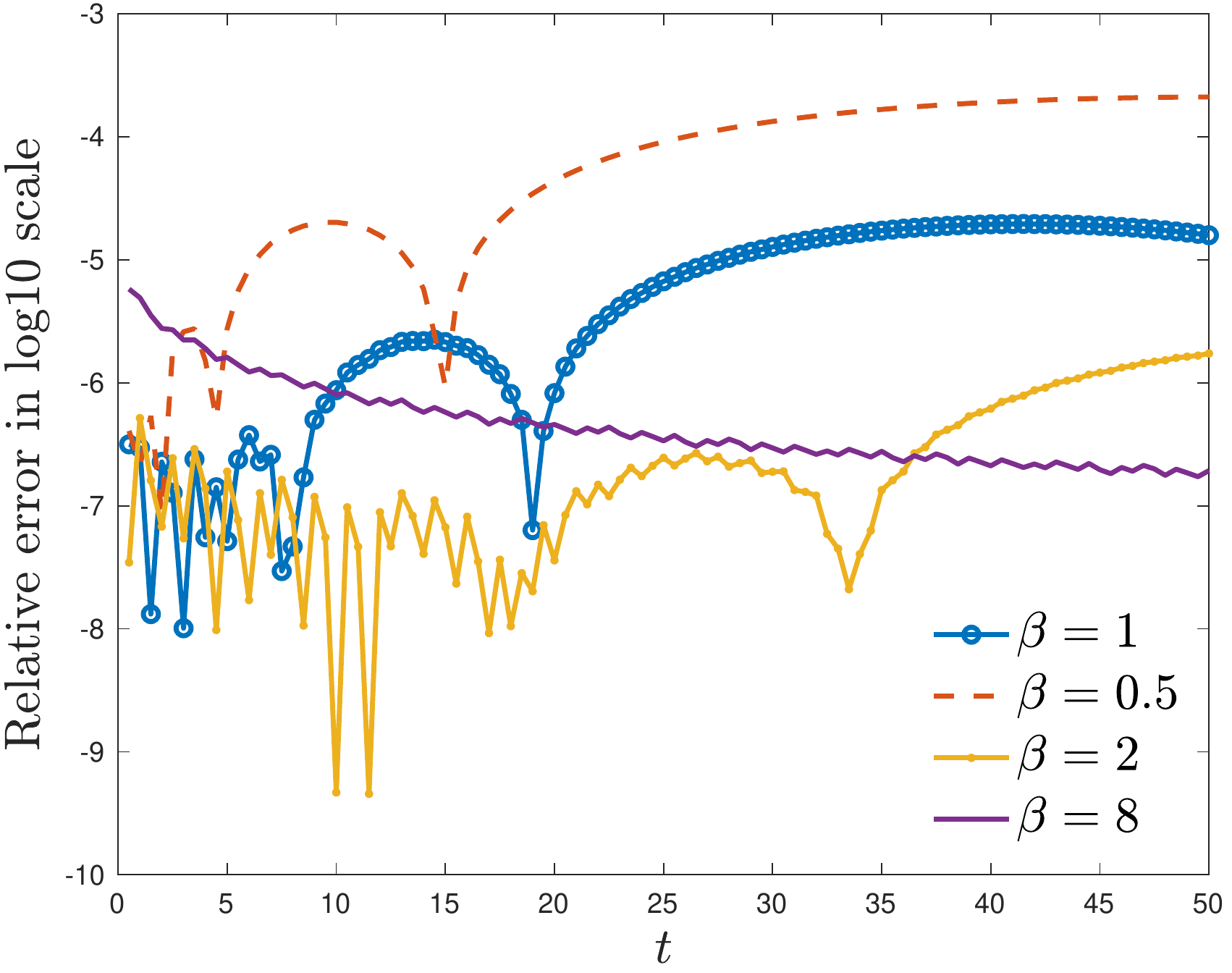}}
\\
\caption{Numerical accuracy of the Laguerre-Gauss quadrature for the integral $I(t)$ in \eqref{dynamics_example}. Numerical  errors will be amplified in time evolution, regardless of the choice of $M$. Choosing a scaling factor $\beta > 1$ will suppress the increments, while $\beta < 1$ will lead to a more severe amplification of numerical errors.}
\label{Laguerre_Gauss_dynamics}
\end{figure}
%%%%%%%%%%%%%%%%%%%%%%%%%%%%%%%%%%%%%%%%%%%%%%

%We will show that with appropriate scaling, the dynamical increments of interpolation errors in real simulations can be alleviated to a large extent. Detailed discussion on numerical performances will be put in Section \ref{sec.numerical}. 

%Similarly, according to Eq.~\eqref{dynamics_response} it yields that
%\begin{equation}
%\frac{\partial }{\partial t} \left(\partial_y \widetilde{\Phi}[\bx, y, t)\right) = \partial_y(\me^{\beta y}\chi(y)) f(\bx, t) - 2y \widetilde{\Phi}[\bx, y, t) - y^2 \partial_y \widetilde{\Phi}[\bx, y, t)
%\end{equation}
%Since $\mathscr{L}_{m-1}^{(\alpha+1, \beta)} = -\beta \mathscr{L}_{m-1}^{(\alpha+1, \beta)}(y) $, it further yields 
%\begin{equation}
%\begin{split}
%-\beta \sum_{m=0}^{M} \frac{\partial }{\partial t}a_m(\bx, t) \left(\mathscr{L}_{m}^{(\alpha+1, \beta)}(y), v\right)_{\omega_{\alpha+1, \beta}} = &\left( \partial_y(\me^{\beta y}\chi(y)), v\right)_{\omega_{\alpha+1, \beta}} \\
%&- 2\sum_{m=0}^{M+1} \left(y \mathscr{L}_{m}^{(\alpha, \beta)}, v \right)_{\omega_{\alpha+1, \beta}} - \sum_{m=0}^{M+1} \left(y \mathscr{L}_{m}^{(\alpha, \beta)}, v \right)_{\omega_{\alpha+1, \beta}} \\
%\end{split}
%\end{equation}

\subsection{A splitting scheme}

Finally, we illustrate the setting of our SMOS scheme. Although the operator $-y^2 \mathcal{I}_d$ breaks the Hamiltonian structure,  the constant-Q wave equation can still be evolved efficiently by an appropriate operator splitting.

Eq.~\eqref{attenuated_wave_eqn_set} are equivalent to 
\begin{equation}
\frac{\partial }{\partial t} \begin{pmatrix} \bv \\ \Phi[\bvarepsilon] \end{pmatrix} = \mathcal{L}\begin{pmatrix} \bv \\ \Phi[\bvarepsilon] \end{pmatrix} + \mathcal{U}\begin{pmatrix} \bv \\ \Phi[\bvarepsilon] \end{pmatrix},
\end{equation}
where the operators $\mathcal{L}$ and $\mathcal{U}$ are given by
\begin{equation}\label{problem_1}
\mathcal{L}  \begin{pmatrix} \bv \\ \Phi[\bvarepsilon] \end{pmatrix} = \begin{pmatrix} 0 & 0 \\ \mathcal{L}_1 & -y^2\mathcal{I}_d \end{pmatrix}  \begin{pmatrix} \bv \\ \Phi[\bvarepsilon] \end{pmatrix} 
\end{equation}
and
\begin{equation}\label{problem_2}
\mathcal{U}  \begin{pmatrix} \bv \\ \Phi[\bvarepsilon] \end{pmatrix} =  \begin{pmatrix} 0 & \mathcal{L}_3 \circ \mathcal{L}_2 \\ 0 & 0 \end{pmatrix} \begin{pmatrix} \bv \\ \Phi[\bvarepsilon] \end{pmatrix}  + \begin{pmatrix} \bm{f} \\ 0 \end{pmatrix}.
\end{equation}

Suppose the time interval $[0, T]$ is divided into $N$ steps, with spacing $\Delta t = T/N$, $t_n = n \Delta t$. Now we can use the Strang splitting 
\begin{equation}
\exp(t(\mathcal{L} + \mathcal{U}) = \left[\exp(\frac{\Delta t}{2} \mathcal{U})  \circ \exp(\Delta t \mathcal{L}) \circ \exp(\frac{\Delta t }{2}\mathcal{U})\right]^n + \mathcal{O}((\Delta t)^3)
\end{equation}

The splitting procedure is rather appealing as the exact solution of each subproblem can be readily obtained.  For the first subproblem \eqref{problem_1}, the exact solution reads that 
\begin{equation}\label{sub1_solution}
\left\{
\begin{split}
&\bv(\bx, t)  = \bv(\bx, t_{n}), \\
&\Phi[\bvarepsilon](\bx, y, t) = \me^{-y^2(t -t_n)} \Phi[\bvarepsilon](\bx, y, t_n) + \frac{1}{y^2} ( 1 -\me^{-y^2(t - t_n)} ) \mathcal{L}_1 \bv(\bx, t_n).
\end{split}
\right.
\end{equation}
This avoids the numerical stiffness  caused by sufficiently large $y^2$ and additional numerical errors in solving the auxiliary relaxed dynamics \eqref{dynamics_response}.

For the second subproblem \eqref{problem_2}, the exact solution reads that
\begin{equation}\label{sub2_solution}
\left\{
\begin{split}
&\bv(\bx, t)  = \bv(\bx, t_{n}) + (t- t_n) \mathcal{L}_3 \circ \mathcal{L}_2 \Phi[\bvarepsilon](\bx, y, t_n)  + \int_{t_n}^t \bm{f}(\bx, s) \D s,\\
&\Phi[\bvarepsilon](\bx, y, t) = \Phi[\bvarepsilon](\bx, y, t_n). 
\end{split}
\right.
\end{equation}

The SMOS scheme for unidimensional constant-Q wave equation is given in Algorithm \ref{leap_frog}, and its generalization to higher-dimensional case is straightforward. Here the strain tensor is added for physical meaning and a more concise description of the algorithm. Compared with the Strang splitting scheme for elastic wave equation, the only difference lies in the second step and there is no additional cost in evaluating the gradients of velocity vectors and stress tensors.

\begin{algorithm}[!h] 
\caption{Short-memory operator splitting (SMOS) scheme\label{leap_frog}}

\vspace{2mm}
\textbf{Input:} The order $M$ of Laguerre-Gauss quadrature, the scaling factor $\beta$

\begin{itemize}

\item[1.] Half-step update of velocity (vector):
\begin{equation*}
\bv(\bx, t_{n+\frac{1}{2}}) = \bv(\bx, t_{n}) + \frac{\Delta t}{2} \mathcal{L}_3  \sigma(\bx, t_n) +  \int_{t_n}^{t_n +\frac{\Delta t}{2}} \bm{f}(\bx, s) \D s.  
\end{equation*}

\item[2.] Full-step update of stress (tensor), with $\tilde{y}_j$ short for $y_j^{(4\gamma-1, \beta)}$:
\begin{equation*}
\left\{
\begin{split}
&\bm{\sigma}(\bx, t_{n+1}) = \rho(\bx)  \frac{2C\sin(2\pi \gamma)}{\pi}  \sum_{j=0}^{M} \omega_j^{(4\gamma-1, \beta)} \me^{\beta \tilde{y}_j}  \Phi[\bvarepsilon](\bx, \tilde{y}_j, t), \\
&\Phi[\bvarepsilon](\bx, \tilde{y}_j , t_{n+1}) = \me^{-\tilde{y}_j^2 \Delta t} \Phi[\bvarepsilon](\bx, \tilde{y}_j , t_{n}) + \frac{1}{\tilde{y}_j^2}(1 - \me^{-\tilde{y}_j^2 \Delta t}) \mathcal{L}_1 \bv(\bx, t_{n+\frac{1}{2}}).
\end{split}
\right. 
\end{equation*}

\item[3.] Half-step update of velocity (vector):
\begin{equation*}
\bv(\bx, t_{n+1}) = \bv(\bx, t_{n+\frac{1}{2}}) + \frac{\Delta t}{2} \mathcal{L}_3  \sigma(\bx, t_{n+1}) +  \int_{t_{n+\frac{1}{2}}}^{t_{n+\frac{1}{2}} +\frac{\Delta t}{2}}\bm{f}(\bx, s) \D s.  
\end{equation*}

\end{itemize}
\end{algorithm}

The gradient operators $\mathcal{L}_1$ and $\mathcal{L}_3$ can be solved by the standard staggered-grid pseudo-spectral method \cite{Carcione2010,bk:Fornberg1998}, and the integration of source term can be either tackled by the Gauss quadrature, or simply by the mid-point quadrature
\begin{equation}
\int_{t_{n+\frac{1}{2}}}^{t_{n+\frac{1}{2}} +\frac{\Delta t}{2}} \bm{f}(\bx, s) \D s \approx \frac{\Delta t}{2} \left( \bm{f}(\bx, t_{n+\frac{1}{2}}) +  \bm{f}(\bx, t_{n+\frac{1}{2}} + \frac{\Delta t}{2}) \right).
\end{equation}

For two-dimensional case, a trick is adopted to reduce memory variables. Suppose $\gamma_P \le \gamma_S$ so that $y^{4\gamma_S -4\gamma_P}$ is not singular, then we have
\begin{equation*}
 \int_0^{+\infty} y^{4\gamma_S-1} \Phi[\bvarepsilon](x, z, y, t) \D y   =  \int_0^{+\infty}  y^{4\gamma_P-1} y^{4\gamma_S -4\gamma_P} \Phi[\bvarepsilon](x, z, y, t) \D y,
\end{equation*}
and the Laguerre functions can be chosen as $\mathscr{L}_{m
}^{(4\gamma_P-1, \beta)}(y)$, instead of $\mathscr{L}_{m}^{(4\gamma_S-1, \beta)}(y)$.

%%%%%%%%%%%%%%%%%%%%

\section{Numerical experiments}
\label{sec.numerical}

In this section, numerical experiments on 1-D diffusive wave equation and 2-D constant-Q wave equation are performed to give a thorough benchmark on the SMOS scheme. Our concern includes an investigation of  the convergence with respect to the time step $\Delta t$, number of spatial collocation points $N_x$ and the memory length $M$, as well as the scaling factor $\beta$. Afterward we present the benchmarks on 2-D constant-Q viscoelastic wave equation and give a detailed study of convergence with respect to the memory length $M_P = M_S = M$ and the effect of scaling.

For spatial discretization, we adopt the staggered-grid pseudo-spectral method in discretizing the spatial derivatives, which is suggested to improve both stability and accuracy in applications of fluid dynamics or elasticity \cite{bk:Fornberg1998}. Suppose the finite computational domain is $[x_{\min}, x_{\max}]$ in each dimension, then the staggered grid mesh is $x_{\min} = x_0 < x_{1/2} < x_1 < x_{3/2} < \dots < x_{N-1} < x_{N-1/2} < x_{N} = x_{\max}$, where $x_j = x_{\min} + j \Delta x$, $x_{j+1/2} = x_{\min} + (j+1/2) \Delta x$ with the spacing $\Delta x = \frac{x_{\max} - x_{\min}}{N}$. For 1-D problem, the grid mesh for strain $\sigma$ and velocity $v$ are $\{x_j\}_{j=0}^{N-1}$ and  $\{x_{j+1/2}\}_{j=0}^{N-1}$, respectively, while that in the 2-D problem can be found in \cite{bk:Fornberg1998}. In order to preclude the errors induced by the artificial boundary condition \cite{Zhu2017,DuHanZhangZheng2018}, we simply enlarge  the computational domain to avoid the artificial wave reflection.

To measure the numerical errors, we adopt the relative $L^\infty$-error $\mathcal{E}_\infty$ as the performance metric. 
\begin{equation}
\mathcal{E}_{\infty}[{\varphi}](t) = \frac{\max_{\bx \in \mathcal{X}}  |\varphi_{\textup{num}}(\bx, t) - {\varphi}_{\textup{ref}}(\bx, t)|}{\max_{\bx \in \mathcal{X}}  |{\varphi}_{\textup{ref}}(\bx, t)|},
\end{equation}
where $\varphi$ is a scalar quantity. For 1-D problem, $\varphi(x, t) = v(x, t)$ and for 2-D problem, $\varphi(\bx, t) = u_3(\bx, t)$ or $\sigma_{13}(\bx, t)$.

The subroutines $\textup{GEN\_LAGUERRE\_RULE}$ \cite{Burkardt2010} is used to obtain nodes and weights for the generalized Laguerre-Gauss quadrature. All the simulations performed via MATLAB (1-D) or Fortran (2-D) implementations run on the platform:  AMD Threadripper 1920X (3.50GHz, 32MB Cache, 12 Cores, 24 Threads) with 128GB Memory. The parallelization for 2-D problem is realized by the OpenMP library using up to 24 threads.

\subsection{1-D diffusive wave equation without source term}

The first benchmark is to solve the diffusive wave equation \eqref{diffusive_wave_eqn} with $\rho \equiv 1$ and $C=1$. 
%Suppose the initial condition is $\varepsilon(x, 0) = \me^{-|x|}$, then its exact solution reads that 
%{\cf
%\begin{equation}
%\varepsilon(x, t) = E_{2-2\gamma}(-t^\alpha) \me^{-|x|},
%\end{equation} 
%where $E_\alpha(z) =  \sum_{k=0}^\infty \frac{z^n}{ \Gamma(\alpha k + 1)}$ is the one-parametric Mittag-Leffler function. Here we utilize the fact that the Mittag-Leffler function is the eigenfunction of the Caputo fractional differential operator
%\begin{equation}
%{_C}D_t^\alpha E_{\alpha}(-t^{\alpha}) = - E_{\alpha}(t^{\alpha}).
%\end{equation}
%\begin{equation}
%\frac{\partial^2}{\partial t^2} \varepsilon(x, t) = C \frac{\partial^2}{\partial x^2} \left({_C}D_t^{2\gamma} \varepsilon(x, t)\right).
%\end{equation}
The initial condition is set as
\begin{align}
v(x, t) |_{t = 0^+} &= \me^{-x^2}, \quad \frac{\partial}{\partial t}v(x, t) |_{t = 0^+} = 0, \quad &-\infty < x < \infty,\\
\sigma(x, t) |_{t = 0^+} &= {_C}D_t^{2\gamma} \varepsilon(x, t)|_{t = 0^+} =  0,  \quad &-\infty < x < \infty, \\
\Phi(x, y, t) |_{t = 0^+} & = \Gamma(1 - 2\gamma) \me^{-y^2} \sigma(x, t)|_{t = 0^+} = 0, \quad & -\infty < x, y < \infty,
\end{align}
and $v(\pm \infty, t) = 0$. The initial strain is derived by $\frac{\partial }{\partial x} \varepsilon(x, t)  |_{t = 0^+} = \frac{\partial }{\partial t} v(x, t) |_{t = 0^+}$. 

In order to provide a reasonable reference for small $\gamma$, we utilize the intriguing spectral approximation to the exact solution of \eqref{exact_solution_1d_diffusive}. Let $s = \sqrt{C} t^{1 - \gamma}$, then
\begin{equation*}
\begin{split}
\varepsilon(x, t) &= \frac{1}{2} \int_0^{+\infty}  M_{1- \gamma}(y) \left[ \me^{-(x - sy)^2 } + \me^{(x+sy)^2}\right]  \D y \\
&\approx
\left\{
\begin{split}
&\frac{1}{2} \sum_{j=0}^{M_y} \omega_j^{(0, \frac{1}{s})}  \me^{y_j^{(0, 1)}} M_{1-\gamma}(y_j^{(0, \frac{1}{s})})\left[ \me^{-(x - s y_j^{(0, \frac{1}{s})})^2 } + \me^{(x+ s y_j^{(0, \frac{1}{s})})^2}\right], ~~  t< 1,\\
&\frac{1}{2s} \sum_{j=0}^{M_y} \omega_j^H M_{1-\gamma}\left(\frac{x  - y_j^H}{s}\right), ~~ t \ge 1,
\end{split}
\right.
\end{split}
\end{equation*}
%\begin{equation}
%\begin{split}
%\varepsilon(x, t) & = \frac{1}{2s} \int_{-\infty}^{+\infty} M_{1-\gamma} \left(\frac{x-y}{s}\right) \me^{-y^2} \D y, \\
%& = \frac{1}{2} \int_0^{+\infty}  \me^{-y}M_{1- \gamma}(y) \left[ \me^{-(x - sy)^2 } + \me^{(x+sy)^2}\right] \me^{y} \D y
%\end{split}
%\end{equation}
 which can be approximated by either the Laguerre-Gauss quadrature or the Hermite-Gauss quadrature, with
%\begin{equation}
%\varepsilon(x, t) \approx
%\left\{
%\begin{split}
%&\frac{1}{2} \sum_{j=0}^{M_y} \omega_j^L  \me^{y_j^L} M_{1-\gamma}(y_j^L)\left[ \me^{-(x - s y_j^L)^2 } + \me^{(x+ s y_j^L)^2}\right], \quad&t< 1, \\
%&\frac{1}{2s} \sum_{j=0}^{M_y} \omega_j^H M_{1-\gamma}\left(\frac{x  - y_j^H}{s}\right), \quad &t \ge 1.
%\end{split}
%\right.
%\end{equation}
$y_j^H$ and $\omega_j^H$ the collocations points and weights of the Hermite-Gauss quadrature, respectively. The calculation of the Mainardi function is entirely not trivial. We try to realize it by combining the generalized Laguerre-Gauss quadrature and its asymptotic expansion, and details are put in \ref{sec.Mainardi_function}.

%\begin{equation}
%\begin{split}
%\varepsilon(x, t) & = \frac{1}{2s} \int_{-\infty}^{+\infty} M_{1-\gamma} \left(\frac{x-y}{s}\right) \me^{-y^2} \D y, \\
%& = \frac{1}{2} \int_0^{+\infty}  \me^{-y}M_{1- \gamma}(y) \left[ \me^{-(x - sy)^2 } + \me^{(x+sy)^2}\right] \me^{y} \D y
%\end{split}
%\end{equation}

%\begin{equation}
%\varepsilon(x, t) \approx
%\left\{
%\begin{split}
%&\frac{1}{2} \sum_{j=0}^{M_y} \omega_j^L  \me^{y_j^L} M_{1-\gamma}(y_j^L)\left[ \me^{-(x - s y_j^L)^2 } + \me^{(x+ s y_j^L)^2}\right], \quad&t< 1, \\
%&\frac{1}{2s} \sum_{j=0}^{M_y} \omega_j^H M_{1-\gamma}\left(\frac{x  - y_j^H}{s}\right), \quad &t \ge 1.
%\end{split}
%\right.
%\end{equation}

The computational domain is $\mathcal{X} = [x_{\min}, x_{\max}] = [-15, 15]$. The final time is $T = 8$. Other parameters include: the fractional order $\gamma = 0.1, 0.05, 0.01$, the time step $\Delta t = 2^{-4}, 2^{-5}, 2^{-6}, 2^{-7}, 2^{-8}$, the memory length $M+1 = 4, 8, 16, 32, 64$ and $N_x = 8, 16, 32, 64, 128$. In order to evaluate the performance of SMOS, we first investigate the convergence with  $\Delta t$ (with $N_x = 128$, $M+1 = 256$ fixed), $N_x$ (with $\Delta t = 10^{-4}$, $M+1 = 256$ fixed) as presented in  Figure \ref{fig_1d_convergence}.

%The numerical accuracy of the Laguerre-Gauss quadrature is very subtle and requires a careful investigation. To this end,  we need to exploit the stress-strain relation \eqref{1d_stress_strain} under the Laplace transform. The exact value of the Laplace transform of the strain $\widetilde{\varepsilon}(x, s)$ reads that
%\begin{equation}\label{Laplace_stress}
%\widetilde{ \varepsilon} (x, s) =  \frac{s^{(1-\gamma) - 1} }{2\sqrt{C} }  \int_{0}^{+\infty} \me^{-\frac{s^{1-\gamma}}{\sqrt{C}} \xi }  \left(\me^{-\frac{1}{2} (x - \xi)^2} + \me^{-\frac{1}{2} (x + \xi)^2} \right) \D \xi,
%\end{equation}
%which can be evaluated very accurately by the Laguerre-Gauss quadrature. In addition, Eq.~\eqref{Laplace_stress} can be calculated by the compound Simpson formula, 
%\begin{equation}\label{compound_Simpson_formula}
%\widetilde{\sigma}_{\textup{ref}}(x, s)  \approx \frac{\Delta t}{3} [ \sigma(x, 0) + 4\sum_{k=0}^{\frac{N}{2}-1} \sigma(x, t_{2k+1}) \me^{-s t_{2k+1}}+  2\sum_{k=1}^{\frac{N}{2}-1} \sigma(x, t_{2k}) \me^{-s t_{2k}} + \sigma(x, t_N) \me^{-s t_N}  ].
%\end{equation}

A visualization of the velocity propagation under different memory lengths (with $\Delta t = 10^{-4}$, $N_x = 128$, $\beta =1$) is presented in Figure \ref{fig_1d_velocity_evo}. The velocity splits into two branches, each of which propagates in opposite directions and reaches the centers $x = \pm 8$ at $T= 8$, respectively. In contrast to the elastic case, the attenuation of wave front can  be observed as the height of waveform decreases, and the attenuation level becomes more evident for larger $\gamma$.

The main goal is to investigate how the scaling factor $\beta$ influences the accuracy under the parameters: $\Delta t = 10^{-4}$, $N_x= 128$ and fixed $M$. The maximal errors and averaged computational time are recorded in Table 
\ref{1d_scaling}. The convergence with respect to $M$ is presented in Figure \ref{fig_1d_memory_convergence}. From the results, we can make the following observations.

%%%%%%%%%%%%%%%%%%%%%%%%%%%%%%%%%%%%%%%%%%%%%%
\begin{figure}[!h]
\centering
\subfigure[Convergence w.r.t $\Delta t$.]{{\includegraphics[width=0.49\textwidth,height=0.27\textwidth]{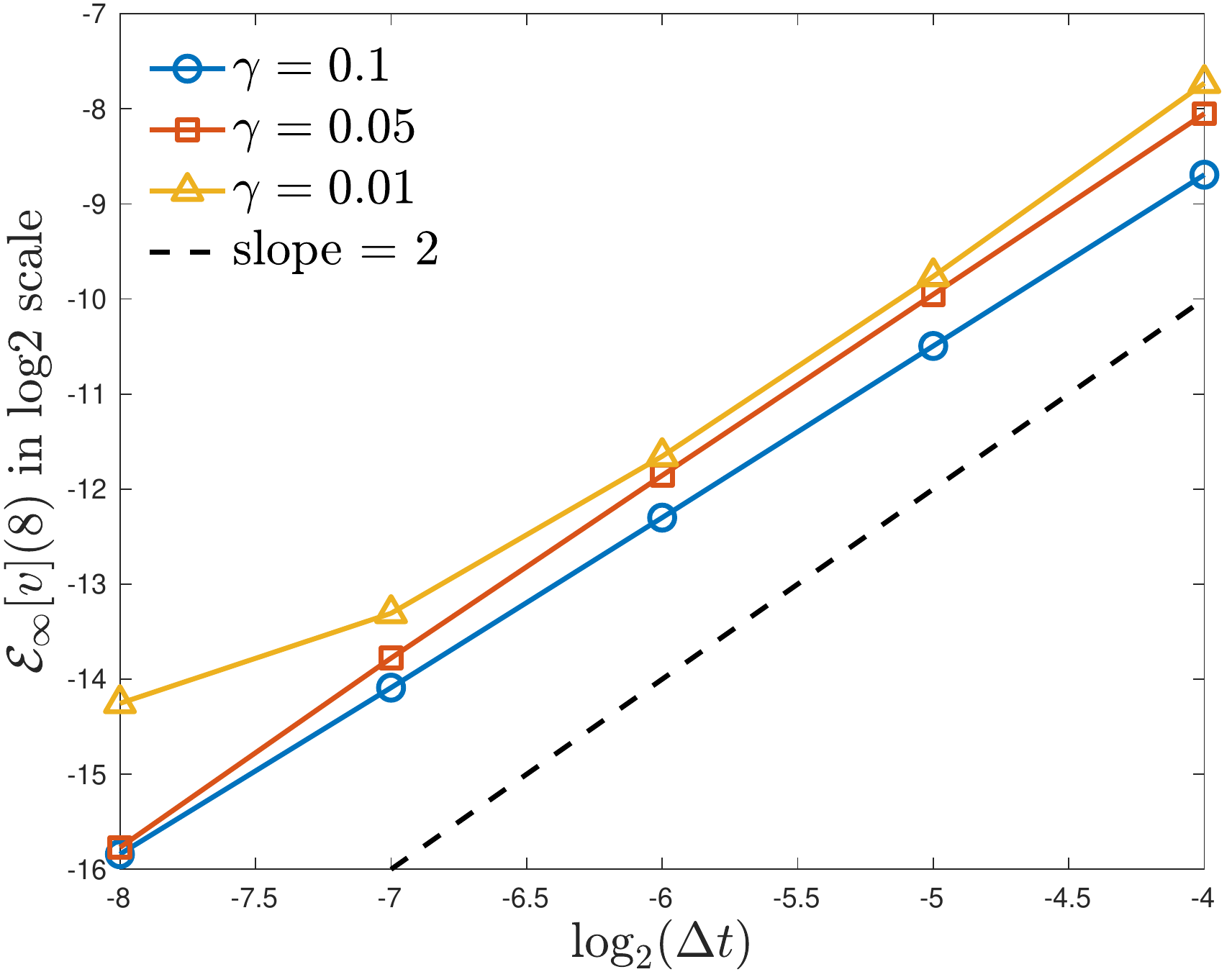}}}
\subfigure[Convergence w.r.t. $N_x$.]{{\includegraphics[width=0.49\textwidth,height=0.27\textwidth]{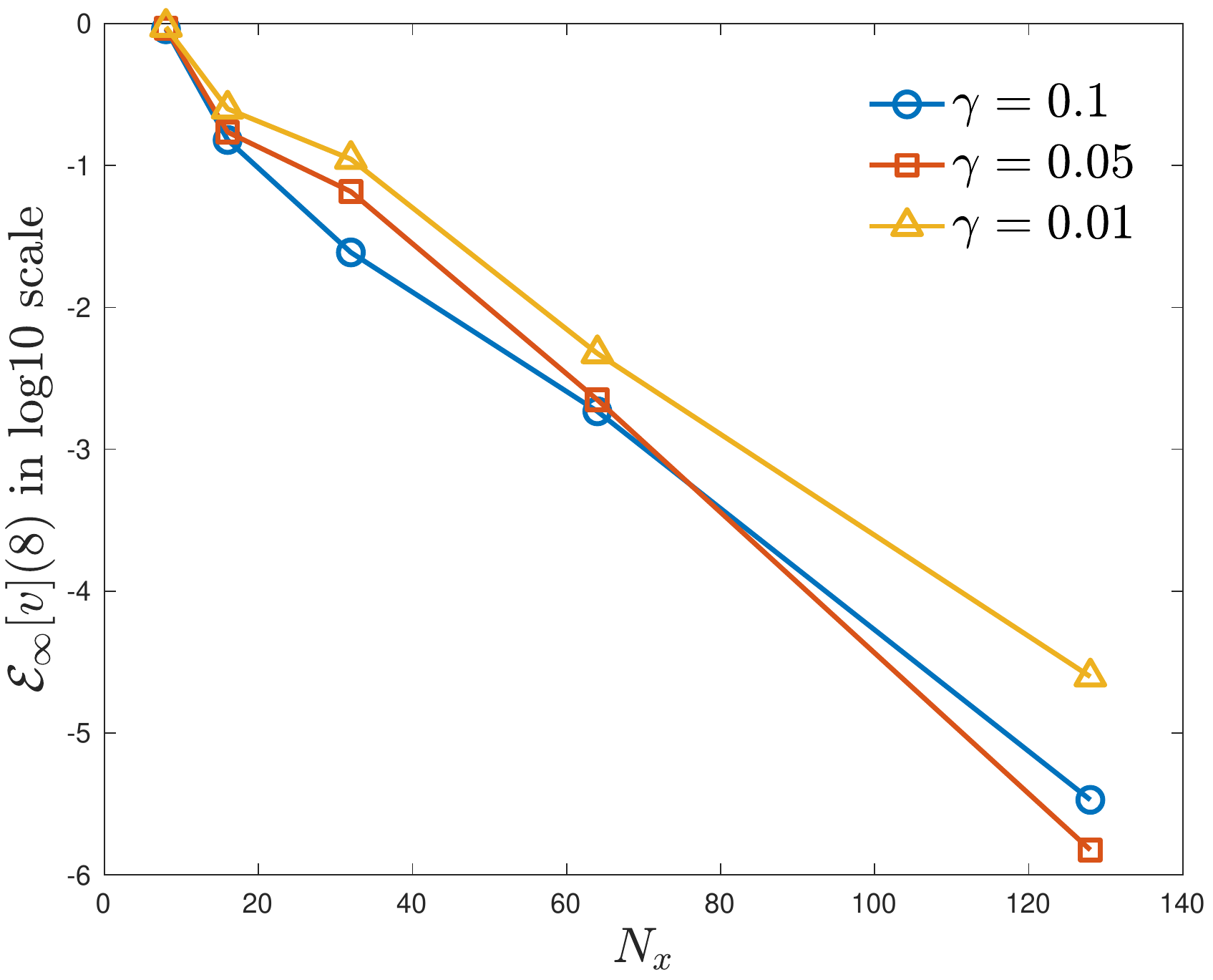}}}
\caption{1-D diffusive wave propagation: The convergence with respect to time step $\Delta t$ and number of spatial collocation points $N_x$. For different $\gamma$, the second-order convergence with $\Delta t$ is always observed, which coincides with the theoretical value in the Strang splitting. Meanwhile, the spectral convergence is achieved in the spatial direction.}
\label{fig_1d_convergence}
\end{figure}
%%%%%%%%%%%%%%%%%%%%%%%%%%%%%%%%%%%%%%%%%%%%%%

%%%%%%%%%%%%%%%%%%%%%%%%%%%%%%%%%%%%%%%%%%%%%%
\begin{figure}[!h]
\centering
\subfigure[$M+1 = 4$.]{{\includegraphics[width=0.49\textwidth,height=0.27\textwidth]{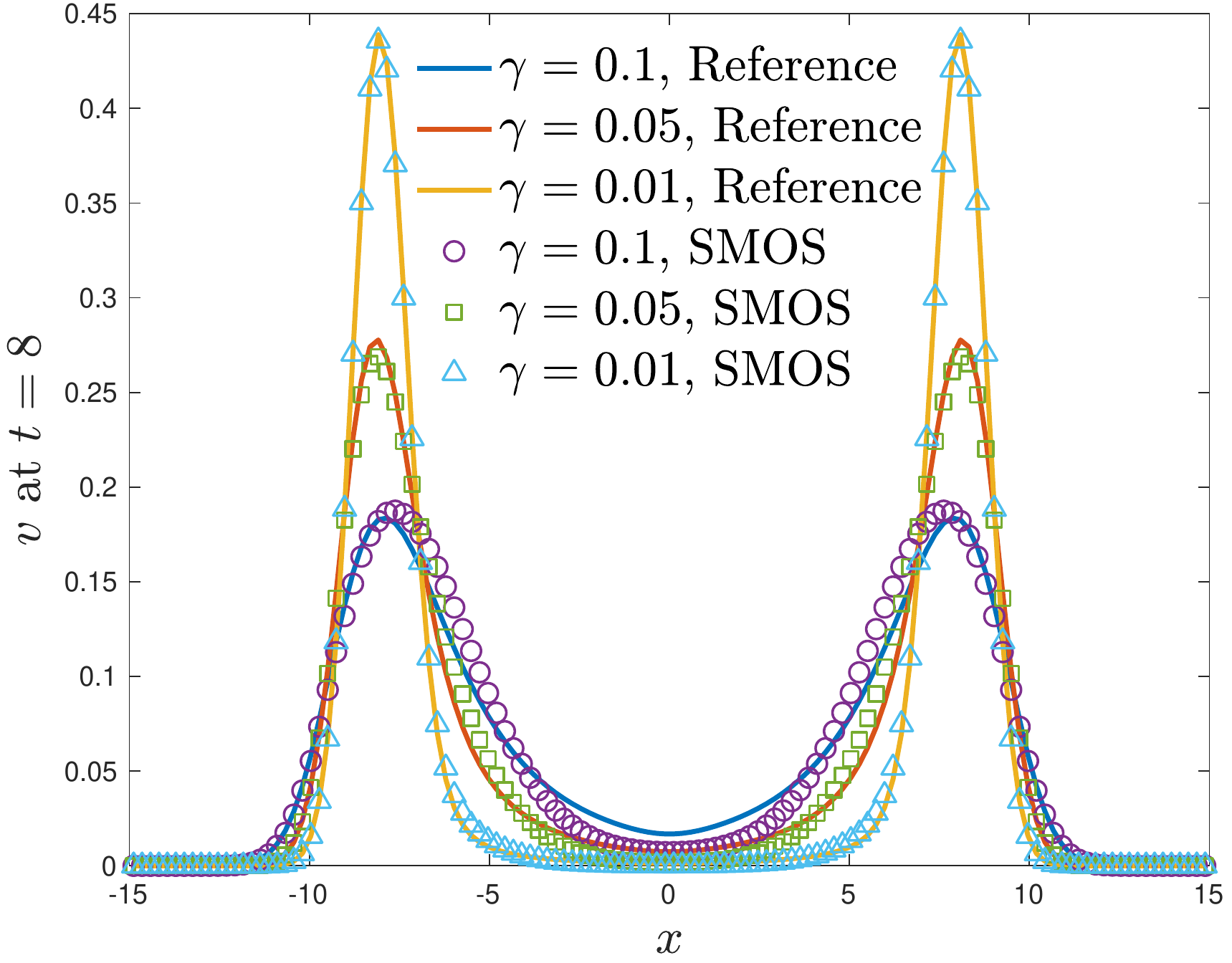}}}
\subfigure[$M+1 = 8$.]{{\includegraphics[width=0.49\textwidth,height=0.27\textwidth]{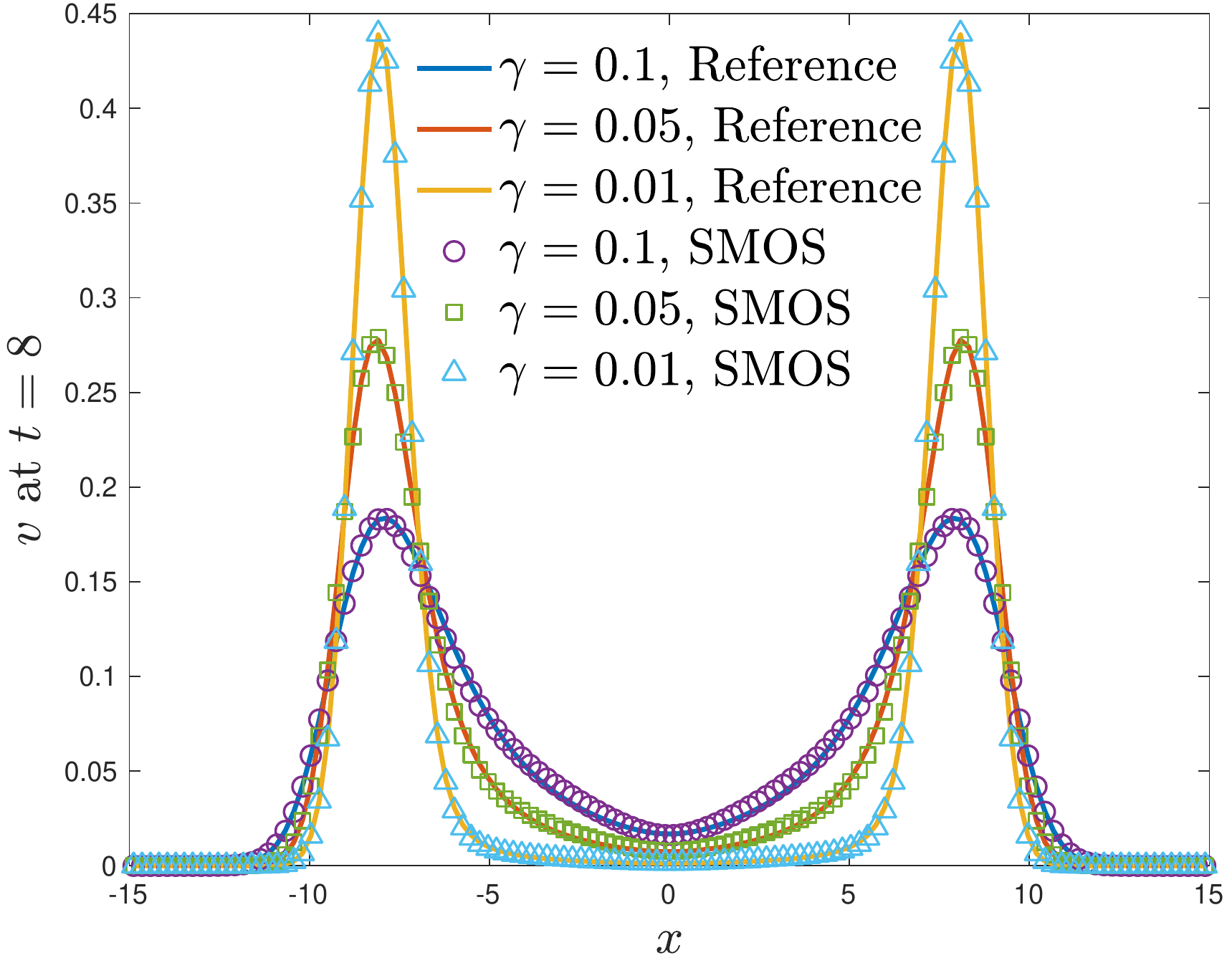}}}
\\
\centering
\subfigure[$M+1 = 16$.]{{\includegraphics[width=0.49\textwidth,height=0.27\textwidth]{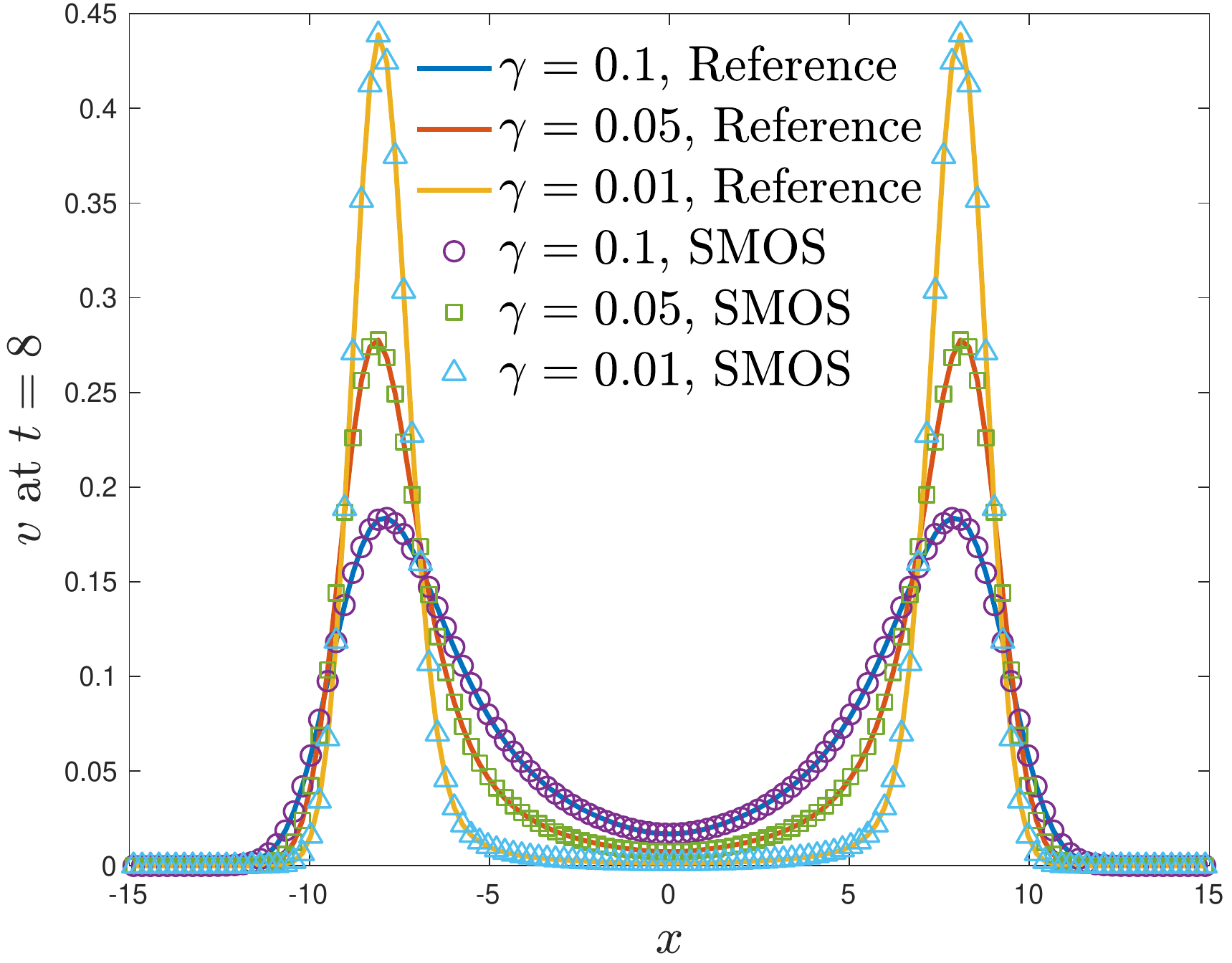}}}
\subfigure[$M+1 = 256$.]{{\includegraphics[width=0.49\textwidth,height=0.27\textwidth]{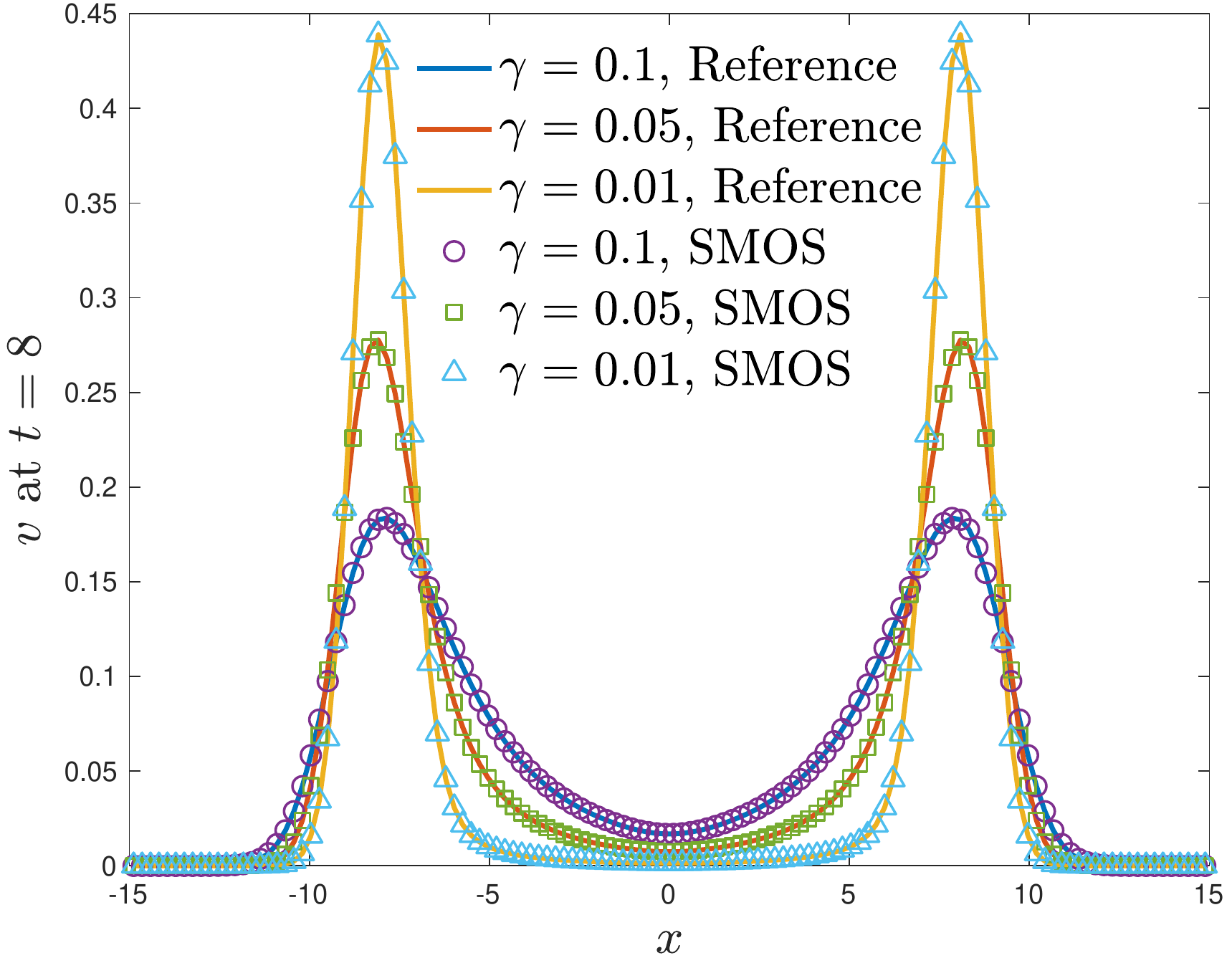}}}
\caption{1-D diffusive wave propagation: The propagation of velocity at $T= 8$, under different memory $M+1$. The velocity splits into two parts and the attenuation of waveform in the propagation is clearly observed. When $M$ is small, numerical solutions slightly deviate the exact ones, but the errors can be evidently suppressed when $M+ 1 =16$. }
\label{fig_1d_velocity_evo}
\end{figure}
\begin{figure}[!h]
\centering
\subfigure[$\beta = 0.5$.]{\includegraphics[width=0.49\textwidth,height=0.27\textwidth]{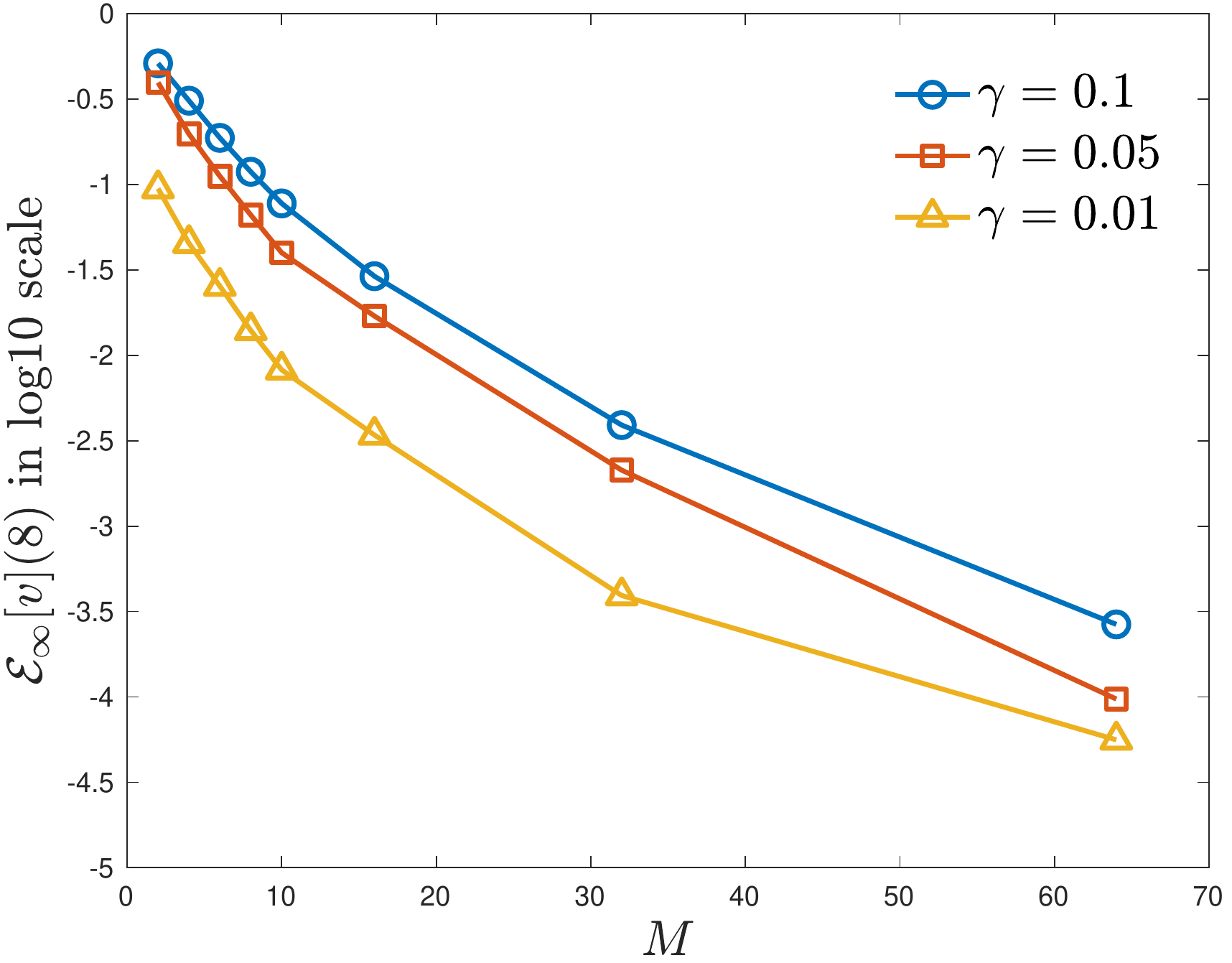}}
\subfigure[$\beta = 1$.]{\includegraphics[width=0.49\textwidth,height=0.27\textwidth]{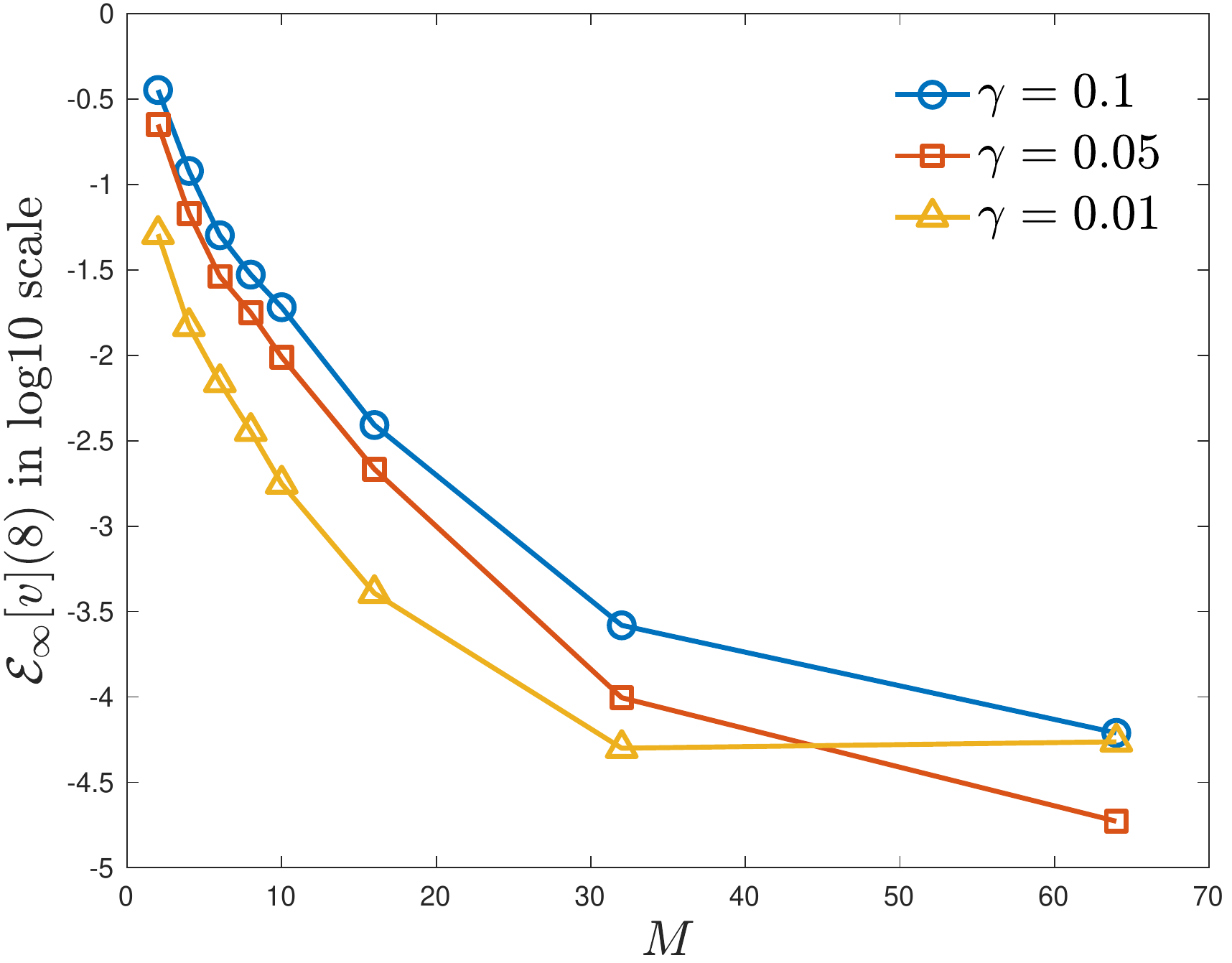}}
\\
\centering
\subfigure[$\beta = 1.5$.]{{\includegraphics[width=0.49\textwidth,height=0.27\textwidth]{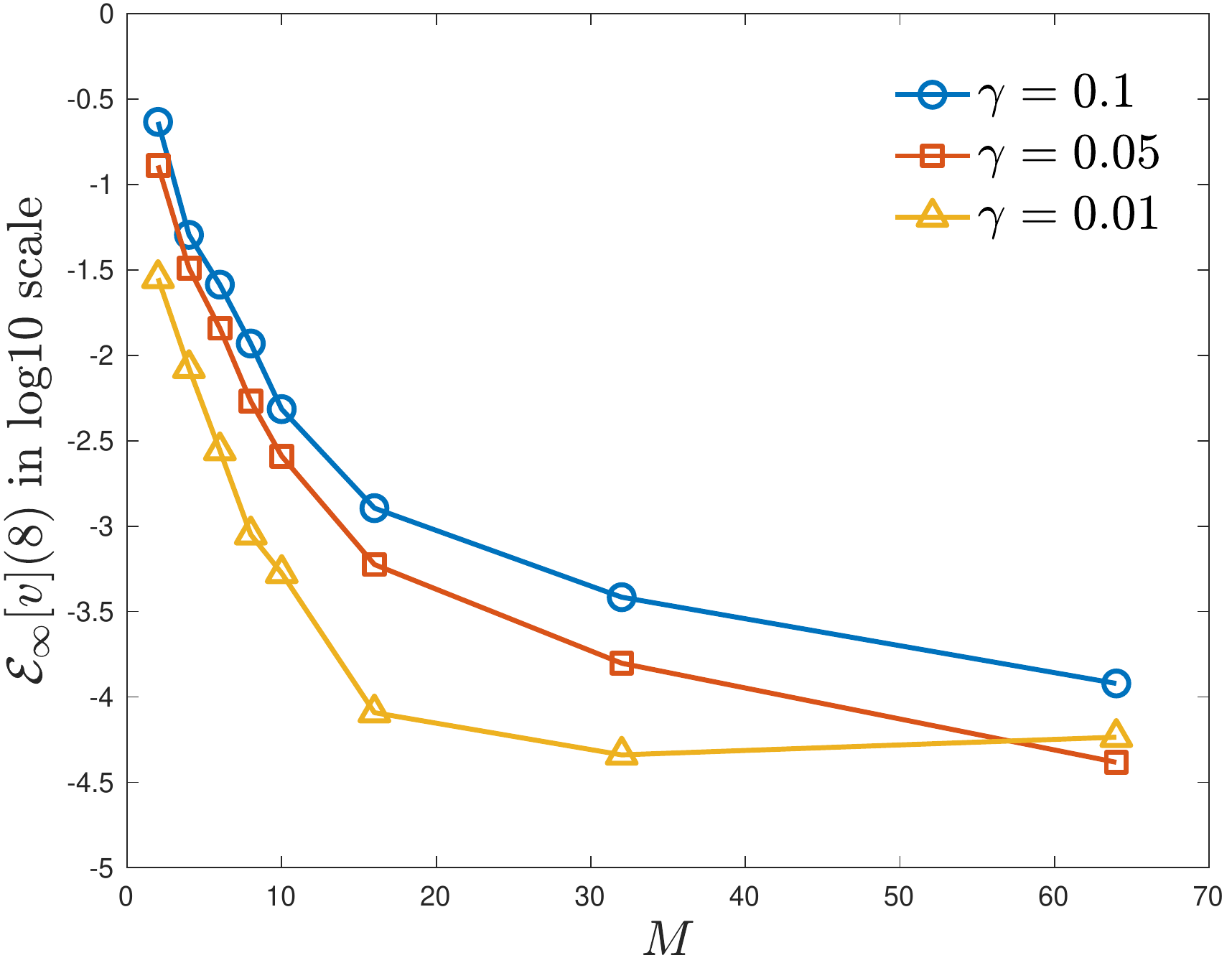}}}
\subfigure[$\beta = 2$.]{{\includegraphics[width=0.49\textwidth,height=0.27\textwidth]{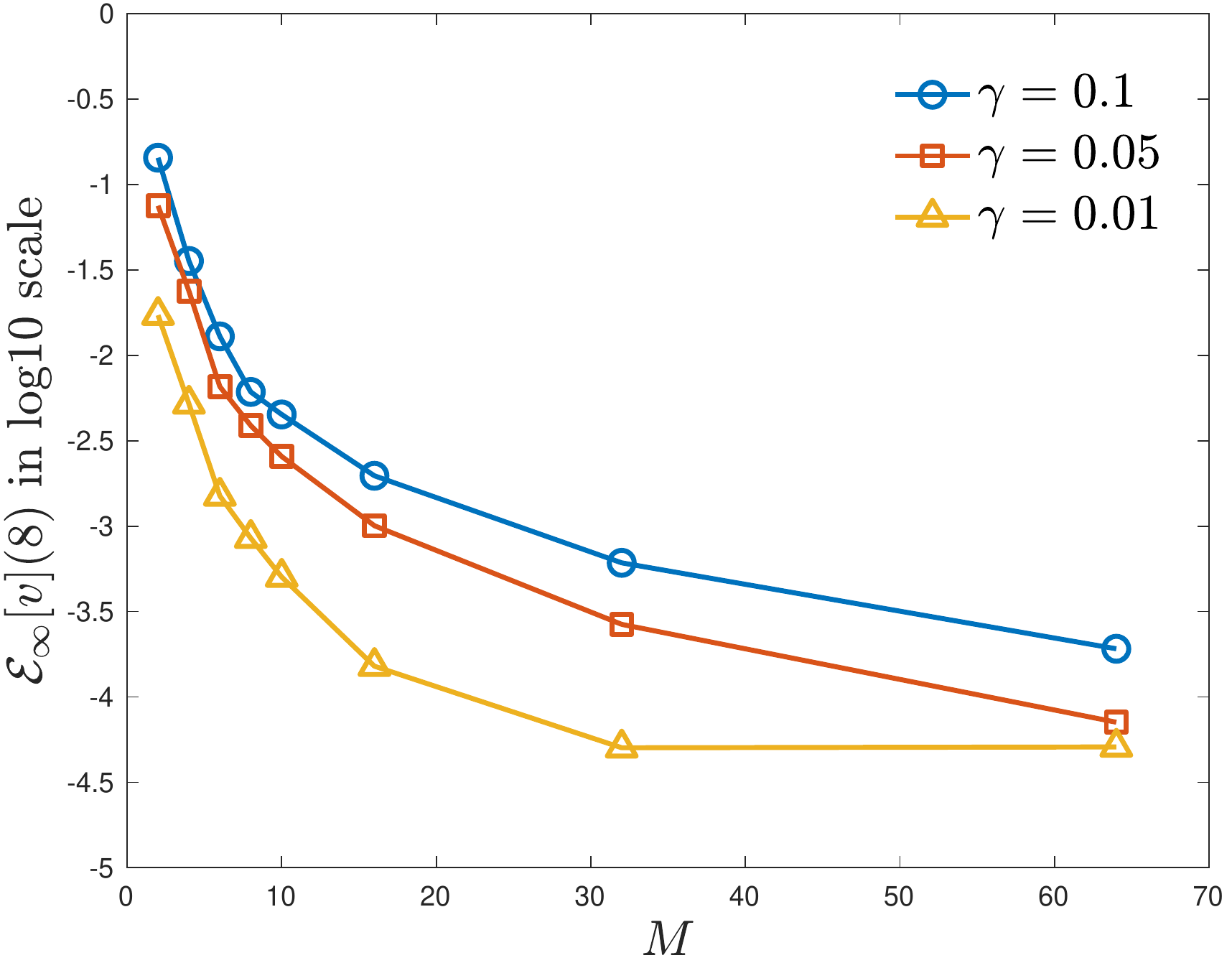}}}
\caption{1-D diffusive wave propagation: The convergence with respect to $M$, under different scaling factors $\beta$. SMOS convergences rapidly when $M+ 1 \le 32$, but the trend slows down for larger $M$. When $\gamma$ is small, only a few memory variables can accurately seize the major contribution.}
\label{fig_1d_memory_convergence}
\end{figure}
%%%%%%%%%%%%%%%%%%%%%%%%%%%%%%%%%%%%%%%%%%%%%%

%%%%%%%%%%%%%%%%%%%%%%%%%%%%%%%%%%%%%%%%%%%%%%
\begin{table}[!h]
  \centering
  \caption{\small 1-D diffusive wave equation: $\mathcal{E}_{\infty}[v]$ at $T= 8$ and the averaged computational time, with $\Delta t = 10^{-4}$ and $N_x = 128$. \label{1d_scaling}
}
 \begin{lrbox}{\tablebox}
 \begin{tabular}{c|c|cccccc|c}
 \hline\hline
 $\gamma$	&\diagbox{$M$}{$\beta$} 	&0.5	&	1	&	1.5	&	2	&	4	&	8	& 	time(s)	\\
 \hline
\multirow{4}{*}{$0.1$}
&	7	&	1.188E-01	&	2.967E-02	&	1.172E-02	&	6.103E-03	&	2.061E-02	&	6.896E-02	&	4.08 	\\
&	15	&	1.399E-02	&	3.908E-03	&	1.274E-03	&	1.971E-03	&	6.042E-03	&	1.882E-02	&	4.27 	\\
&	31	&	3.903E-03	&	2.626E-04	&	3.836E-04	&	6.096E-04	&	1.857E-03	&	5.681E-03	&	4.82 	\\
&	63	&	2.659E-04	&	6.159E-05	&	9.836E-05	&	1.914E-04	&	5.853E-04	&	1.783E-03	&	5.75 	\\
\hline
\multirow{4}{*}{$0.05$}
&	7	&	6.632E-02	&	1.774E-02	&	5.383E-03	&	3.877E-03	&	1.394E-02	&	5.150E-02	&	4.11 	\\
&	15	&	2.905E-02	&	2.157E-03	&	5.959E-04	&	1.005E-03	&	3.522E-03	&	1.246E-02	&	4.29 	\\
&	31	&	2.126E-03	&	9.896E-05	&	1.575E-04	&	2.660E-04	&	9.332E-04	&	3.269E-03	&	4.84 	\\
&	63	&	9.718E-05	&	1.870E-05	&	3.133E-05	&	7.110E-05	&	2.535E-04	&	8.883E-04	&	5.78 	\\
\hline
\multirow{4}{*}{$0.01$}
&	7	&	1.399E-02	&	3.596E-03	&	8.915E-04	&	8.503E-04	&	3.430E-03	&	1.343E-02	&	4.12 	\\
&	15	&	3.418E-03	&	4.047E-04	&	8.091E-05	&	1.513E-04	&	7.348E-04	&	3.008E-03	&	4.28 	\\
&	31	&	3.924E-04	&	5.009E-05	&	4.578E-05	&	5.039E-05	&	1.350E-04	&	6.680E-04	&	4.87 	\\
&	63	&	5.606E-05	&	5.461E-05	&	5.207E-05	&	4.733E-05	&	6.596E-05	&	1.851E-04	&	5.77 	\\
\hline\hline
 \end{tabular}
\end{lrbox}
\scalebox{0.77}{\usebox{\tablebox}}
\end{table}

\begin{itemize}

\item[(1)] The second-order convergence with respect to $\Delta t$ is clearly observed and accords with the theoretical value perfectly.  In addition, the spectral convergence of the staggered pseudo-spectral method is also verified.

\item[(2)] As presented in Figure \ref{fig_1d_memory_convergence}, the spectral convergence for the Laguerre spectral method is clearly observed before $M+1 \le 32$. But after the pre-asymptotic range, the Laguerre spectral approximation converges only algebraically, instead of exponentially. That accounts for the observation that  further increasing the memory length might not improve the accuracy when $\gamma = 0.01$. For sufficiently small $\gamma$ (e.g., $\gamma = 0.01$),  using only a few nodes can produce accurate results, as shown in Figure \ref{fig_1d_velocity_evo}. This provides some evidence to validate our short-memory principle as too many memory variables might not necessarily bring in significant improvement in accuracy.

\item[(3)] The effect of scaling is shown in Figure \ref{fig_1d_memory_convergence} and Table \ref{1d_scaling}. Even without source term, for relatively small $M+1\le 16$, the numerical results under $\beta= 2$ usually outperform those without scaling, whereas choosing $\beta < 1$ diminishes the numerical accuracy. Such result also coincides with our theoretical prediction in Proposition \ref{prop_LG}. In other words, it manifests the importance of choosing an appropriate scaling factor in SMOS as it may indeed significantly enhance the accuracy, without introducing additional numerical cost. However, too large scaling factor is not recommended as it may lead to larger truncation errors.

\end{itemize}

\subsection{2-D Constant-Q P- and S-wave modeling with source term}

For 2-D problem, we consider the wave propagation from a horizontal source function with Ricker-type wavelet history,
\begin{equation}
\begin{split}
f_1(x, z, t) = A(x, z) f_r(t), \quad f_3(x, z, t) = 0,
\end{split}
\end{equation}
where the amplitude function $A(x, z)$ is simply set as a Gaussian profile centered at $(x_0, z_0)$,
\begin{equation}
A(x, z) = \exp(-{(x-x_0)^2}/{2}) \exp(-{(z-z_0)^2}/{2}) 
\end{equation}
and the Ricker wavelet is given by
\begin{equation}
f_r(t) = (1 - 2(\pi f_P (t- d_r)^2) \exp(-(\pi f_P(t- d_r))^2),
\end{equation}
where $f_P$ is the peak frequency and $d_r$ is the temporal delay.

The medium is supposed to be in equilibrium at $t = 0$, namely, stress tensor and velocity vector are set to zero everywhere in the medium \cite{Virieux1986}.
\begin{align}
&v_1(x, z, t) |_{t = 0^+} = v_3(x, z, t) |_{t = 0^+} = 0, \\
&\sigma_{11}(x, z, t) |_{t = 0^+} = \sigma_{33}(x, z, t) |_{t = 0^+} = \sigma_{13}(x, z, t) |_{t = 0^+} = 0, \\
&\Phi_{11}(x, z, y, t) |_{t = 0^+} = \Phi_{33}(x, z, y, t) |_{t = 0^+} = \Phi_{13}(x, z, y, t) |_{t = 0^+} = 0.
\end{align}
Two typical groups of parameters are collected in Table \ref{Parameters}, including wave velocities, density and quality factors for different medium found in \cite{Carcione2009}, and the reference frequency $\omega_0$ is about $2\pi \times 100 $Hz. Thus the constants $C_P$ and $C_S$ can be evaluated by Eq.~\eqref{def.constant}.
%%%%%%%%%%%%%%%%%%%%%%%%%%%%%%%%%%%%%%%%%%%%%%%
%\begin{figure}[!h]
%\centering
%\includegraphics[width=0.48\textwidth,height=0.30\textwidth]{./staggered_grid.eps}
%\caption{An illustration of the 2-D staggered grid.}
%\label{staggered_grid}
%\end{figure}
%%%%%%%%%%%%%%%%%%%%%%%%%%%%%%%%%%%%%%%%%%%%%%%

%%%%%%%%%%%%%%%%%%%%%%%%%%%%%%%%%%%%%%%%%%%%%%
\begin{table}[!h]
  \centering
  \caption{\small Reference wave velocities, density and quality factors 
}\label{Parameters}
 \begin{lrbox}{\tablebox}
 \begin{tabular}{cccccccc}
 \hline\hline
 Medium & $c_P$ (km/s) &  $c_S$ (km/s) & $\rho$ (g/$\textup{cm}^3$) &  $Q_P$ & $Q_S$ &  $\gamma_P$ & $\gamma_S$ \\
 \hline
1	&	3.2	&	1.85	&	2.5	&	32	&	10	&	0.0099	&	0.0317\\	
2	&	3.2	&	1.85	&	2.5	&	100	&	50	&	0.0032	&	0.0064\\	
\hline\hline
 \end{tabular}
\end{lrbox}
\scalebox{1}{\usebox{\tablebox}}
\end{table}
%%%%%%%%%%%%%%%%%%%%%%%%%%%%%%%%%%%%%%%%%%%%%%

The computational domain is $\mathcal{X} = [0, 120]\times [0, 120]$ with $N_x = N_z = 512$. The parameters for the initial source term are $f_P = 100$Hz and $d_r = 60$. The time step is $\Delta t = 0.001$ and the final instant is $T = 15$. The memory length of the reference solution is $M_P = M_S = 499$. Different groups of memory length $M$ and scaling factor $\beta$ are considered: $M_P = M_S = 7, 15, 31, 63$, $\beta = 0.6, 1, 2, 8, 16,32$. For the sake of comparison, we also calculate the elastic modeling  under $c_P = 3.2$, $c_S = 1.85$, $\rho = 2.5$.

The wave attenuation phenomena are much more complicated in 2-D case.  A comparison among the displacement $u_3$ in $z$-direction and the $(xz)$-component $\sigma_{13}$ of the stress tensor is plotted in Figures \ref{fig_attenuated_wave_u3} and \ref{fig_attenuated_wave_stress13}, respectively. The difference between the viscoelastic wave propagation and the elastic counterpart is transparent. Since the Ricker wavelet source has two peaks at $t = 0$ and $t = 0.24$, two wave fronts are observed in the elastic media and they propagate independently. By contrast, two wave fronts in the viscoelastic media seem to stick together due to the time-lag effect (the power creep) in the stress-strain relation. One can see that the wave attenuation is more evident in the first viscoelastic media and the height of waveform decreases considerably in the time evolution. In the second  viscoelastic media, although the shapes of wave fronts change a lot, there is only a slight reduction in the height due to smaller loss of energy.

%%%%%%%%%%%%%%%%%%%%%%%%%%%%%%%%%%%%%%%%%%%%%%
\begin{figure}[!h]
\centering
\subfigure[$t = 5$.]{{\includegraphics[width=0.32\textwidth,height=0.18\textwidth]{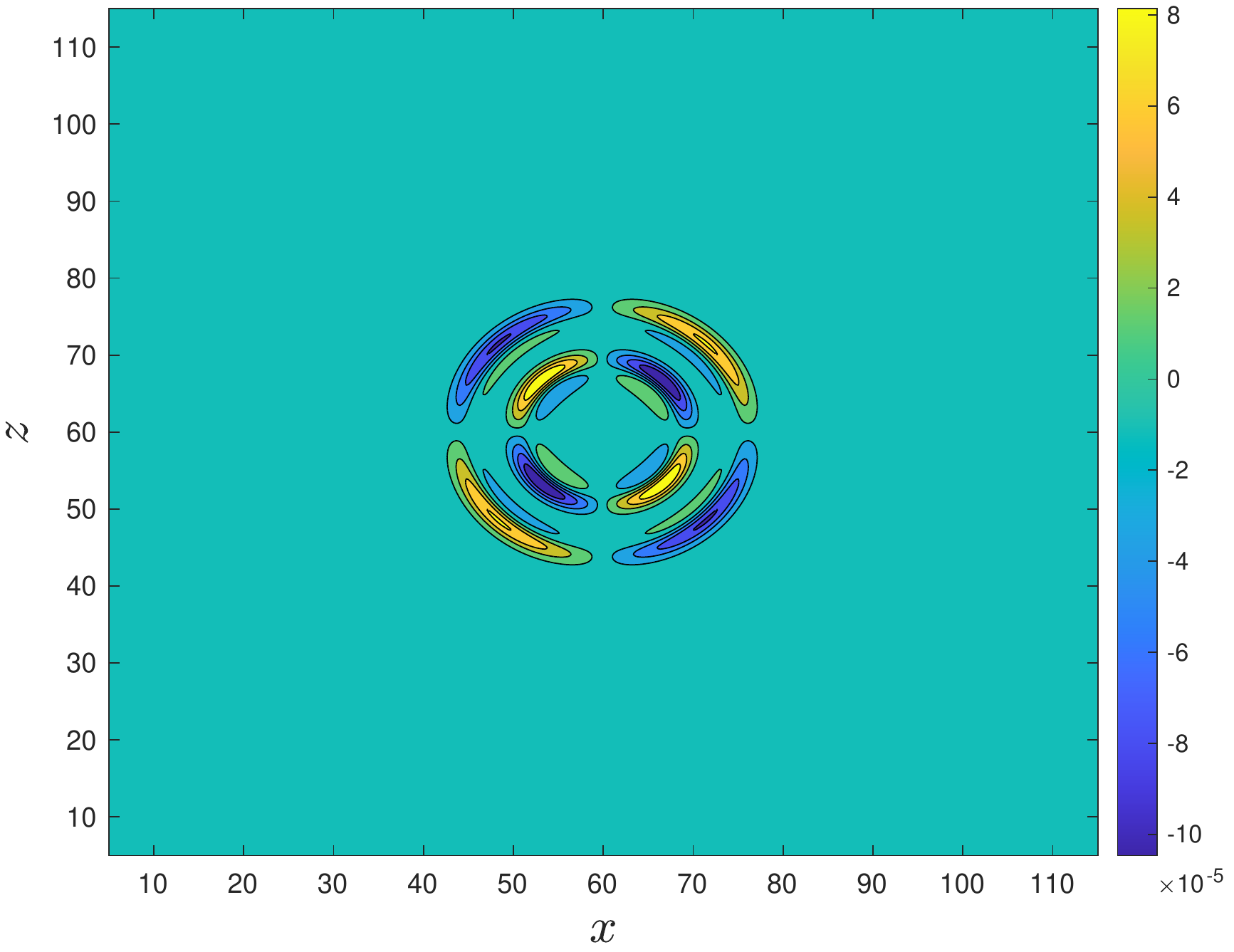}}
{\includegraphics[width=0.32\textwidth,height=0.18\textwidth]{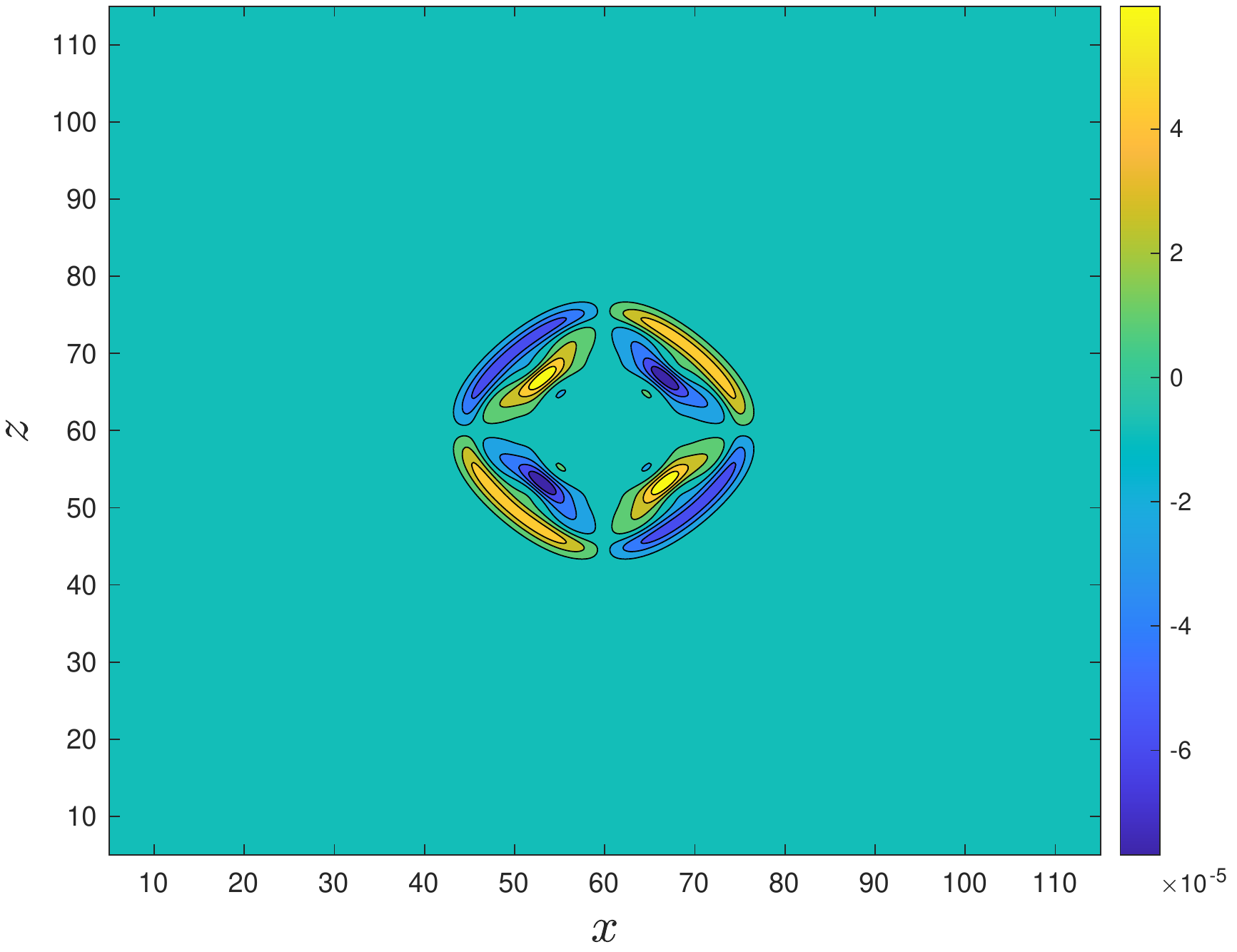}}
{\includegraphics[width=0.32\textwidth,height=0.18\textwidth]{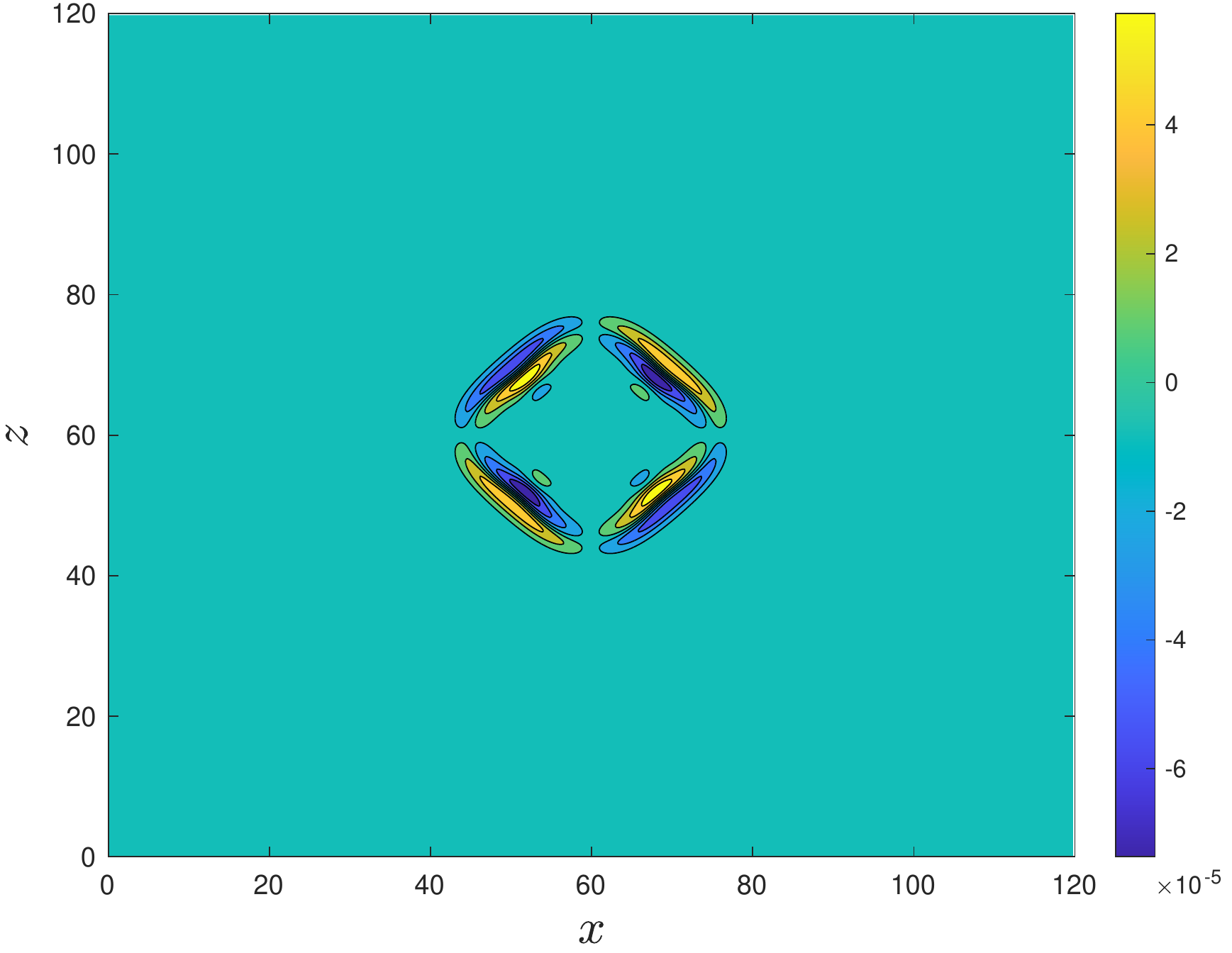}}
}
\\
\centering
\subfigure[$t = 10$.]{{\includegraphics[width=0.32\textwidth,height=0.18\textwidth]{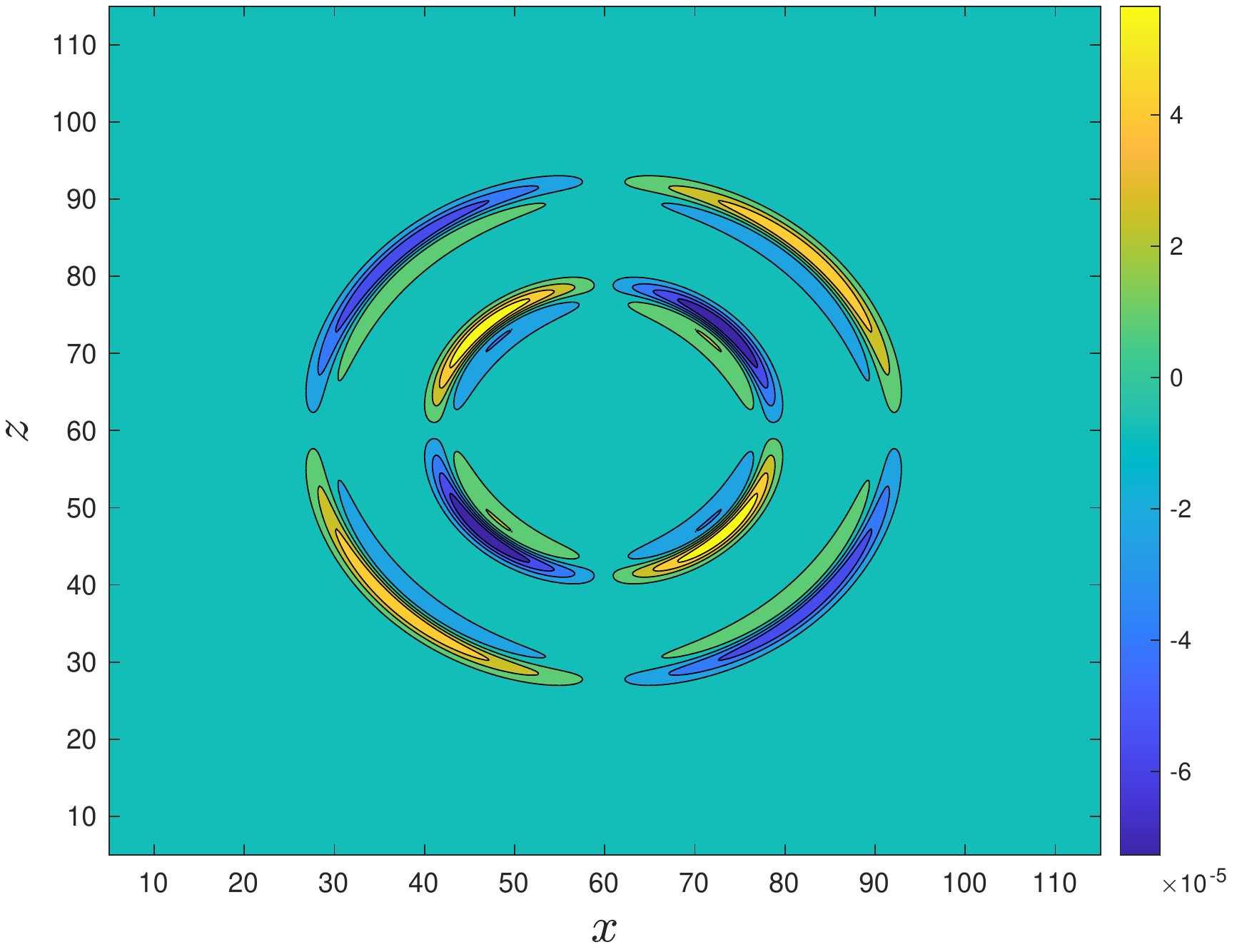}}
{\includegraphics[width=0.32\textwidth,height=0.18\textwidth]{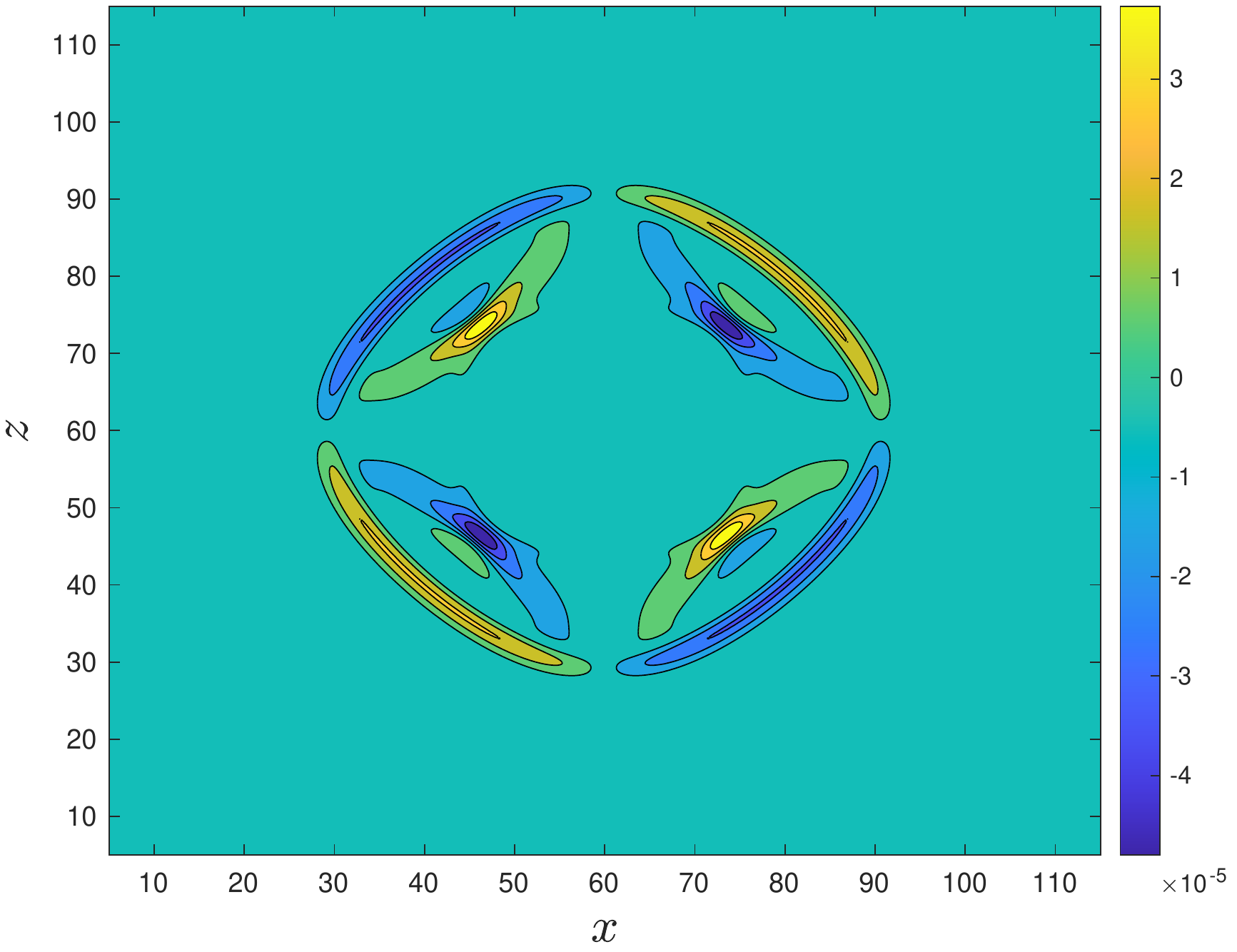}}
{\includegraphics[width=0.32\textwidth,height=0.18\textwidth]{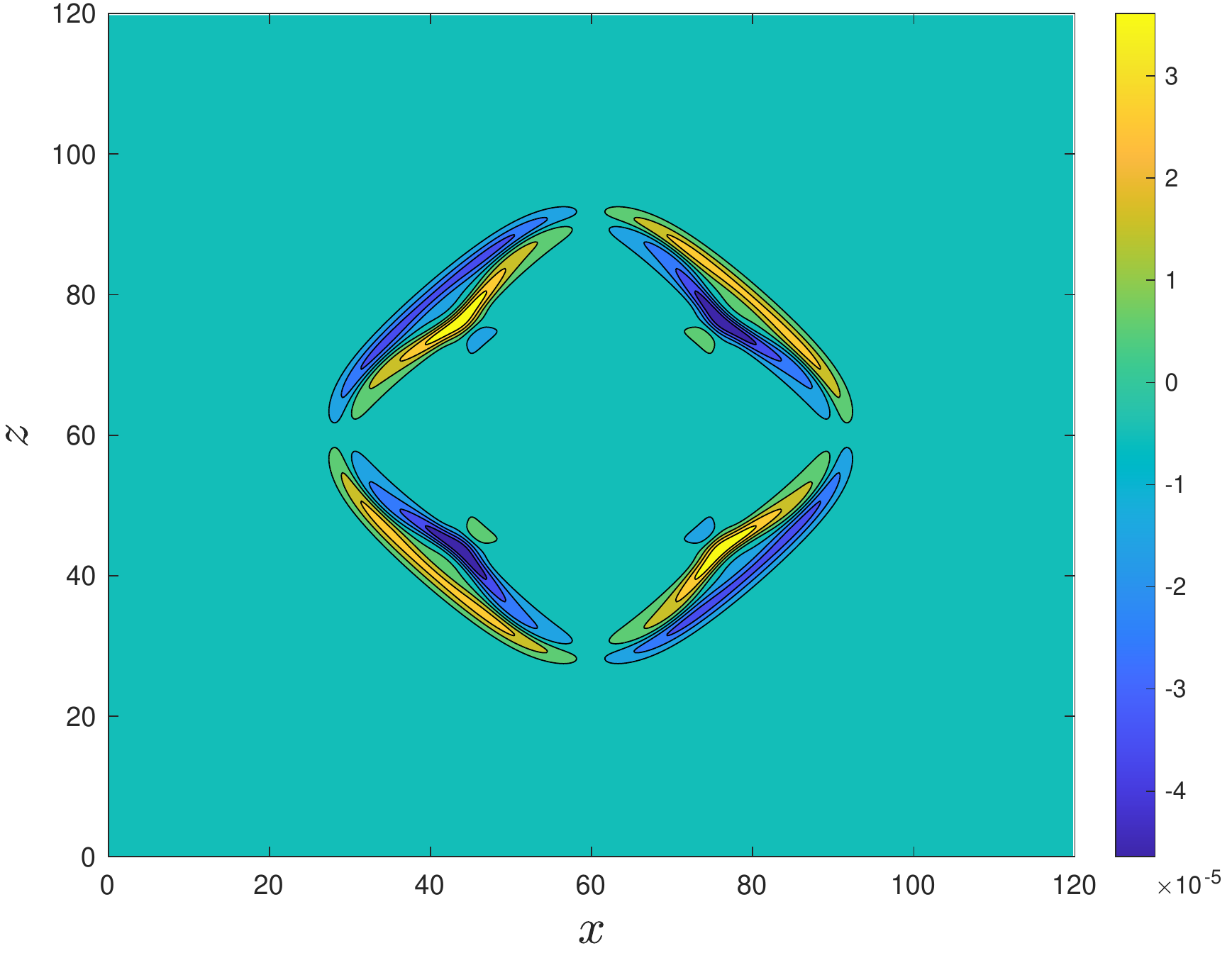}}
}
\\
\centering
\subfigure[$t = 15$.]{{\includegraphics[width=0.32\textwidth,height=0.18\textwidth]{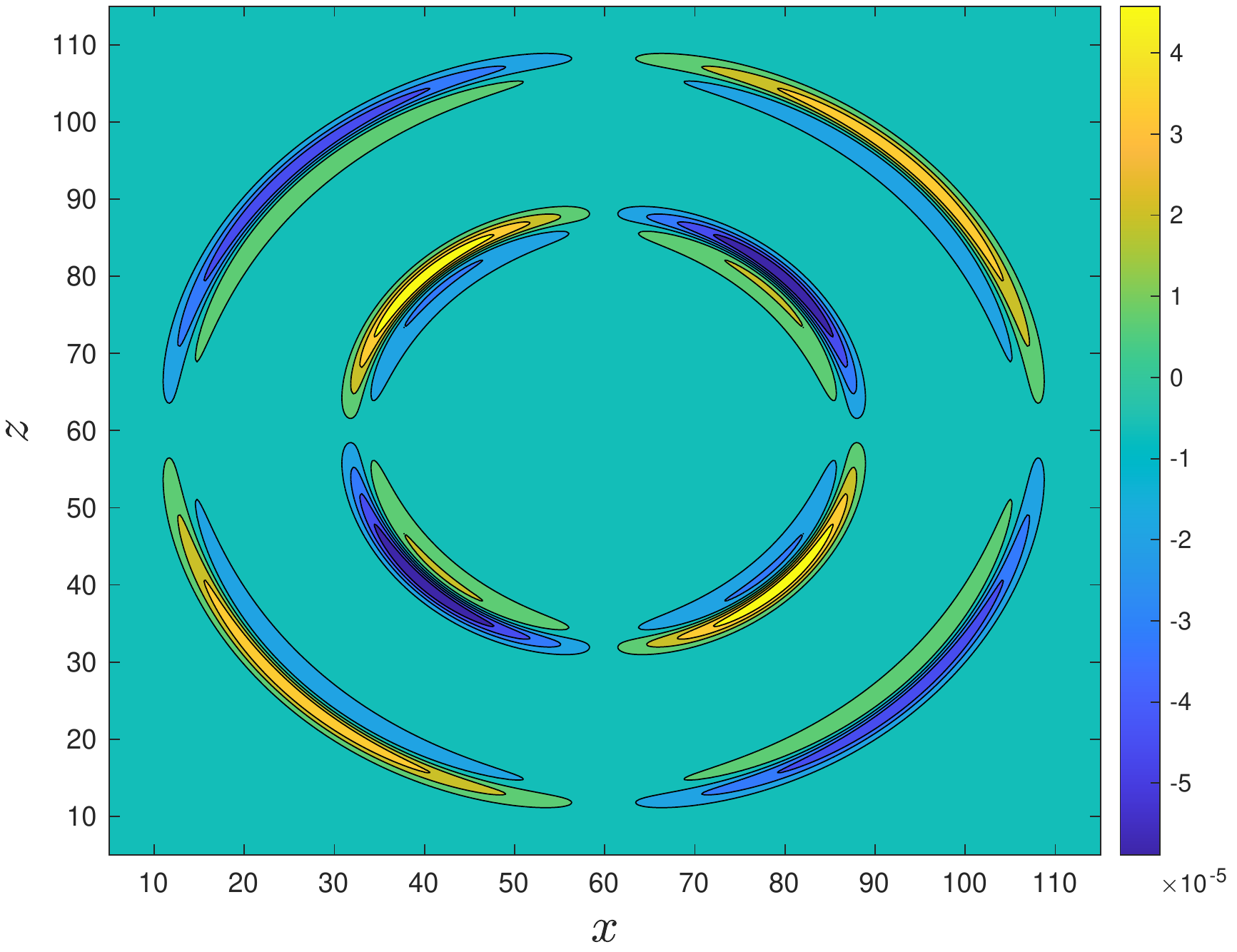}}
{\includegraphics[width=0.32\textwidth,height=0.18\textwidth]{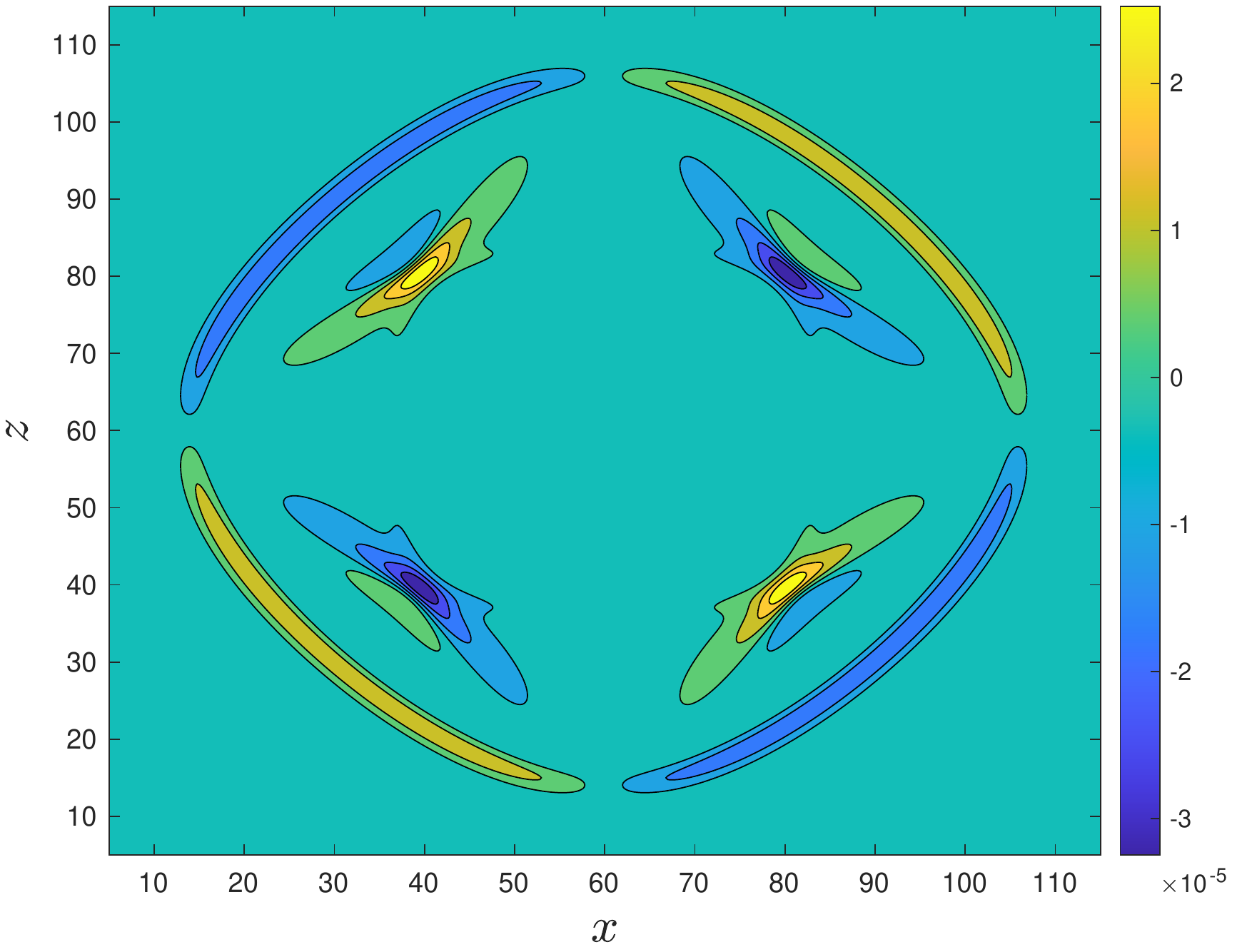}}
{\includegraphics[width=0.32\textwidth,height=0.18\textwidth]{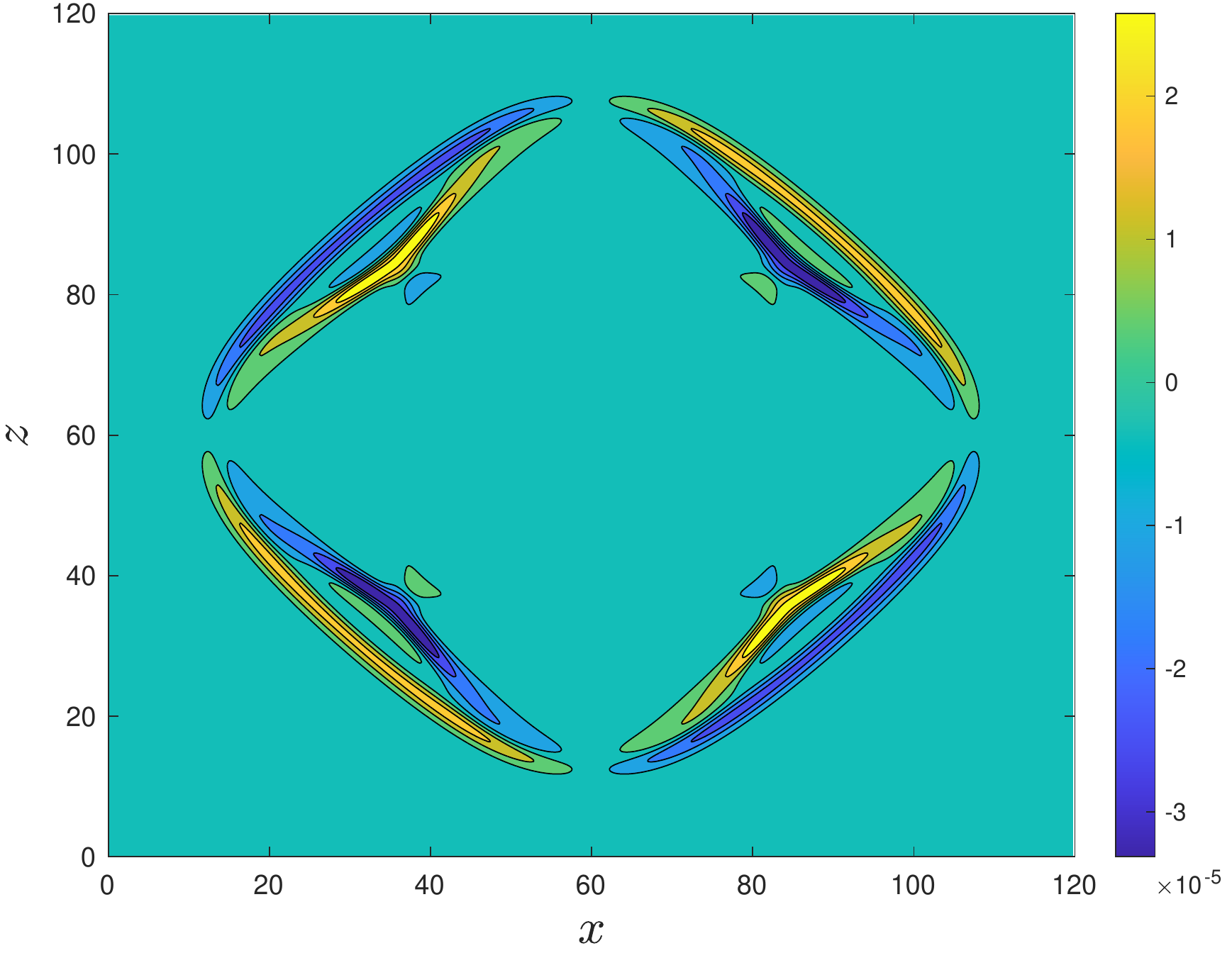}}
}
\caption{2-D constant-Q wave propagation: Displacement $u_3$. Left: elastic modeling, mid: viscoelastic modeling with quality factors $Q_P = 32$, $Q_S = 10$, right: viscoelastic modeling with quality factors $Q_P = 100$, $Q_S = 50$.}
\label{fig_attenuated_wave_u3}
\end{figure}
%%%%%%%%%%%%%%%%%%%%%%%%%%%%%%%%%%%%%%%%%%%%%%

%%%%%%%%%%%%%%%%%%%%%%%%%%%%%%%%%%%%%%%%%%%%%%
\begin{figure}[!h]
\centering
\subfigure[$t = 5$.]{{\includegraphics[width=0.32\textwidth,height=0.18\textwidth]{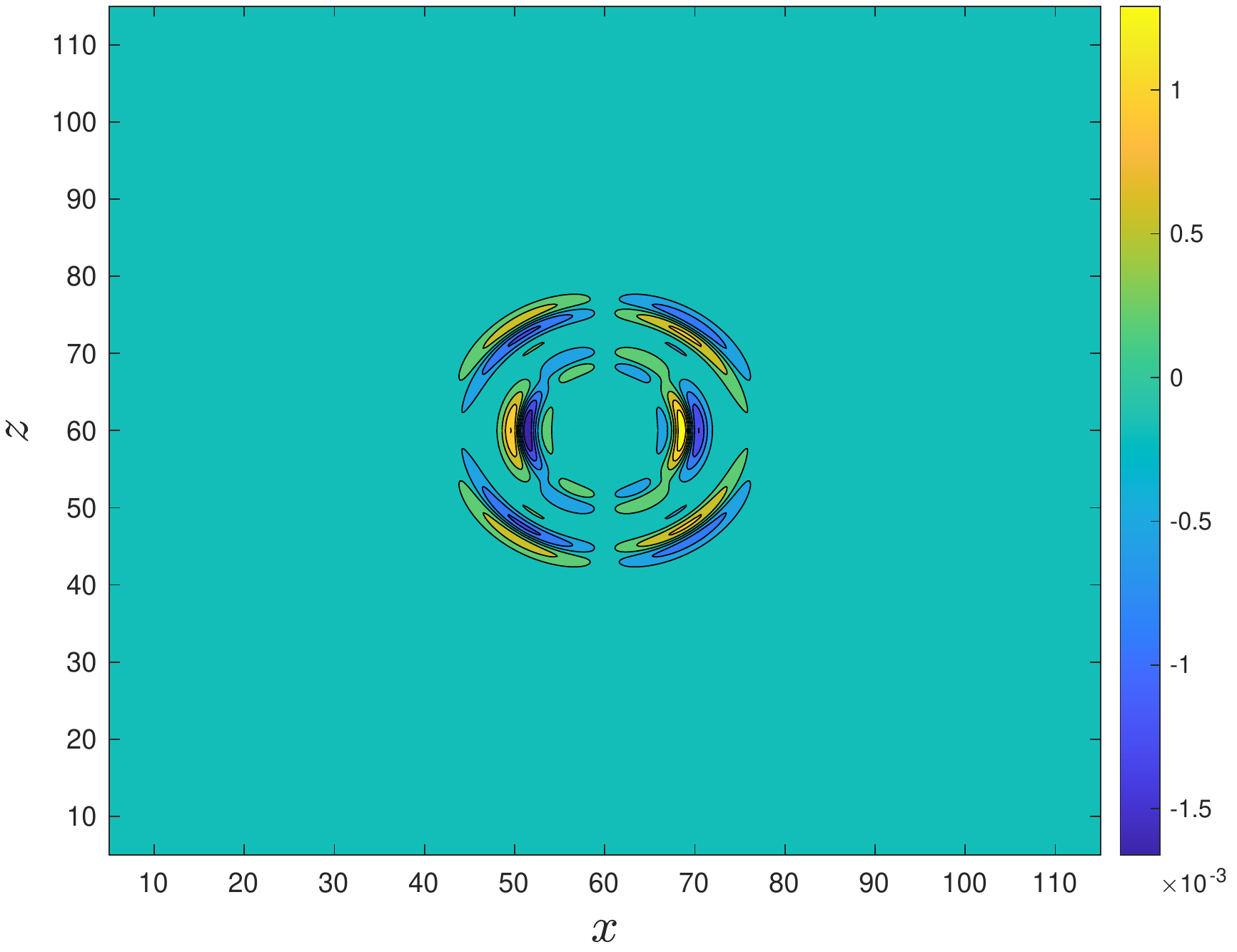}}
{\includegraphics[width=0.32\textwidth,height=0.18\textwidth]{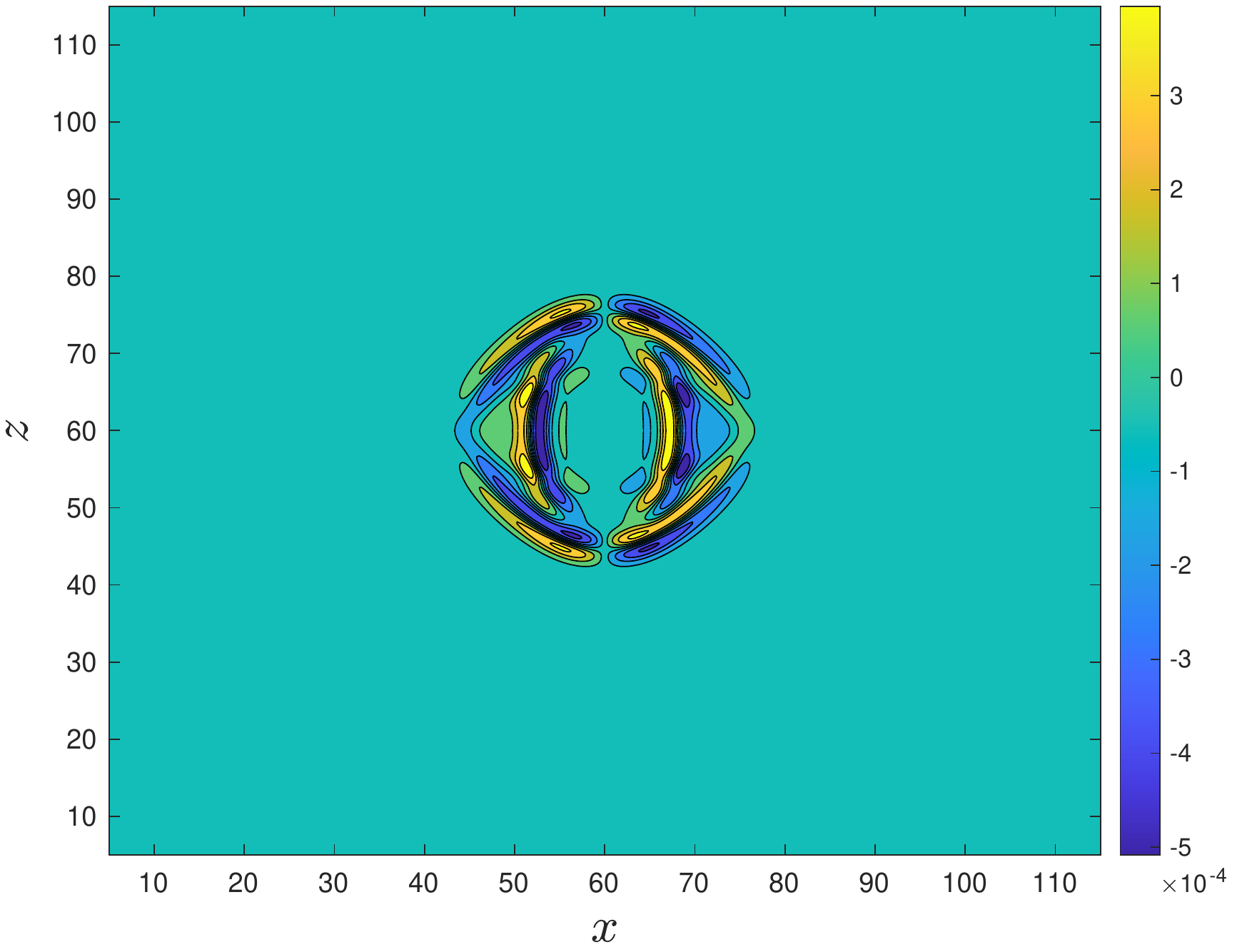}}
{\includegraphics[width=0.32\textwidth,height=0.18\textwidth]{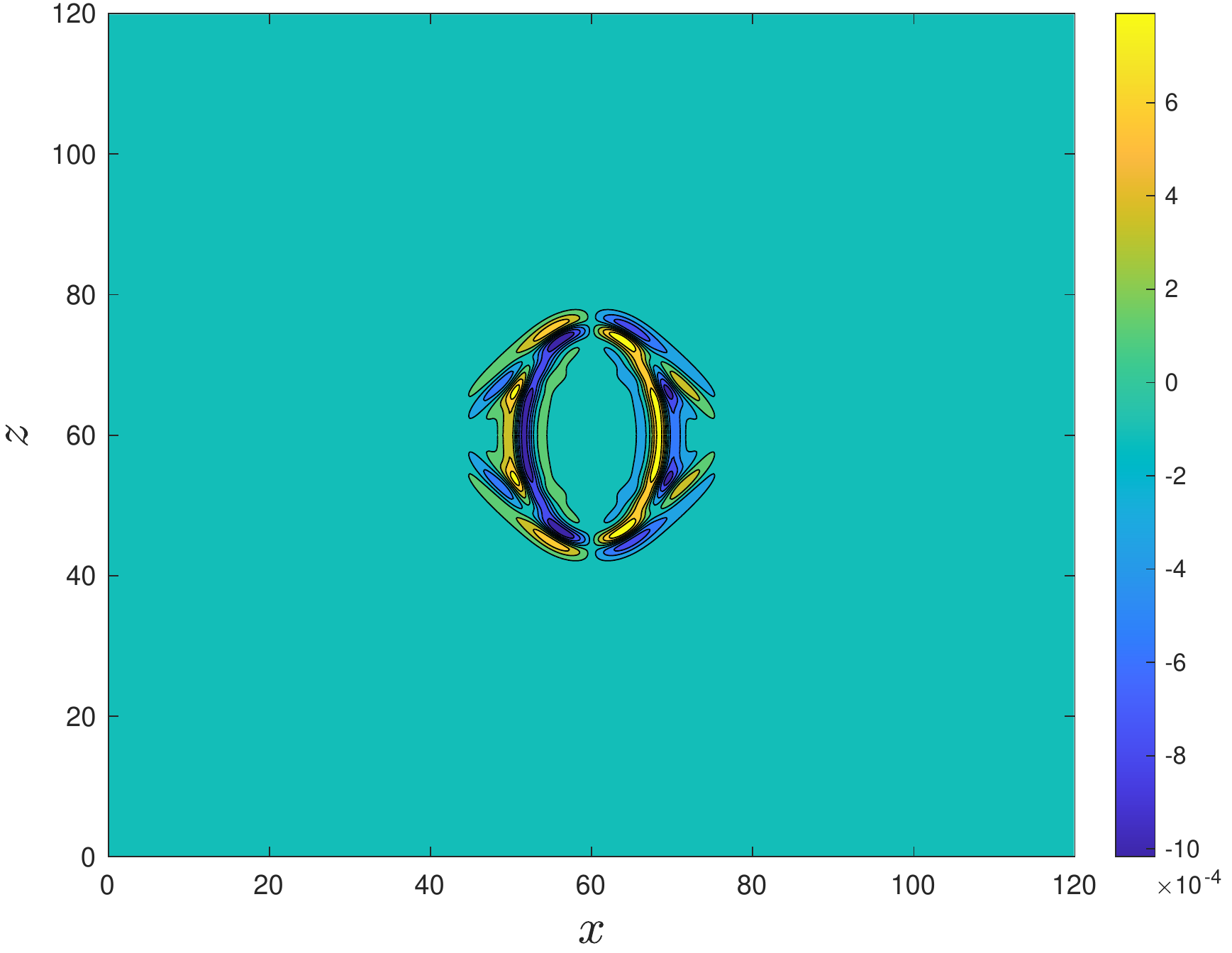}}
}
\\
\centering
\subfigure[$t = 10$.]{{\includegraphics[width=0.32\textwidth,height=0.18\textwidth]{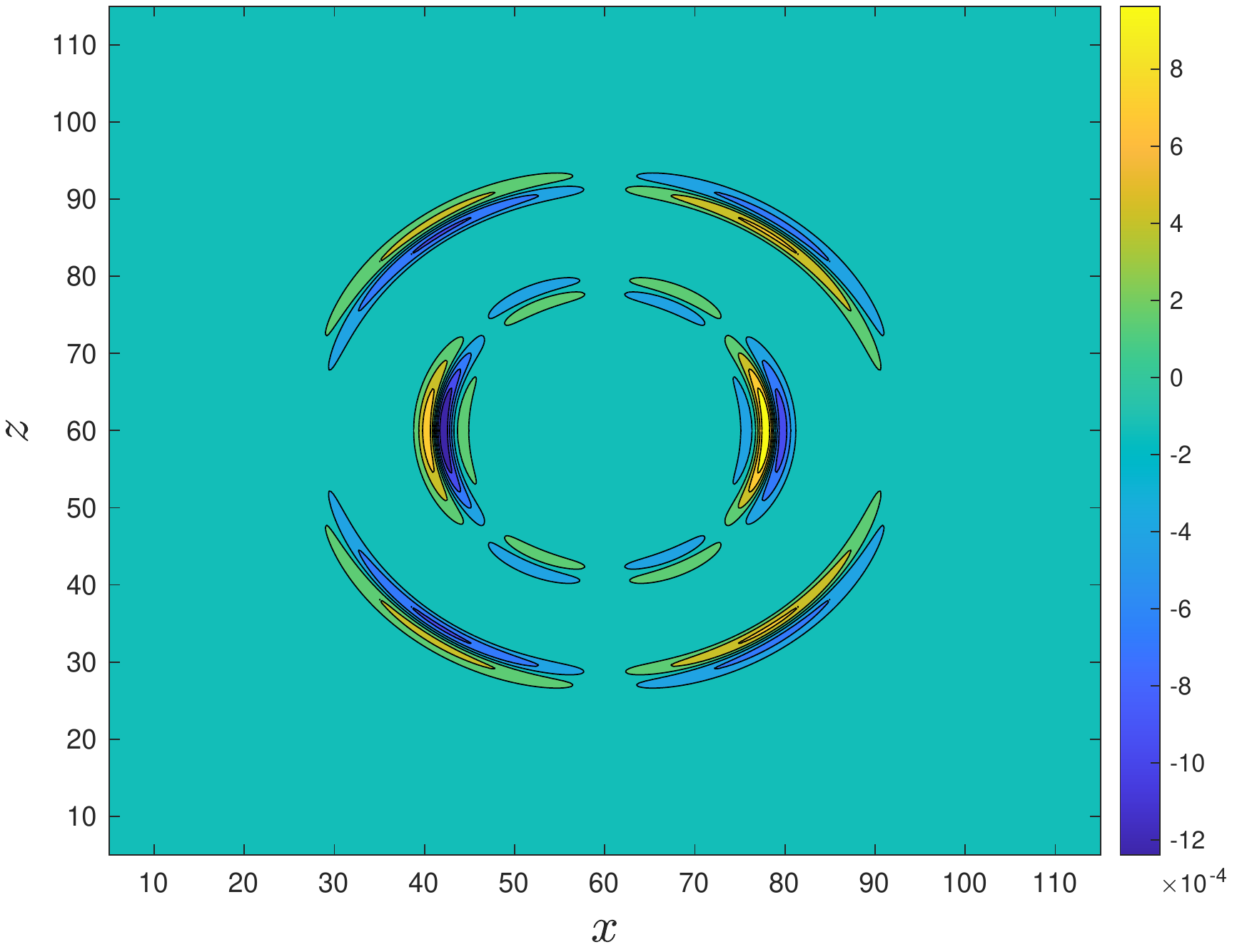}}
{\includegraphics[width=0.32\textwidth,height=0.18\textwidth]{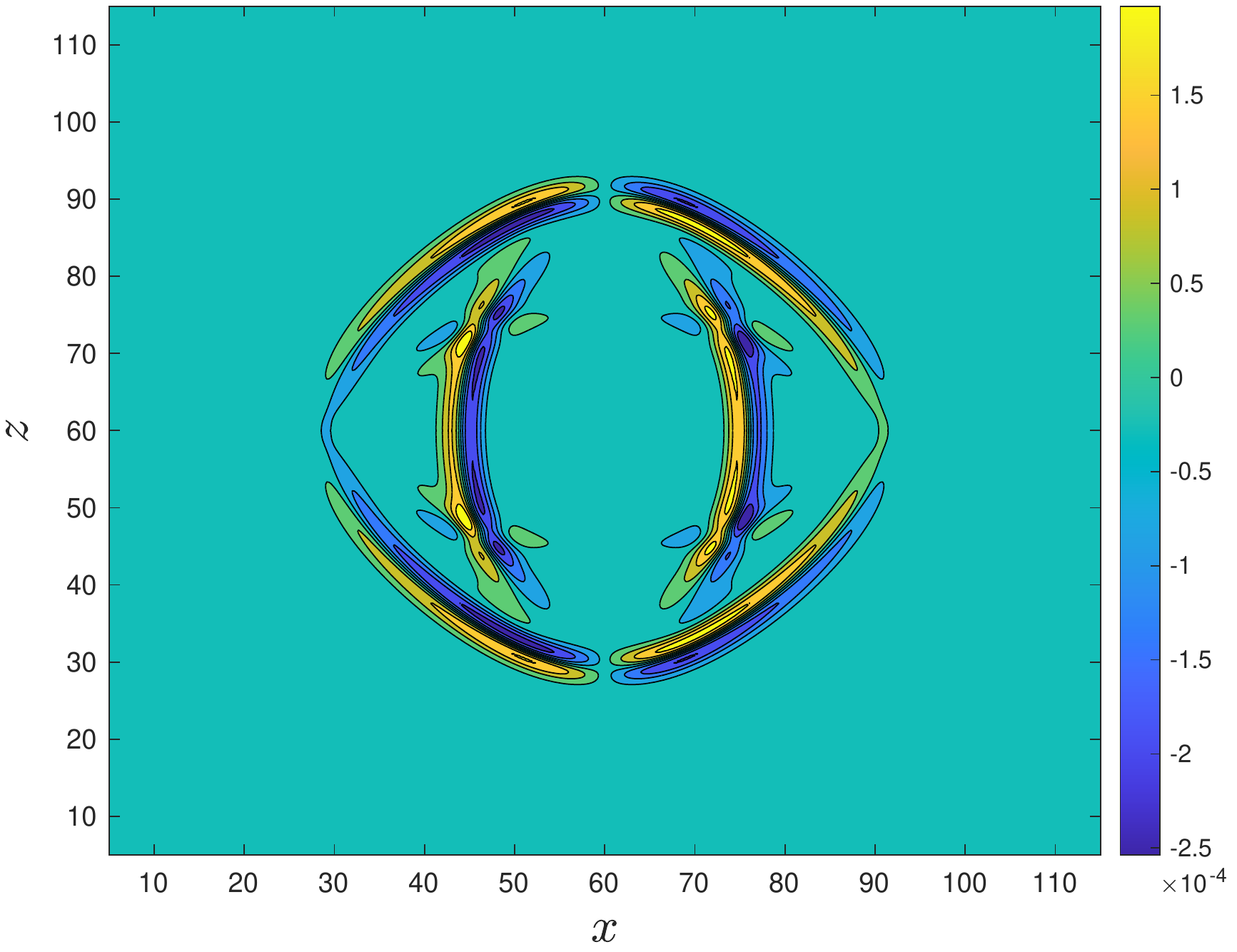}}
{\includegraphics[width=0.32\textwidth,height=0.18\textwidth]{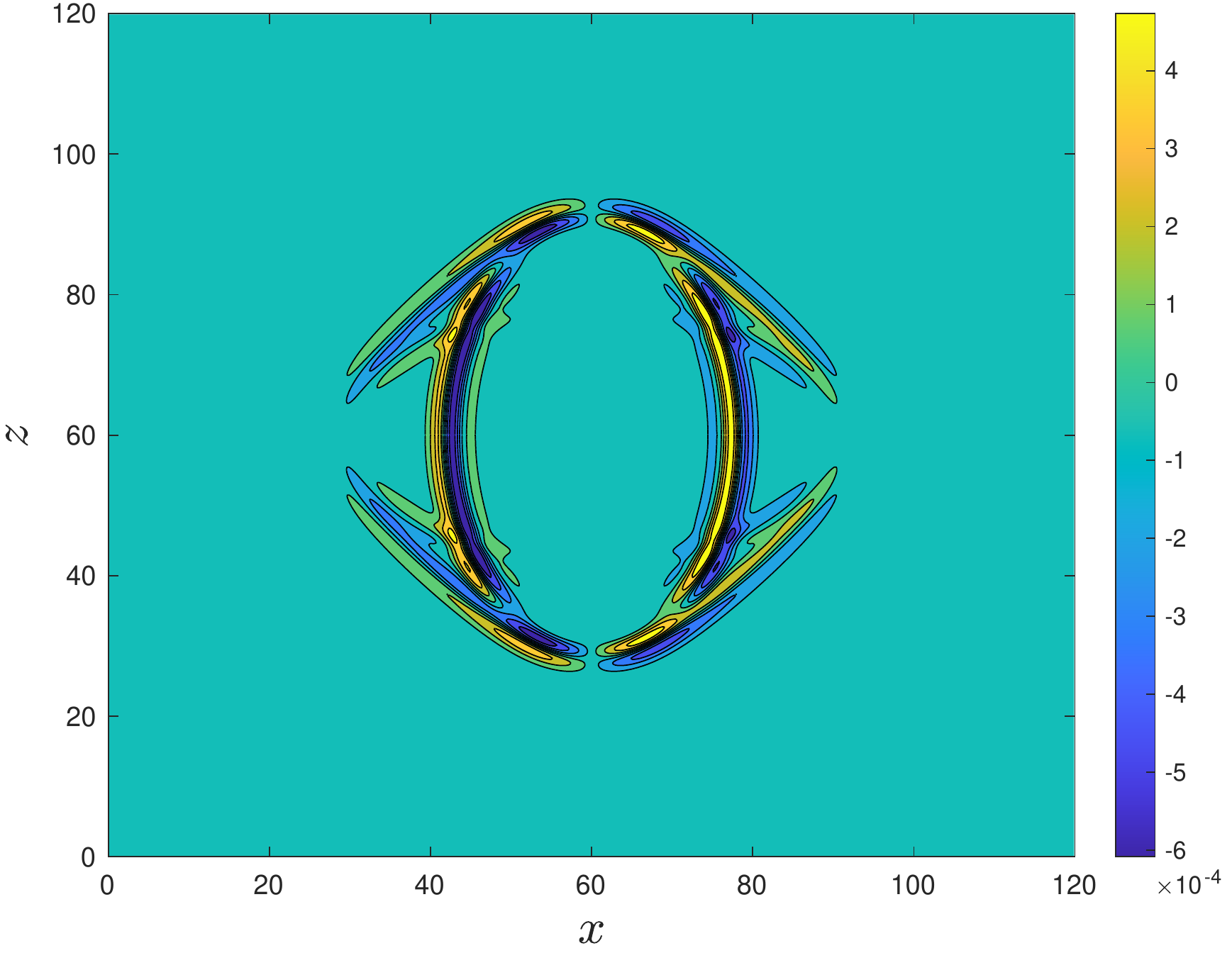}}
}
\\
\centering
\subfigure[$t = 15$.]{{\includegraphics[width=0.32\textwidth,height=0.18\textwidth]{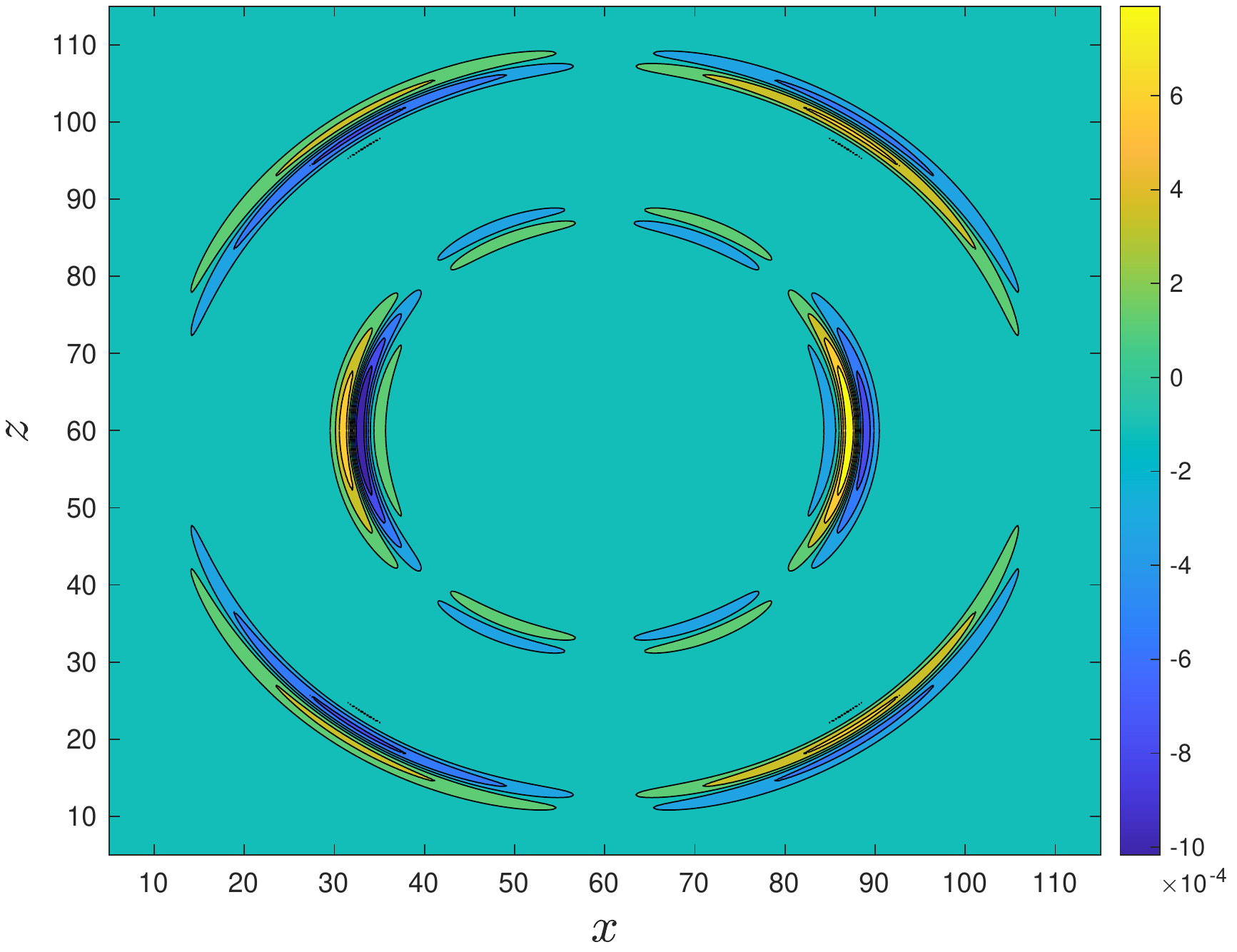}}
{\includegraphics[width=0.32\textwidth,height=0.18\textwidth]{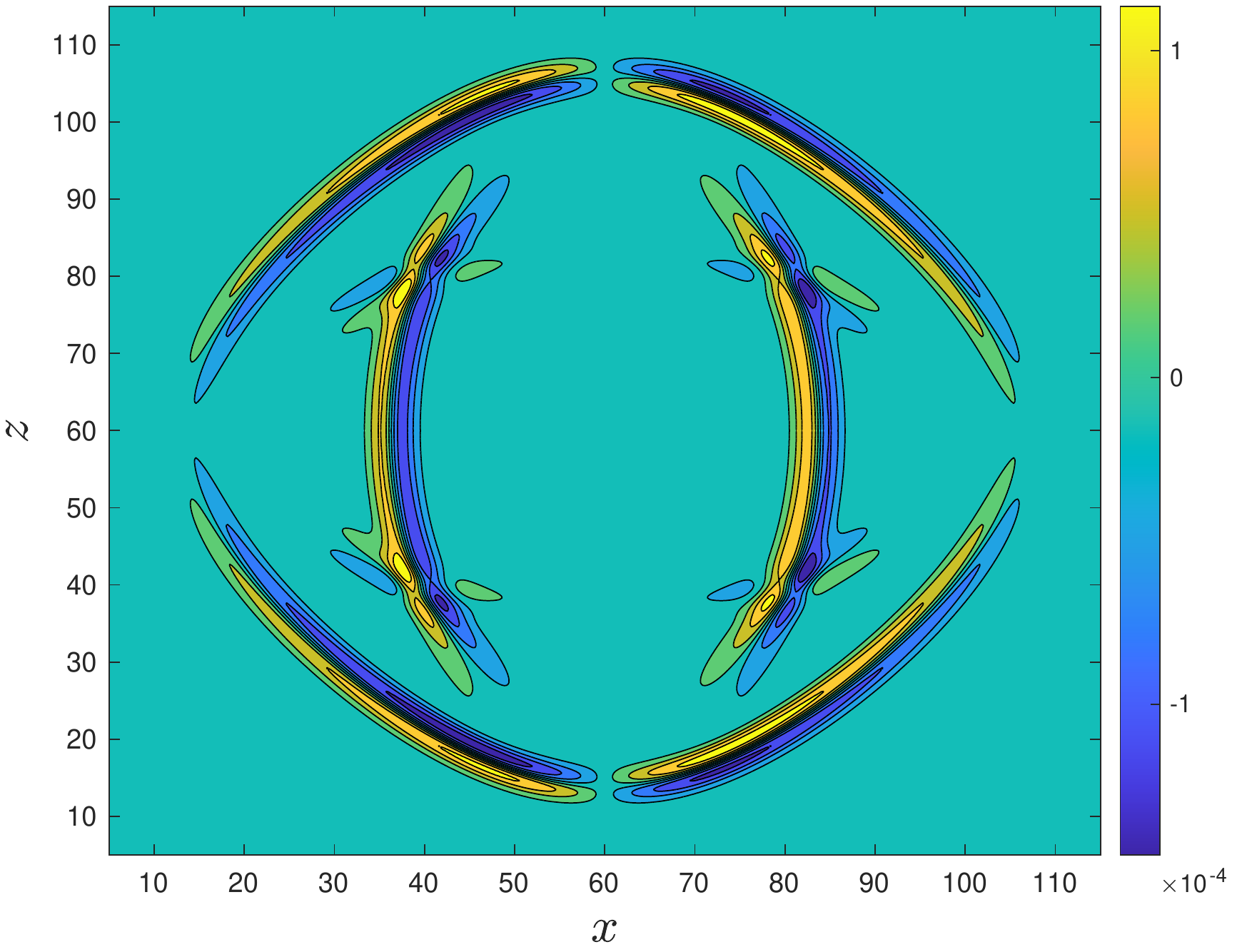}}
{\includegraphics[width=0.32\textwidth,height=0.18\textwidth]{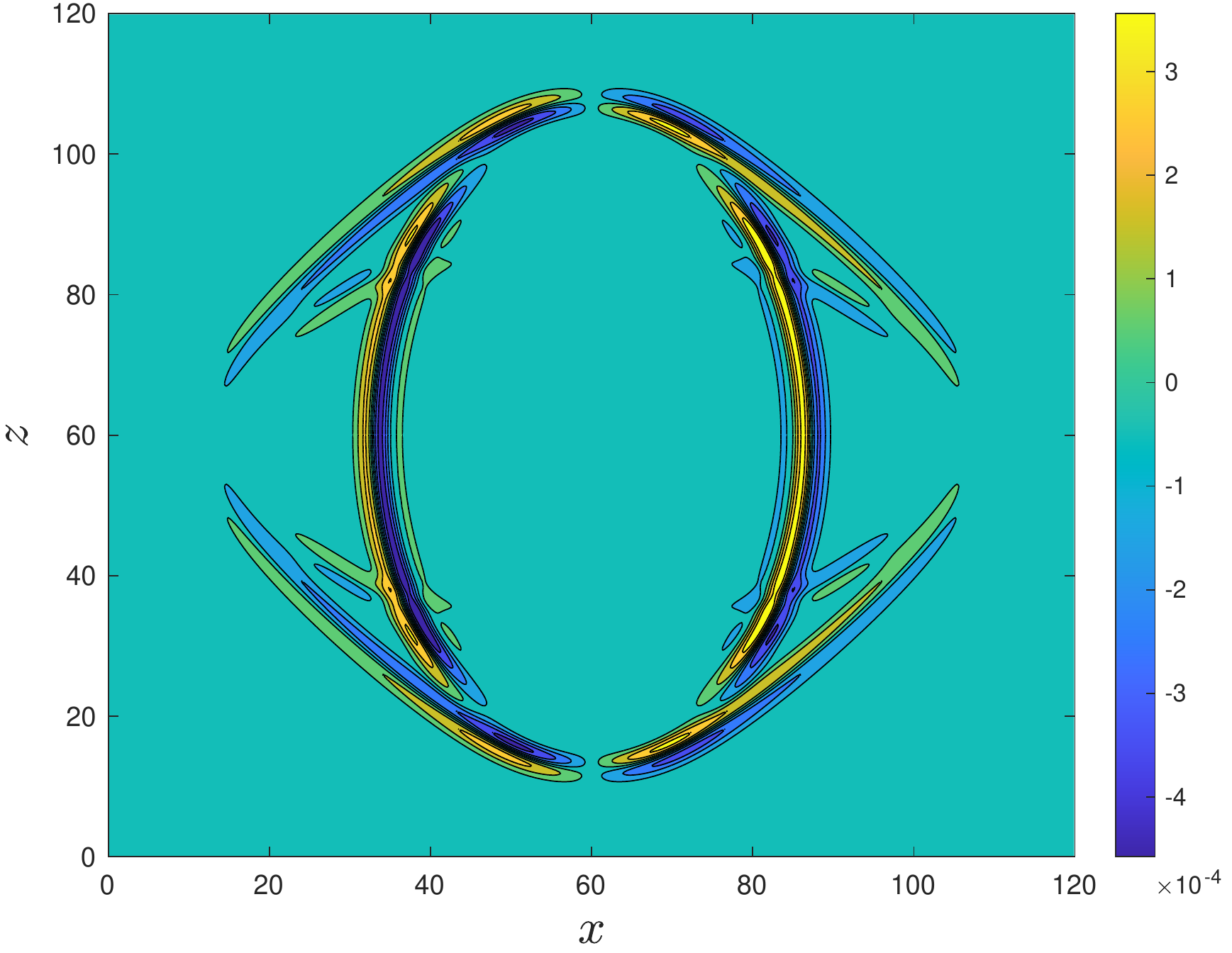}}
}
\caption{2-D constant-Q wave propagation: Stress $\sigma_{13}$. Left: elastic modeling, mid: viscoelastic modeling with quality factors $Q_P = 32$, $Q_S = 10$, right: viscoelastic modeling with quality factors $Q_P = 100$, $Q_S = 50$.}
\label{fig_attenuated_wave_stress13}
\end{figure}
%%%%%%%%%%%%%%%%%%%%%%%%%%%%%%%%%%%%%%%%%%%%%%

For 2-D constant-Q wave equation under $Q_P = 32$, $Q_S = 10$, we focus on the study on the dynamical increments of numerical errors, as presented in Figure \ref{2d_evo_unscaled} associated a validation on the convergence of memory length $M_P = M_S = M$, and how the scaling technique ameliorates such problem. The numerical errors under different scaling factors and the averaged computational time are collected in Table \ref{2d_scaling}. A visualization of the effect of scaling is presented in  Figure \ref{fig_comparison}. According to the results, we have the following observations.

\begin{itemize}

\item[(1)] Although  the spectral convergence of the Laguerre-Gauss quadrature is validated, the dynamical increments of projection errors is still observed in Figure \ref{2d_evo_unscaled}. Compared with Figure \ref{fig_1d_memory_convergence}, the reduction of accuracy is more evident because of the non-decay property of the source term. This observation coincides with our theoretical prediction.

\item[(2)] With appropriate scaling technique, the dynamical increments of numerical errors can be dramatically alleviated, albeit not eliminated. One can see in Table \ref{2d_scaling} that a large scaling factor $\beta > 1$ may considerably enhance the numerical accuracy under the prescribed memory length without additional memory requirement or arithmetic complexity, while $\beta < 1$ may even lead to a reduction in accuracy.  In practice, $\beta$ should be chosen large enough in order to suppress the amplification induced by the source term (see $M=32,\beta = 16$), but too large $\beta$ is still not recommended as the truncation errors are augmented.

\item[(3)] In Figure \ref{fig_comparison}, we compare the unscaled and scaled results of $u_3$ under $M = 32$. Actually, the profile of wave propagation is very similar. The difference mainly lies in the height of waveforms as inadequate memory variables may underestimate the contribution of the singular integral and cause artificial loss of energy. Fortunately simply choosing a scaling factor $\beta >1$ can dramatically alleviate this problem.

\item[(4)] We can see in Table \ref{2d_evo_unscaled} that under $M= 32$ and $\beta =16$, SMOS can achieve a relative $L^\infty$-error about $2\%-4\%$ at final instant $T = 15$, and save about $82\%$ computational time and $97\%$ memory storage compared to $M+1 = 500$. The memory length can be even shortened for smaller $\gamma$. 

\end{itemize}

%%%%%%%%%%%%%%%%%%%%%%%%%%%%%%%%%%%%%%%%%%%%%
\begin{figure}[!h]
\centering
\subfigure[Convergence w.r.t. $M$.]{\includegraphics[width=0.48\textwidth,height=0.30\textwidth]{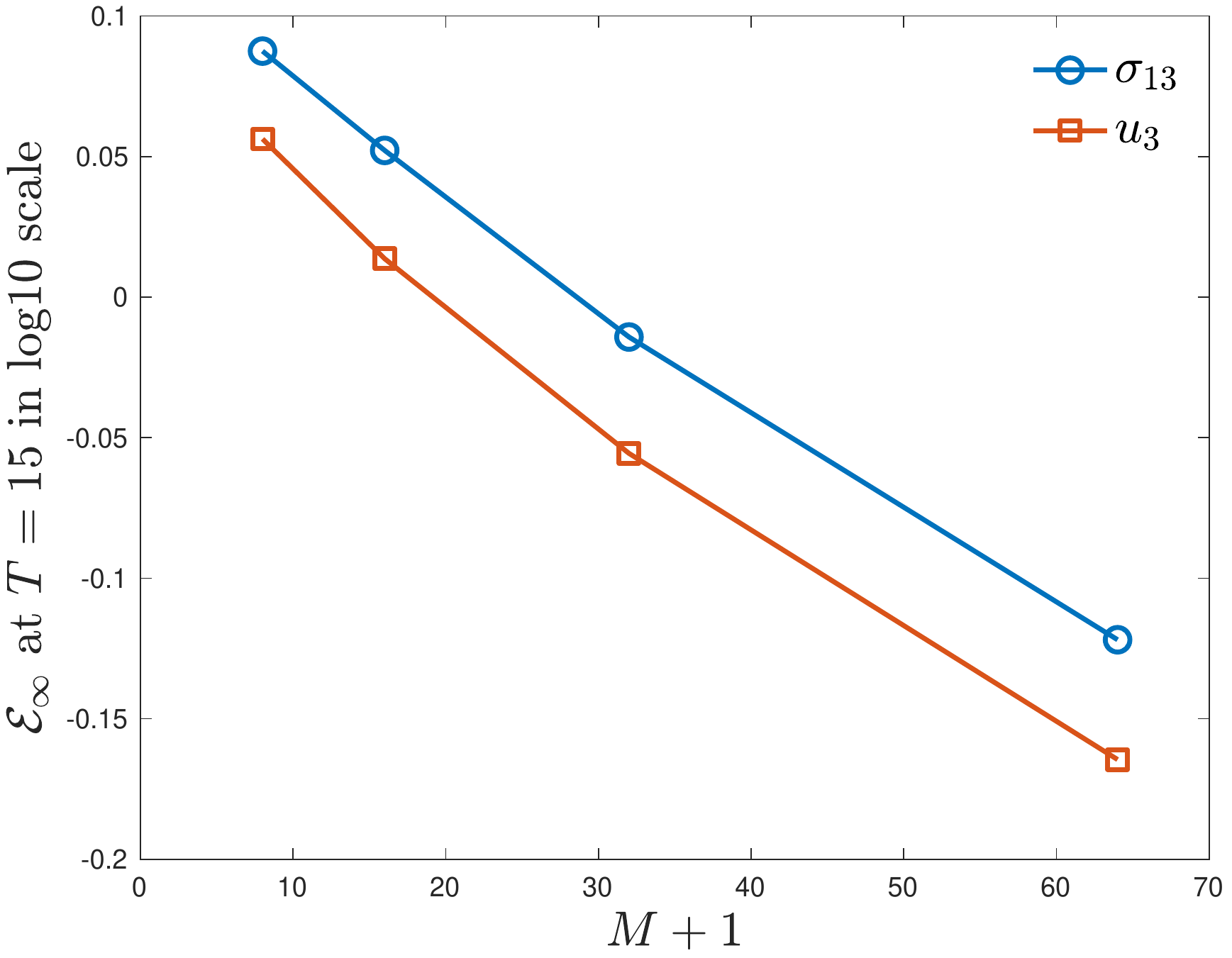}}
\\
\centering
\subfigure[Dynamical errors  in $u_3$.]{\includegraphics[width=0.48\textwidth,height=0.30\textwidth]{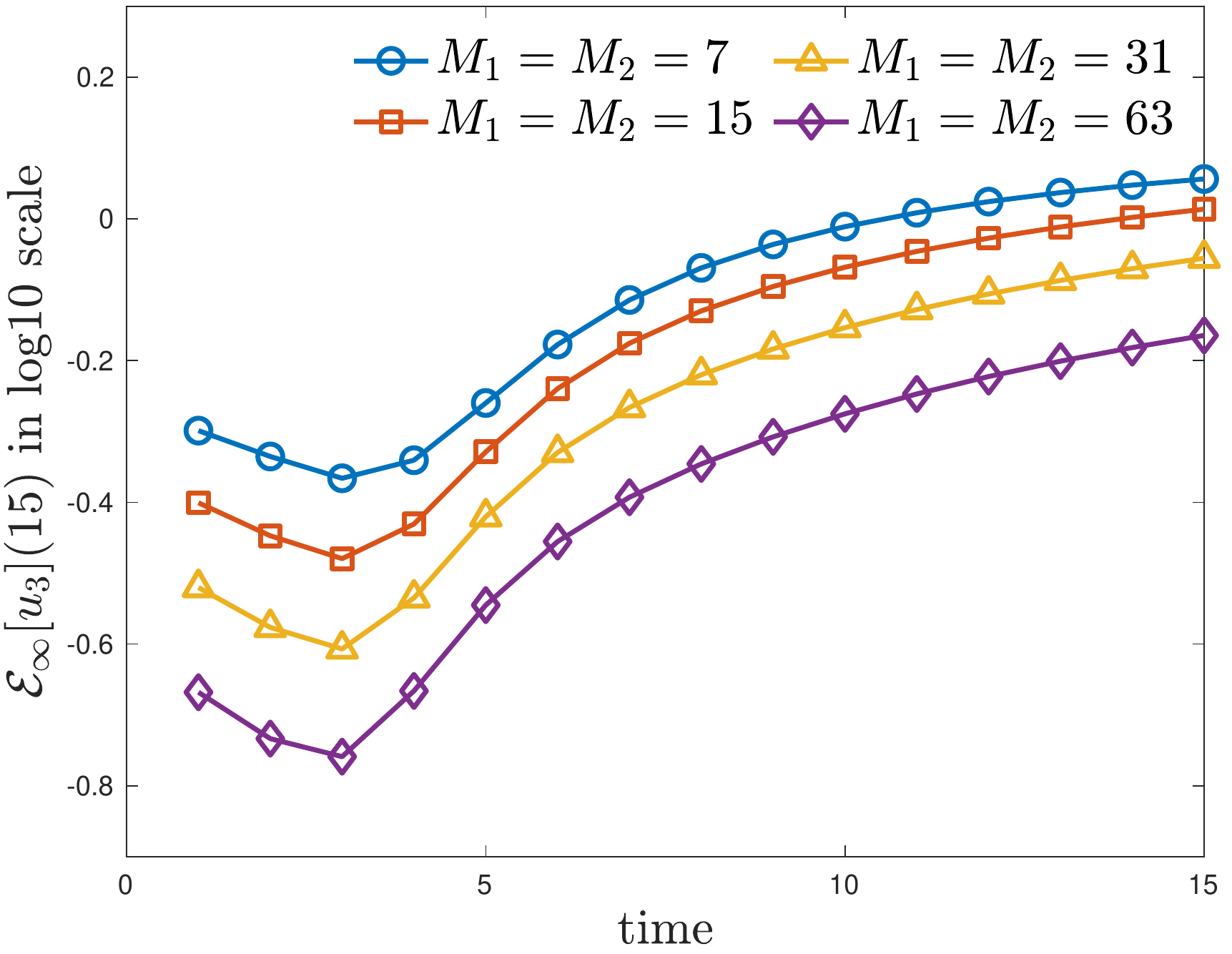}}
\subfigure[Dynamical errors in $\sigma_{13}$.]{\includegraphics[width=0.48\textwidth,height=0.30\textwidth]{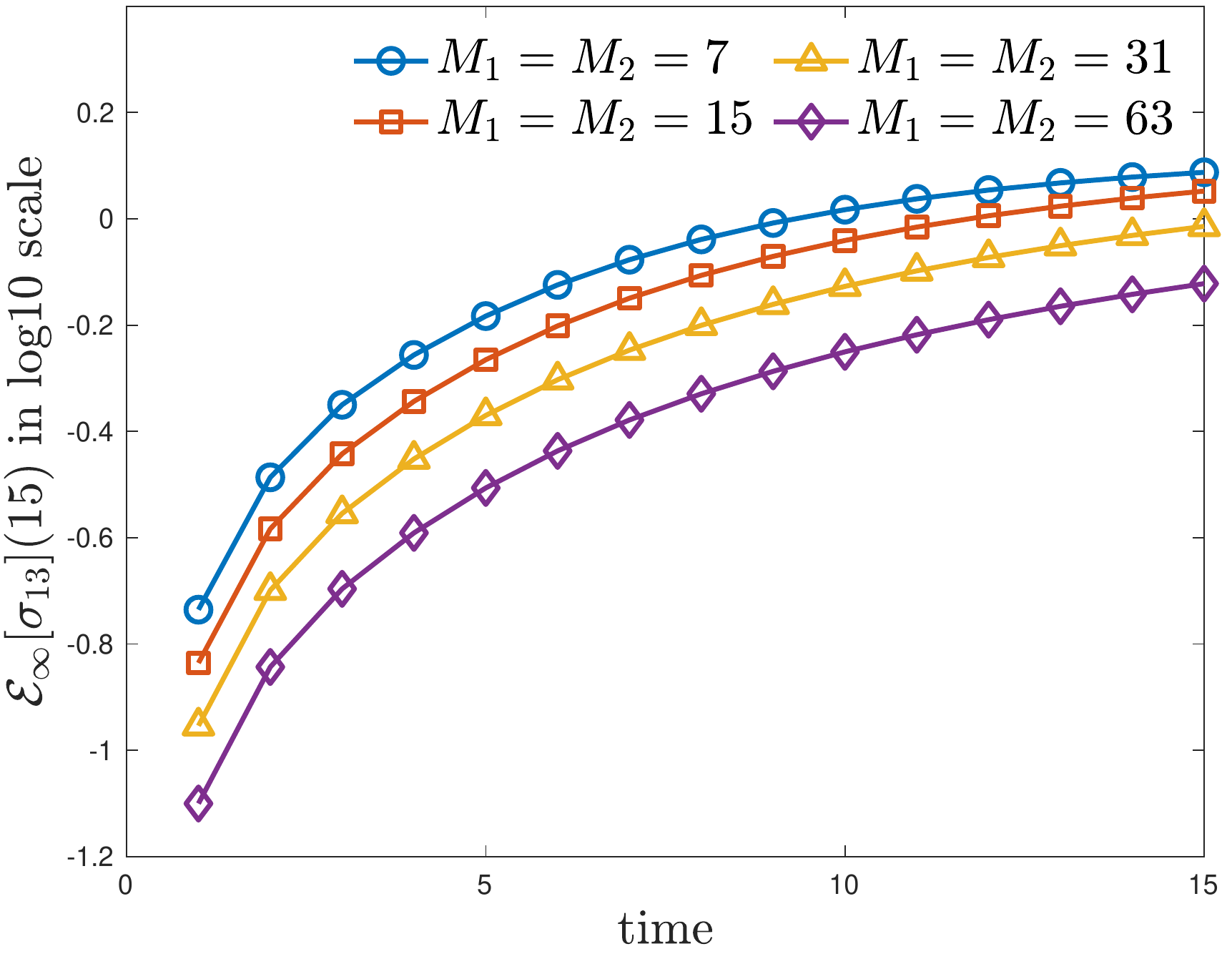}}
\caption{2-D constant-Q wave propagation: The dynamical errors under different memory lengths $M$. Although the spectral accuracy of the Laguerre-Gauss quadrature is verified, the dynamical increments of projection errors are clearly observed, regardless of choice of $M$.}
\label{2d_evo_unscaled}
\end{figure}
%%%%%%%%%%%%%%%%%%%%%%%%%%%%%%%%%%%%%%%%%%%%%

%%%%%%%%%%%%%%%%%%%%%%%%%%%%%%%%%%%%%%%%%%%%%
%\begin{figure}[!h]
%\centering
%\subfigure[$M = 8$.]
%{\includegraphics[width=0.48\textwidth,height=0.30\textwidth]{./2d_u3_error_evo_scaled_M8.pdf}
%{\includegraphics[width=0.48\textwidth,height=0.30\textwidth]{./2d_strain13_error_evo_scaled_M8.pdf}}}
%%\\
%%\centering
%%\subfigure[$M = 16$.]
%%{\includegraphics[width=0.48\textwidth,height=0.30\textwidth]{./2d_u3_error_evo_scaled_M16.pdf}
%%{\includegraphics[width=0.48\textwidth,height=0.30\textwidth]{./2d_strain13_error_evo_scaled_M16.pdf}}
%%}
%\\
%\centering
%\subfigure[$M = 32$.]{
%{\includegraphics[width=0.48\textwidth,height=0.30\textwidth]{./2d_u3_error_evo_scaled_M32.pdf}}
%{\includegraphics[width=0.48\textwidth,height=0.30\textwidth]{./2d_strain13_error_evo_scaled_M32.pdf}}
%}
%\caption{2-D constant-Q wave propagation: The dynamical increments of projection errors (left: $u_3$, right: $\sigma_{13}$) can be suppressed by a scaling factor $\beta >1$, but be augmented when $\beta < 1$.}
%\label{2d_evo_scaled}
%\end{figure}
%%%%%%%%%%%%%%%%%%%%%%%%%%%%%%%%%%%%%%%%%%%%%%

%%%%%%%%%%%%%%%%%%%%%%%%%%%%%%%%%%%%%%%%%%%%%%
\begin{figure}[!h]
\centering
\subfigure[$M = 499$.]{\includegraphics[width=0.32\textwidth,height=0.18\textwidth]{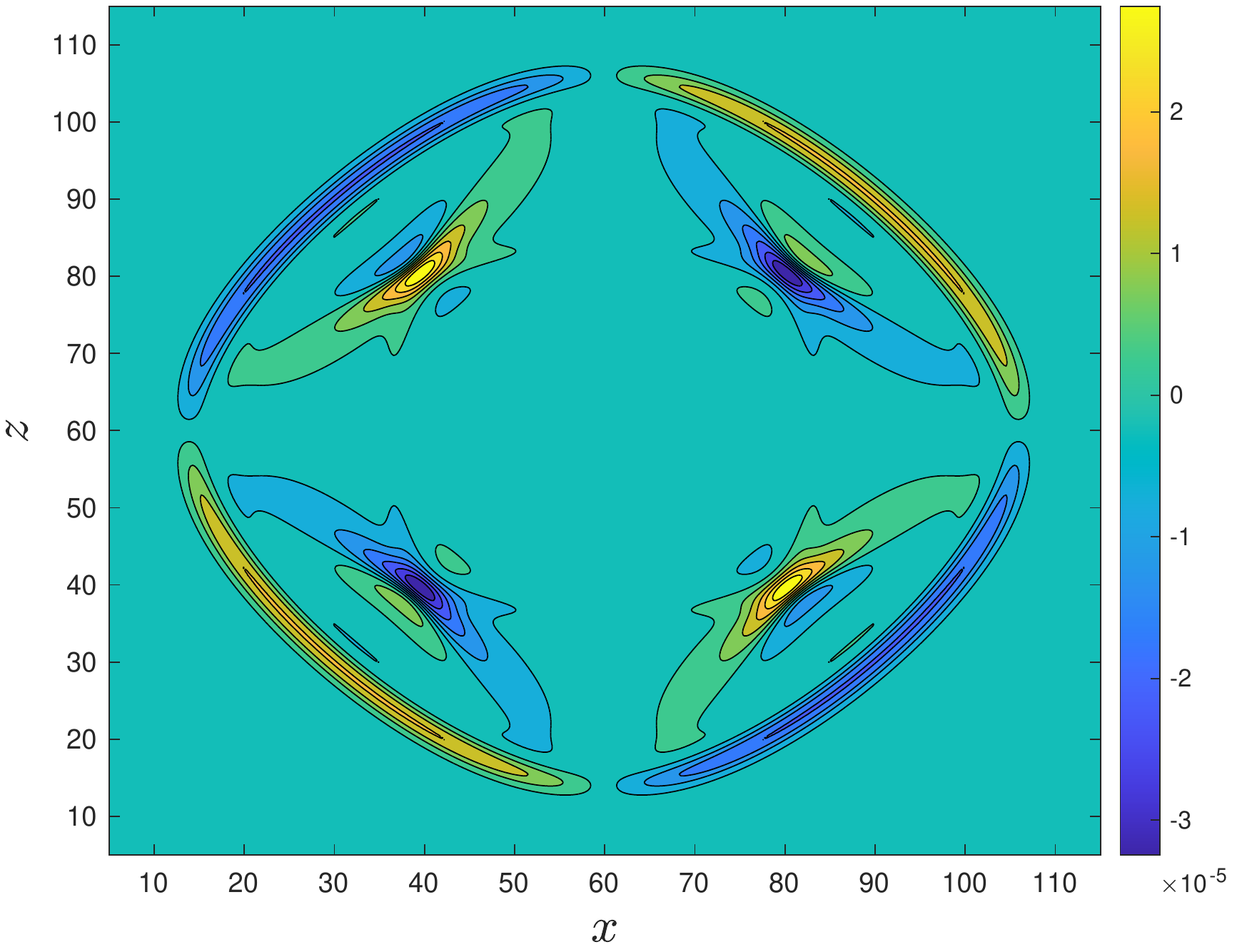}}
\subfigure[$M = 31$, unscaled.]{\includegraphics[width=0.32\textwidth,height=0.18\textwidth]{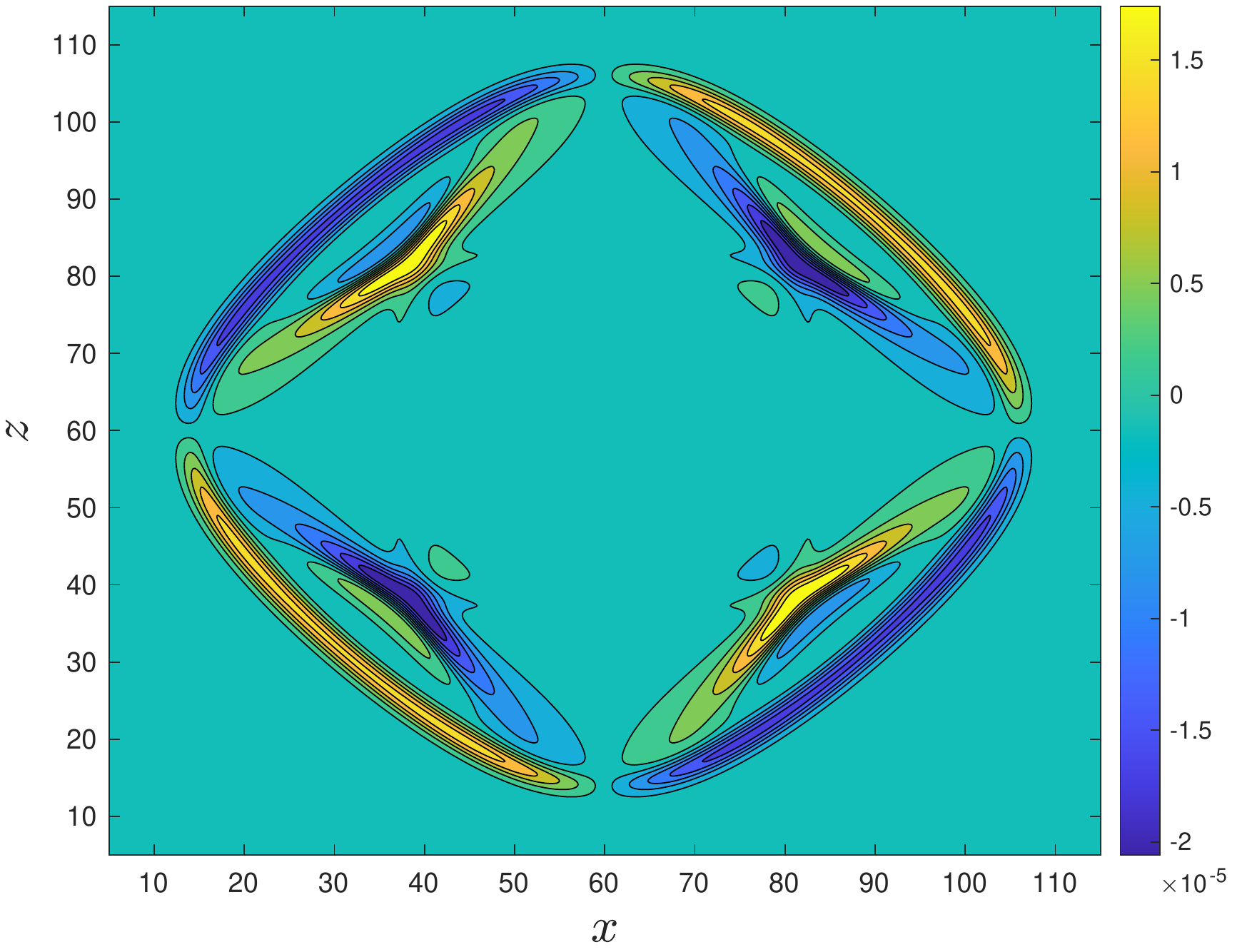}}
\subfigure[$M = 31$, $\beta =16$. ]{\includegraphics[width=0.32\textwidth,height=0.18\textwidth]{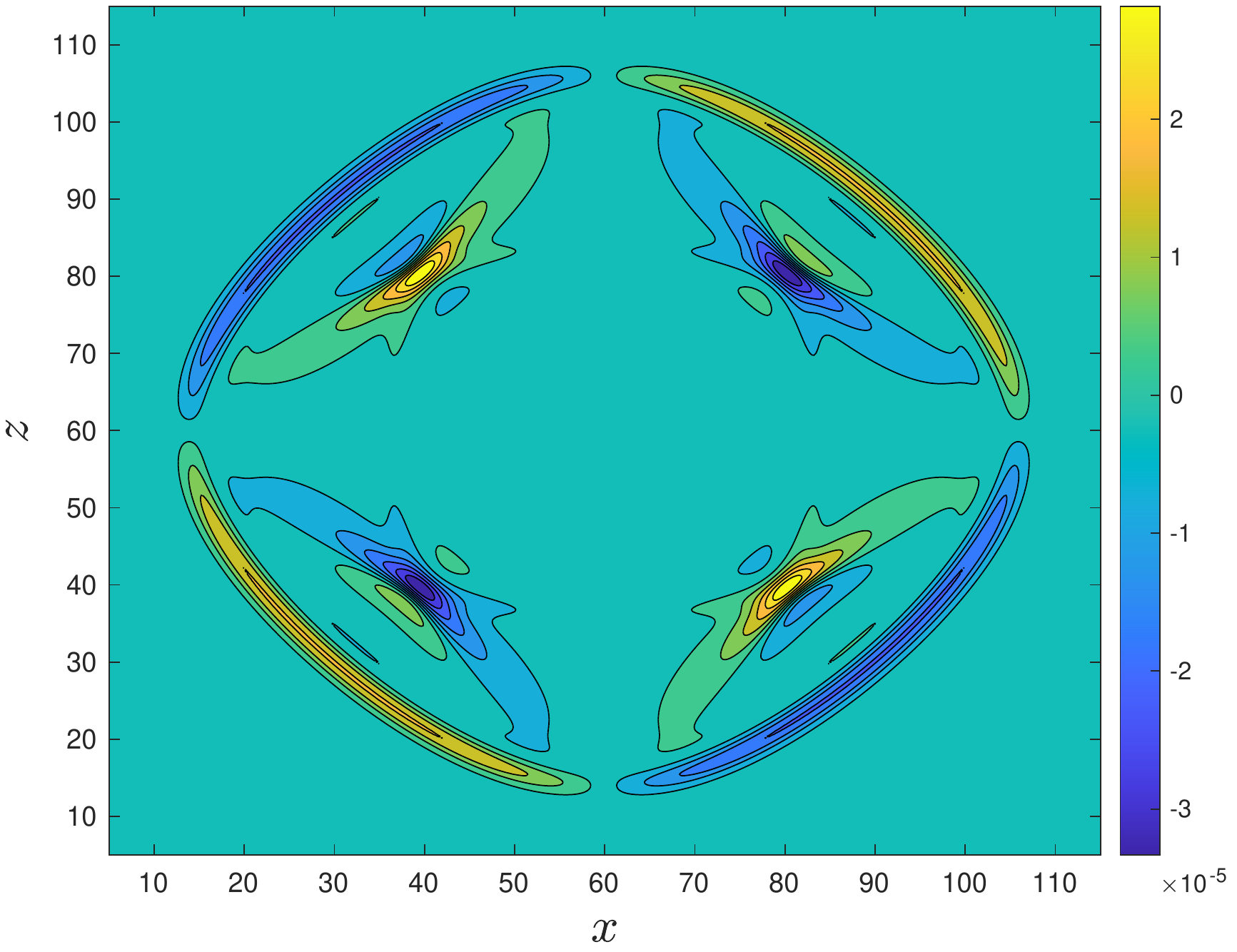}}
\caption{2-D constant-Q wave propagation: A visualization of numerical errors before and after scaling. The profile of wave fronts is similar. The main difference lies in the height, corresponding to the loss of energy.}
\label{fig_comparison}
\end{figure}
%%%%%%%%%%%%%%%%%%%%%%%%%%%%%%%%%%%%%%%%%%%%%%

%%%%%%%%%%%%%%%%%%%%%%%%%%%%%%%%%%%%%%%%%%%%%%
\begin{table}[!h]
  \centering
  \caption{\small 2-D constant-Q wave propagation: Relative errors in $u_3$ and $\sigma_{13}$ at $T= 15$ and the averaged computational time 
  ($\Delta t = 10^{-3}$, $N_x = N_z = 512$). The memory length for the reference solution is $M+1 = 500$ and the cpu time is $27994$s.\label{2d_scaling}
}
 \begin{lrbox}{\tablebox}
 \begin{tabular}{c|c|cccccc|c}
 \hline\hline
Metric	&\diagbox{$M$}{$\beta$} 	&0.6	&	1	&	2	&	8	&	16	& 	32	&	time(s)	\\
 \hline
\multirow{3}{*}{$\mathcal{E}_{\infty}[{u_3}] $}
&	8	&	1.161E+00	&	1.138E+00	&	1.035E+00	&	7.314E-01	&	6.897E-01	&	1.005E+00	&	3095.38 	\\
&	16	&	1.111E+00	&	1.032E+00	&	8.804E-01		&	4.669E-01	&	2.519E-01	&	3.652E-01		&	4023.80 	\\
&	32	&	9.969E-01		&	8.798E-01		&	6.847E-01		&	2.253E-01	&	2.237E-02	&	2.553E-01		&	5097.17 	\\
\hline
\multirow{3}{*}{$\mathcal{E}_{\infty}[{\sigma_{13}}]$}
&	8	&	1.280E+00	&	1.223E+00	&	1.135E+00	&	8.897E-01	&	1.204E+00	&	2.026E+00	&	3095.38 	\\
&	16	&	1.201E+00	&	1.128E+00	&	9.693E-01		&	5.265E-01	&	3.533E-01		&	7.727E-01		&	4023.80 	\\
&	32	&	1.091E+00	&	9.678E-01		&	7.555E-01		&	2.497E-01	&	4.222E-02		&	3.041E-01		&	5097.17 	\\
\hline\hline
 \end{tabular}
\end{lrbox}
\scalebox{0.7}{\usebox{\tablebox}}
\end{table}
%%%%%%%%%%%%%%%%%%%%%%%%%%%%%%%%%%%%%%%%%%%%%%

%%%%%%%%%%%%%%%%%%%%
\section{Conclusion and discussion}
\label{sec.conclusion}

We propose a short-memory operator splitting (SMOS) scheme for solving the constant-Q viscoelastic wave equation with the fractional order $2\gamma$ of the Caputo fractional derivative much smaller than 1. Two main features of SMOS are presented. One is to shorten the effective memory length via the extension problem of the fractional derivative. The other is to use a scaling factor $\beta > 1$ to alleviate the dynamical increments of the projection errors. Combining with the operator splitting scheme to exploit the exact solution of the auxiliary dynamics, our scheme can maintain the numerical accuracy, as well as significantly alleviate both memory requirement and arithmetic complexity. Our ongoing work is to apply SMOS in real 2-D and 3-D seismic applications. We would like to discuss such issues in our future work, such as combining SMOS with recently developed frequency-adaptive scaling technique \cite{XiaShaoChou2020}.

\section*{Acknowledgement}
This research was supported by  the Project funded by China Postdoctoral Science Foundation (Nos. 2020TQ0011), the Foundation of China under Grants (Nos.~11901354),  the Natural Science Foundation of Shandong Province for Excellent Youth Scholars (Nos.~ZR2020YQ02), the Taishan Scholars Program of Shandong Province of China (Nos.~tsqn201909044) and the High-performance Computing Platform of Peking University. The authors would like to thank Prof. Sihong Shao for fruitful discussions.

%% The Appendices part is started with the command \appendix;
%% appendix sections are then done as normal sections
\appendix

%\section{Exact solution to 1-D diffusive wave equation}
%\label{}
\section{Evaluation of the Mainardi function by the generalized Laguerre-Gauss quadrature}
\label{sec.Mainardi_function}

In order to calculate the Mainardi  function, we need to investigate the numerical algorithms for approximating the Wright function. Direct calculation of  contour integral representation is somehow difficult. Instead, one can start from an equivalent form. 
\begin{lemma}[Theorem 2.1 in \cite{Luchko2008}]
 For $-1 < \lambda < 0$ and $\mu < 1$, for $x > 0$
\begin{equation}
W_{\lambda, \mu}(- x) = \frac{1}{\pi} \int_0^{+\infty} K_{\lambda, \mu}(-x, r ) \D r,
\end{equation}
%and
%\begin{equation}
%W_{\lambda, \mu}(x) = \frac{1}{\pi} \int_0^{+\infty} K_{\lambda, \mu}(x, r ) \D r,
%\end{equation}
where the kernel reads that
\begin{equation}\label{kernel}
K_{\lambda, \mu}(x, r ) = r^{-\mu} \me^{-r} \left[\me^{x r^{-\lambda} \cos(\lambda \pi)} \sin(x r^{-\lambda} \sin(\pi \lambda) + \pi \mu)\right].
\end{equation}
\end{lemma}

When $\lambda = -1/2$ and $\cos(\lambda \pi) = 0$, the integral can be approximated by the generalized Gauss-Laguerre quadrature
\begin{equation}\label{Gaussian_estimate}
W_{\lambda, \nu}(- x)\approx \frac{1}{\pi} \sum_{j= 0}^{M_r} \omega_{j}^{(-\mu)}  \left[\me^{-x (r_j^{(-\mu)})^{-\lambda}} \sin(-x (r_j^{(-\mu)})^{-\lambda} \sin(\pi \lambda) + \pi \mu)\right].
\end{equation}
so that it  is more computational feasible compared with the contour integral.

When $-1/2 < \lambda < 0$, we find that the above formula might not get accurate results when $x$ is large because it fails to capture the proper asymptotic of the integral. In order to tackle the situation for large $x$, it is convenient to introduce a scaling,
\begin{equation}
x r^{-\lambda} = y, \quad r = x^{1/\lambda} y^{-1/\lambda} , \quad \D r = -\frac{x^{1/\lambda}}{\lambda}   y^{-1/\lambda - 1} \D y
\end{equation}
so that for $x > 0$ and $\cos(\lambda \pi) > 0$,
\begin{equation*}
W_{\lambda, \mu}(-x) = -\frac{x^{-\frac{\mu-1}{\lambda}}}{\pi \lambda} \int_0^{+\infty} y^{\frac{\mu-1}{\lambda} -1} \me^{- y \cos(\lambda \pi) }  \left[ \me^{-\left(\frac{x}{y}\right)^{1/\lambda}} \sin(- y \sin(\pi \lambda) + \pi \mu)  \right] \D y.
\end{equation*}
Thus
\begin{equation}\label{subGaussian_estimate}
W_{\lambda, \mu}( - x) \approx -\frac{x^{-(s_1+1)}}{\pi \lambda} \sum_{j=0}^{M_r} \omega_{j}^{(s_1, s_2)}\me^{-\left(\frac{x}{y_j^{(s_1, s_2)}}\right)^{1/\lambda}} \sin(- y_j^{(s_1, s_2)} \sin(\pi \lambda) + \pi \mu),
\end{equation}
where $s_1 = \frac{\mu-1}{\lambda} - 1$, $s_2 = \cos(\lambda \pi)$.

When $-1 < \lambda < -1/2$ and $\cos(\lambda \pi) < 0$, the exponent in the kernel $K_{\lambda, \mu}(-x, r )$ is no longer monotone. This causes troubles in the convergence of the Laguerre-Gauss quadrature. Instead, one needs to combine both the integral representation for small $x$ and asymptotic expansion for  large $x$.

For small $x < -\frac{1}{\cos(\lambda \pi)}$ , it starts from
 \begin{equation*}
\begin{split}
W_{\lambda, \mu}(-x) = &\frac{1}{\pi} \int_0^{+\infty} r^{-\mu} \me^{-r} \left[\me^{-x r^{-\lambda} \cos(\lambda \pi)} \sin( - x r^{-\lambda} \sin(\pi \lambda) + \pi \mu)\right] \D r \\
= &\frac{1}{\pi} \int_0^{+\infty} r^{-\mu} \me^{-r(1 + x\cos(\lambda \pi))} \left[\me^{x (r - r^{-\lambda}) \cos(\lambda \pi)} \sin( - x r^{-\lambda} \sin(\pi \lambda) + \pi \mu)\right] \D r,
\end{split}
\end{equation*}
so that
\begin{equation}\label{super_Gaussian_part1}
W_{\lambda, \nu}(-x)\approx \frac{1}{\pi} \sum_{j= 0}^{M_r} \omega_{j}^{(s_1, s_2)}  \left[\me^{x (r_j^{(s_1, s_2)} - (r_j^{(s_1, s_2)})^{-\lambda}) \cos(\lambda \pi)} \sin(-x (r_j^{(-\mu)})^{-\lambda} \sin(\pi \lambda) + \pi \mu)\right],
\end{equation}
where $s_1 = -\mu, s_2 = 1 + x\cos(\lambda \pi)$.

For the regime $x \ge -\frac{1}{\cos(\lambda \pi)}$, we introduce an appropriate threshold $x_{\lambda}$ depending on $\lambda$. For $x \le x_{\lambda}$, we use another scaling
\begin{equation}
-x r^{-\lambda}\cos(\lambda \pi) = y, \quad r = \tilde{x}^{1/\lambda} y^{-1/\lambda} , \quad \D r = -\frac{\tilde{x}^{1/\lambda}}{\lambda}   y^{-1/\lambda - 1} \D y.
\end{equation}
with $\tilde{x} = - \cos(\lambda \pi) x$. That yields that
\begin{equation*}
\begin{split}
W_{\lambda, \mu}(-x) &= -\frac{\tilde{x}^{-\frac{\mu-1}{\lambda}}}{\pi \lambda} \int_0^{+\infty} y^{\frac{\mu-1}{\lambda} -1} \me^{y }  \left[ \me^{-\left({\tilde{x}}/{y}\right)^{1/\lambda}} \sin(y \tan(\pi \lambda) + \pi \mu)  \right] \D y \\
&= -\frac{\tilde{x}^{-\frac{\mu-1}{\lambda}}}{\pi \lambda} \int_{y_0}^{+\infty} y^{\frac{\mu-1}{\lambda} -1} \me^{-(\tilde{x} -1) y  }   \left[ \me^{- \left( \left(\tilde{x}/{y}\right)^{1/\lambda}  -  \tilde{x} y\right)} \sin(y \tan(\pi \lambda) + \pi \mu)  \right] \D y.
\end{split}
\end{equation*}
 so that it can also be evaluated by the Laguerre-Gauss quadrature rule,
\begin{equation}\label{super_Gaussian_part2}
W_{\lambda, \mu}( - x) \approx -\frac{\tilde{x}^{-(s_1+1)}}{\pi \lambda} \sum_{j=0}^{M_r} \omega_{j}^{(s_1, s_2)}\me^{-\left(\frac{\tilde{x}}{y_j^{(s_1, s_2)}}\right)^{1/\lambda} +  \tilde{x} y_j^{(s_1, s_2)}} \sin(y_j^{(s_1, s_2)} \tan(\pi \lambda) + \pi \mu),
\end{equation}
where $s_1 = \frac{\mu-1}{\lambda} - 1$, $s_2 = - x \cos(\lambda x) - 1$. In particular, when $\mu = 1 - \nu$, $\lambda = -\nu$, the singular kernel $y^{\frac{\mu-1}{\lambda} - 1}$ vanishes and $s_1 = 0$.

%In order to evaluate the above integral, we need to decompose the integral into two parts because for $-1 < \lambda < 0$ and  sufficiently large $y$, $\me^{\tilde{x}^{1/\lambda} y^{-1/\lambda}}$ grows faster than $\me^y$,
%\begin{equation*}
%\begin{split}
%W_{\lambda, \mu}(-x) = &-\frac{\tilde{x}^{-\frac{\mu-1}{\lambda}}}{\pi \lambda} \int_0^{y_0} y^{\frac{\mu-1}{\lambda} -1} \left[ \me^{- \left( \left(\tilde{x}/{y}\right)^{1/\lambda}  -  y\right)} \sin(y \tan(\pi \lambda) + \pi \mu)  \right] \D y \\
%& -\frac{\tilde{x}^{-\frac{\mu-1}{\lambda}}}{\pi \lambda} \int_{y_0}^{+\infty} y^{\frac{\mu-1}{\lambda} -1} \me^{-(\tilde{x} -1) y  }   \left[ \me^{- \left( \left(\tilde{x}/{y}\right)^{1/\lambda}  -  \tilde{x} y\right)} \sin(y \tan(\pi \lambda) + \pi \mu)  \right] \D y,
%\end{split} 
%\end{equation*}
%where $ y_0 =  \tilde{x}^{\frac{1- \lambda}{1+\lambda}}$, such that when $y > y_0$, $(\tilde{x}/y)^{1/\lambda} > \tilde{x} y$.

%Therefore, the first integral can be evaluated by the standard quadrature rule and the second integral over the semi-infinite interval can be evaluated by the Laguerre-Gauss quadrature rule.

Finally, to avoid the numerical instability for large $x \ge  x_{\lambda}$, we can utilize the asymptotic expansion of the Mainardi function given by the saddle-point approximation \cite{bk:Mainardi2010}
\begin{equation}\label{super_Gaussian_asymptotic}
M_{\nu}\left(\frac{z}{\nu}\right) \sim \frac{1}{\sqrt{2\pi(1-\nu)}} z^{\frac{\nu -	1/2}{1- \nu}} \exp\left( - \frac{1-\nu}{\nu} z^{\frac{1}{1-\nu}}\right)
\end{equation}
Despite its simple form, the leading term in the asymptotic expansion provides a very accurate approximation for sufficiently large $x$ as the Mainardi function decays very rapidly when $|x| \to \infty$.

Figure \ref{Mainardi_function} plots the Mainardi functions $M_{\nu}(x) = W_{\nu, 1-\nu}(x)$ for different $\nu$. When $\nu = -1/2$, it is just the Gaussian kernel. When $\nu > -1/2$, it corresponds to the Green's function for the fractional diffusion equations. The more interesting regime is
 $-1< \nu < -1/2$, where the Mainardi function becomes a double-well function and exhibits the wave dispersion property.
%%%%%%%%%%%%%%%%%%%%%%%%%%%%%%%%%%%%%%%%%%%%%%
\begin{figure}[!h]
\centering
\subfigure[Mainardi function for fractional diffusion, $0 < \nu < 1/2$.]{
{\includegraphics[width=0.49\textwidth,height=0.29\textwidth]{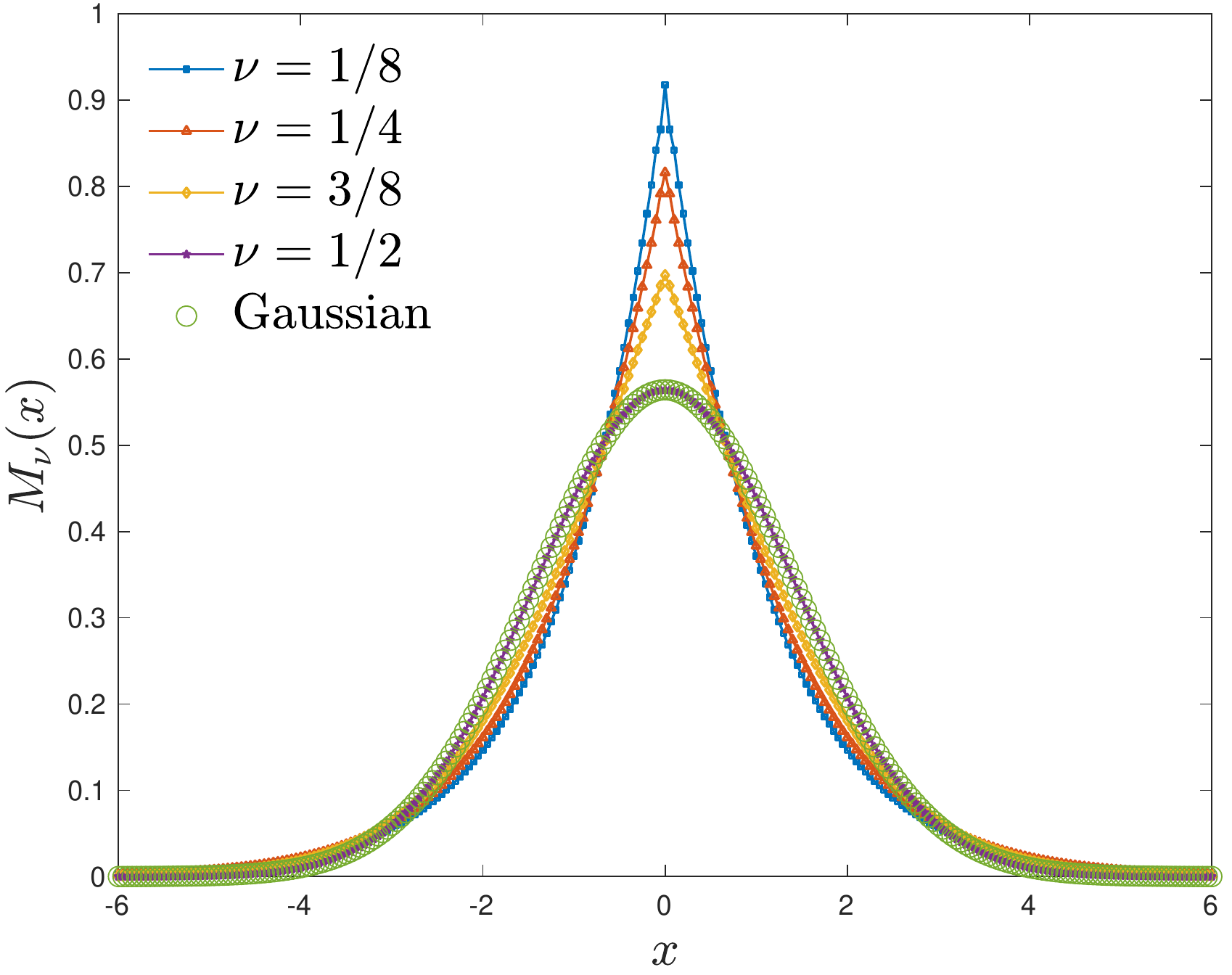}}
{\includegraphics[width=0.49\textwidth,height=0.29\textwidth]{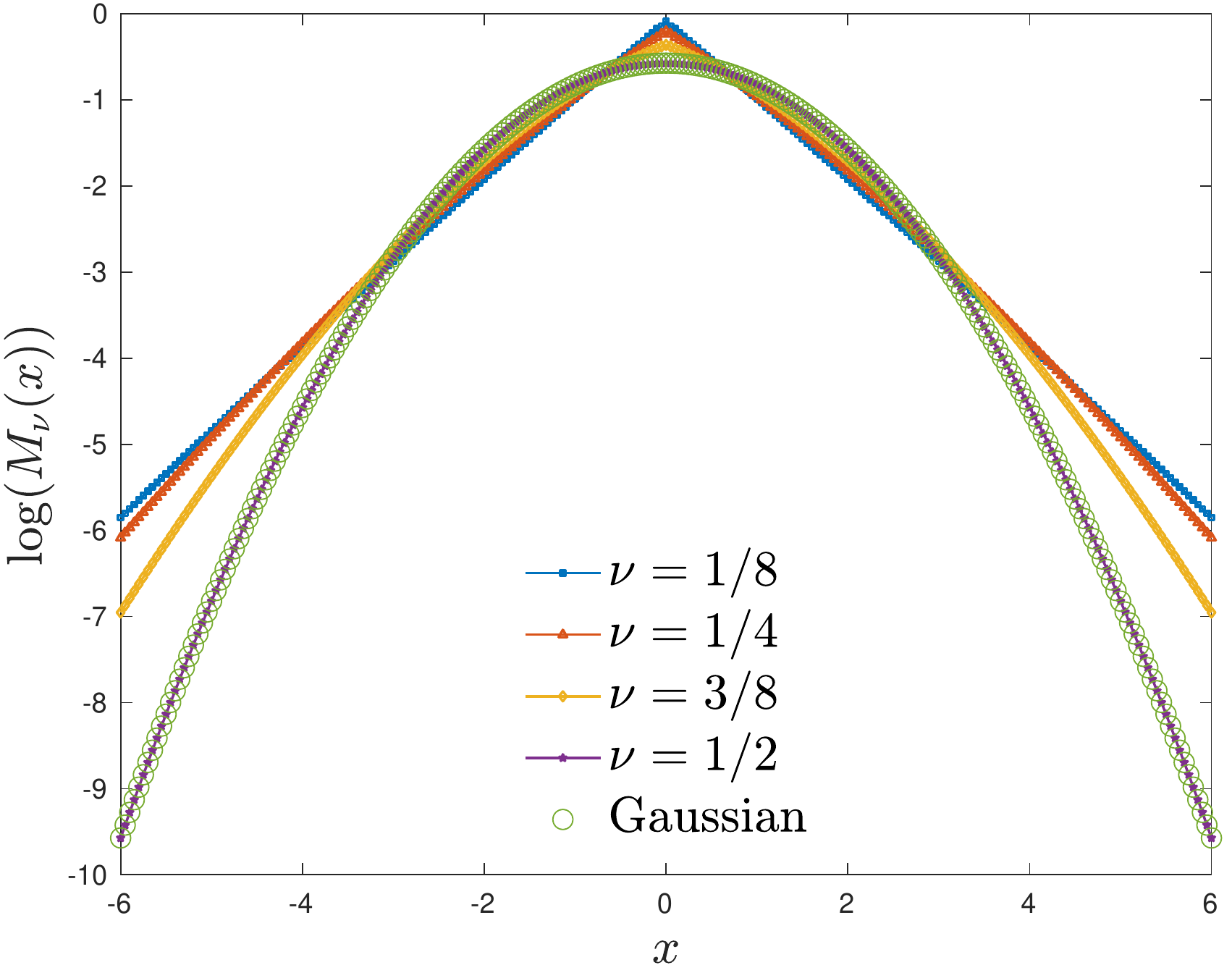}}}
\\
\centering
\subfigure[Mainardi function for fractional wave propagation, $1/2 < \nu < 1$.]{
{\includegraphics[width=0.49\textwidth,height=0.29\textwidth]{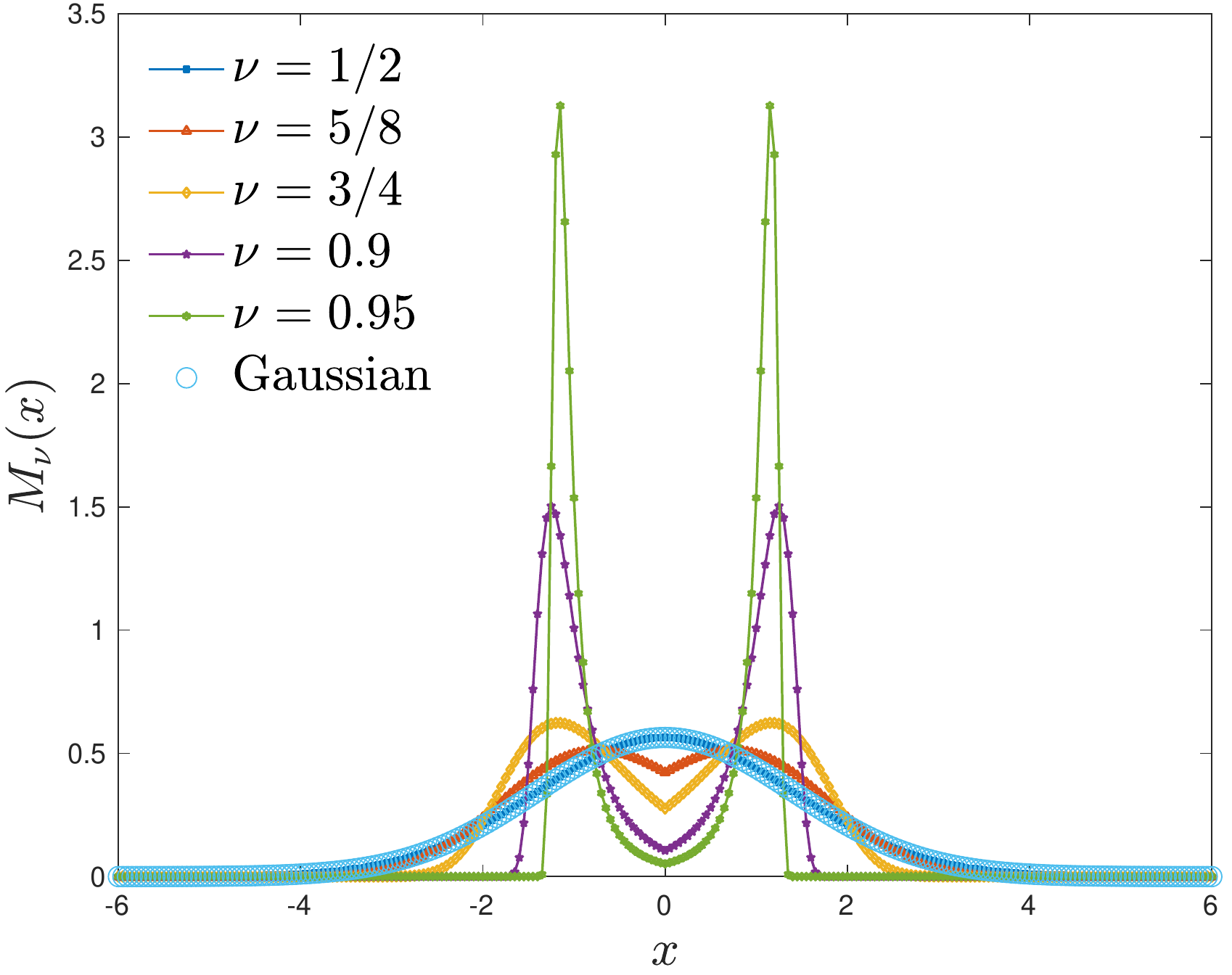}}
{\includegraphics[width=0.49\textwidth,height=0.29\textwidth]{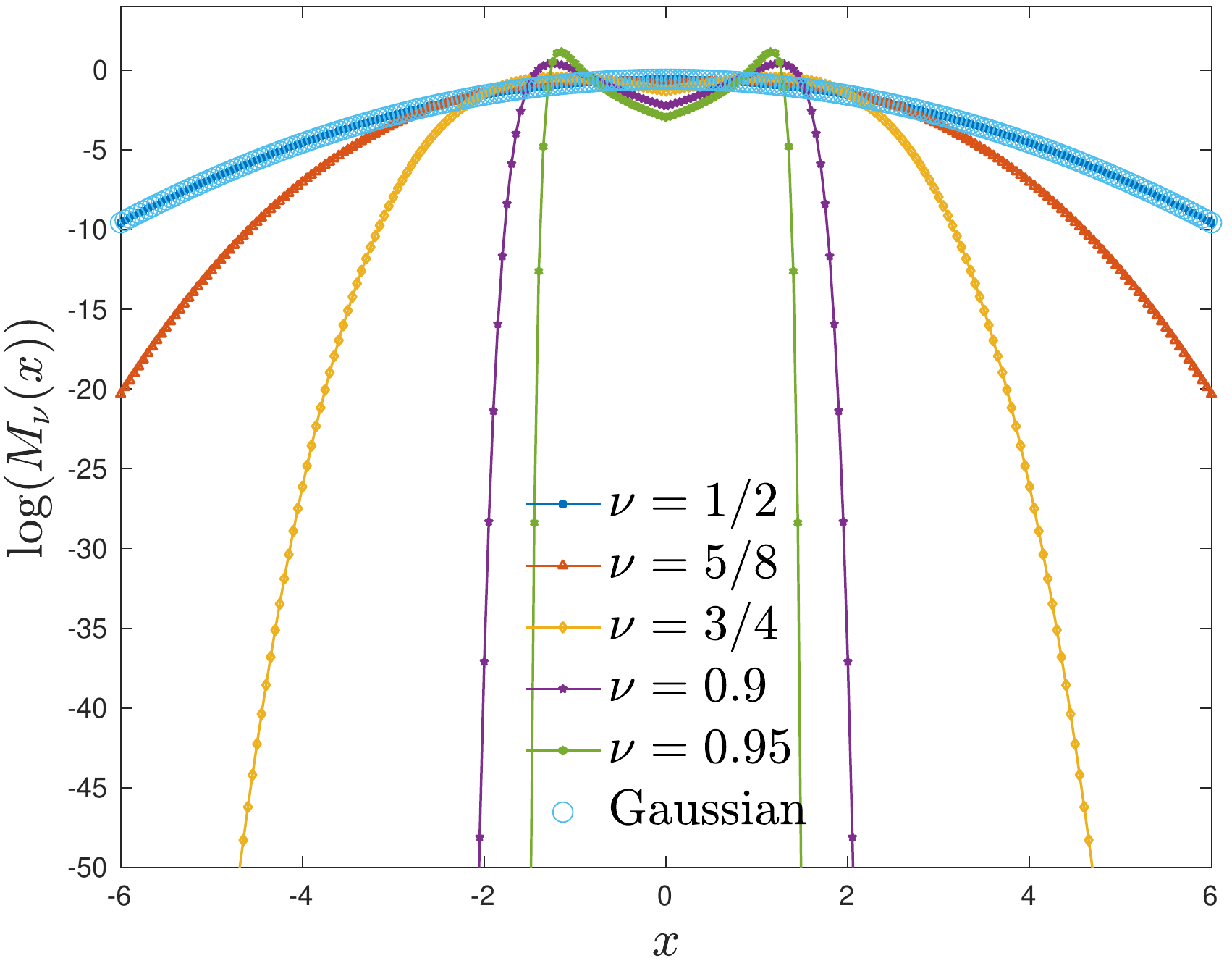}}}
\caption{The Mainardi function (left) and that in logarithm scale (right) for fractional diffusion equations and fractional wave equations. }
\label{Mainardi_function}
\end{figure}
%%%%%%%%%%%%%%%%%%%%%%%%%%%%%%%%%%%%%%%%%%%%%%

The choice of threshold $x_{\lambda}$ can be made by searching a point that matches the Laguerre-Gauss quadrature and the asymptotic leading term. For instance, we have found that $x_{\lambda,\mu} = 1.852$ for $\nu = 0.9$, $x_{\lambda,\mu} = 1.41$ for $\nu = 0.95$ and $x_{\lambda,\mu} = 1.087$ for $\nu = 0.99$ can achieve a difference less than $10^{-7}$ between the results produced by the Laguerre-Gauss quadrature and the asymptotic leading term.

To summarize, we use the following program to estimate the Mainardi function. 

\begin{itemize}

\item[(1)] When $\lambda = -\frac{1}{2}$ , we use the formula \eqref{Gaussian_estimate}.

\item[(2)] When $-\frac{1}{2} < \lambda < 0$, we use the formula \eqref{subGaussian_estimate}.

\item[(3)] When $-1 < \lambda < -\frac{1}{2}$ and $x < -\frac{1}{\cos(\lambda \pi)}$, we use the formula \eqref{super_Gaussian_part1}.

\item[(4)] When  $-1 < \lambda < -\frac{1}{2}$ and $ -\frac{1}{\cos(\lambda \pi)} \le x < x_{\lambda}$, we use the formula \eqref{super_Gaussian_part2}.

\item[(5)] When $-1 < \lambda < -\frac{1}{2}$ and $x \ge x_{\lambda}$, we use the asymptotic formula \eqref{super_Gaussian_asymptotic}.

\end{itemize}

\bibliographystyle{elsarticle-num} 

%% else use the following coding to input the bibitems directly in the
%% TeX file.

\end{thebibliography}
\end{document}